\documentclass[a4paper,12pt,final]{article}

\usepackage[margin=21.5mm]{geometry}

\usepackage[utf8]{inputenc}
\usepackage{amsmath,amsfonts,amssymb,bm,bbm,mathrsfs,esint,relsize}
\numberwithin{equation}{section}
\numberwithin{figure}{section}
\usepackage{bef_alex,scrtime,tikz,xcolor}
\renewcommand{\ProofFont}{\itshape}

\renewcommand{\FG}{\bm}

\renewcommand{\ti}{{\times}}
\renewcommand{\varkappa}{p}
\renewcommand{\dd}{\:\rmd}

\newcommand{\bddin}{\in_{\rmb\rmd}}

\newcommand{\STEP}[1]{\noindent\underline{\itshape Step #1}}

\newcommand{\AAA}{}
\newcommand{\BBB}{}
\newcommand{\RRR}{}
\newcommand{\SSS}{}
\newcommand{\EEE}{}

\usepackage[colorlinks=true]{hyperref}
\hypersetup{urlcolor=blue, citecolor=brown!50!black, linkcolor=blue!50!black}
\newcommand{\TOS}{\texorpdfstring}

\usepackage[notref,notcite]{showkeys}

\makeatletter
\renewcommand*\env@cases[1][1.2]{%
  \let\@ifnextchar\new@ifnextchar
  \left\lbrace
  \def\arraystretch{#1}%
  \array{@{\,}c@{\quad}l@{}}%
}
\makeatother

\newcommand{\PM}{{(\pm)}}
\newcommand{\eff}{\mathrm{eff}}
\newcommand{\mb}{\mathrm{mb}}

\newcommand{\LEB}{\mathscr L}
\newcommand{\CCC}{\mathscr{C}}
\newcommand{\DDD}{\mathscr D}

\newcommand{\mfD}{\mathfrak D}
\newcommand{\mfI}{\mathfrak I}
\newcommand{\mfP}{\mathfrak P_\mathrm{fo}}
\newcommand{\mfM}{\mathfrak M}
\newcommand{\mfS}{\mathfrak S}
\newcommand{\mfT}{\mathfrak T}

\newcommand{\sfg}{\mathsf{g}} 
\newcommand{\sfn}{\mathsf{n}} 
\newcommand{\sfE}{\mathsf{E}} 
\newcommand{\sfF}{\mathsf{F}} 
\newcommand{\sfT}{\mathsf{T}} 
 
\newcommand{\sfC}{\mathsf{C}}
\newcommand{\sfD}{\mathsf{D}}
\newcommand{\sfCE}{\mathsf{C}_\mathrm{en}}
\newcommand{\sfCF}{\mathsf{C}_\mathrm{flux}}
\newcommand{\sfCS}{\mathsf{C}_\mathrm{slope}}
\newcommand{\mob}{\mathsf{m}}

\newcommand{\qsubE}{{q_{_E}}}
\newcommand{\qsubH}{{q_{_H}}}

\newcommand{\qand}{\quad \text{and}\quad}
\newcommand{\weak}{\rightharpoonup}
\newcommand{\weaks}{\overset{*}{\rightharpoonup}}
\newcommand{\scrB}{\mathscr B}
\newcommand{\scrE}{\mathscr E}
\newcommand{\scrG}{\mathscr G}
\newcommand{\scrQ}{\mathscr Q}
\newcommand{\scrR}{\mathscr R}

\begin{document}

\title{From diffusion to transmission via EDP-convergence: \\ 
a paradigmatic multiscale limit%
\thanks{The research was partially supported by DFG via SFB 1114,
  subproject C05}} 

\author{Thomas Frenzel\thanks{WIAS Berlin, Anton-Wilhelm-Amo-Stra\ss{}e 39, 10117
    Berlin, Germany}\ \ and Alexander
  Mielke$^\dagger$}

\date{14 January 2026. \AAA Revised 12 July 2026 \EEE} 

\maketitle

\centerline{\AAA \itshape Dedicated to Felix Otto on the occasion of his
  sixtieth birthday} 


 {\footnotesize\tableofcontents}

 
\newpage

\begin{samepage}
\begin{abstract}
We consider nonlinear diffusion equations on an interval where the diffusion
coefficient in a small region near the center is scaled such that it approximates
a transmission condition at a membrane. While the limiting behavior of the
solutions is well-understood, we study the convergence in the sense of the
energy-dissipation principle (EDP) of associated gradient structures given in
terms of a free energy and a dissipation potential.

EDP-convergence provides a uniquely specified limiting gradient structure that
reformulates the transmission condition in terms of an effective kinetic
relation for the membrane, which relates the jump of the chemical potential and
the flux through the membrane. We show how \AAA properties of the chosen free
energy and the mobility of small-scale diffusion migrate to the effective
kinetic relation. \EEE A surprising \BBB result \EEE is that starting from the
linear Onsager relation of Otto's gradient structure for the linear diffusion
equation, one obtains an exponentially growing kinetic relation, the so-called
Marcelin-De\;Donder kinetic.
\end{abstract}
\end{samepage}

\section{Introduction}
\label{se:Intro}

The theory of gradient-flow structures for linear and nonlinear
diffusion equations was initiated in \cite{JoKiOt98VFFP, Otto01GDEE}.
In particular, for the porous-medium equation
\begin{equation}
  \label{eq:I.PME}
  \pl_t v = \DIV\big( A(x) \nabla v^\alpha\big) \
\text{ in } \DDD, \qquad A(x)\nabla v^\alpha  \cdot \nu = 0 \
\text{ on }\pl\DDD ,
\end{equation}
with exponent $ \alpha > 0 $ can be written formally as a gradient-flow
equation with a free energy $\calE$ and a dual dissipation potential $\calR^*$ in
the form
\begin{equation}
  \label{eq:I.calE.calR}
   \calE(v) = \int_\DDD E(v(x)) \dd x \qand 
\calR^*(v,\xi) = \frac12\int_\DDD A(x)
\mob(v) |\nabla\xi|^2 \dd x . 
\end{equation}
As $\calR^*$ is quadratic in $\xi$ we have
$ \calR^*(v,\xi)= \frac12\langle \xi,\bbK(v) \xi\rangle $ with a linear Onsager
operator $ \bbK ( v ) $ of Otto-type that acts as the inverse of a metric
tensor:
\[
\bbK(v)\xi : = - \DIV \big( A(x) \mob(v) \nabla \xi
\big) \ \text{ in }\DDD, 
\quad  A(x) \mob(v) \nabla \xi\cdot \nu= 0 \ \text{ on } \pl\DDD, 
\]
where $\mob$ denotes the mobility law, while $A$ includes
heterogeneities in the domain $ \DDD $.
With this, the porous-medium equation \eqref{eq:I.PME} can be written
as the gradient-flow equation
\begin{equation}
  \label{eq:I.GenDiffEqn}
    \pl_t v = - \bbK(v) \rmD \calE(v)
 = \DIV\Big( A \,\mob(v) \nabla E'(v)\Big) 
 =  \DIV\Big( A \,\mob(v) E''(v) \nabla v\Big) ,
\end{equation}
if the relation $ \alpha v ^{\alpha-1}= \mob(v) E''(v)$ holds. In
particular, we see that the choices $\mob(v)=v^\beta$ and $E= \sfE_\varkappa$ with
$\alpha= \beta{+}\varkappa{-} 1$ are possible, where $\sfE_\varkappa$ satisfies
$\sfE''_\varkappa(v)=v^{\varkappa-2}$. However,
there are even more general families of gradient structures for
\eqref{eq:I.PME}.

Thus, we see that the gradient structures induced by $\calE$ and $ \bbK$
provide additional information to the partial differential equation
\eqref{eq:I.PME} that is not contained in the equation itself. This avenue was
opened up by the pioneering work of Otto (cf.\ \cite[p.\,108]{Otto01GDEE}):
\begin{quote}
\itshape The merit of the right gradient flow formulation of a dissipative
evolution equation is that it separates energetics and kinetics: The energetics
endow the state space $\calM$ with a functional $E$, the kinetics endow the state
space with a (Riemannian) geometry via the metric tensor $g$.
\end{quote}
\AAA We call the triple $(\calM,E,g)$ a \emph{gradient structure} for the
equation and will use the notation $\big(\bfQ,\calE,\bbK\big)$, where \EEE  
our operator $\bbK$ is the inverse of the metric tensor $g$, i.e.\ 
$g(v)=\big(\bbK(v) \big)^{-1}$. 

\AAA Hence, for a given $\alpha>0$ we have infinitely many choices for the
gradient structure, e.g.\ by choosing $m(v)=v^\beta$ and $E=E_p$ with
$\alpha =\beta+p-1$, see Figure \ref{fig:Intro}. All of these choices are
mathematically valid gradient structures, each leading to a different
functional analytic setup. For instance, the case $\beta=0$ leads to classical
gradient flows in the Hilbert space $\rmH^{-1}(\Omega)$ that were introduced in
\cite{Brez71MMHS}.  The analysis of the case $\beta=1$ was initiated by Felix
Otto in \cite{Otto96DDDE, JoKiOt98VFFP, Otto01GDEE} and let to the beautiful
connections to the Wasserstein distance and the theory of optimal transport.
In the recent papers \cite{FehGes23NELD, GesHey25PMEL} the gradient structure
for the case $p=1$ and $\beta>1$ was derived via Large-Deviation theory.  \EEE
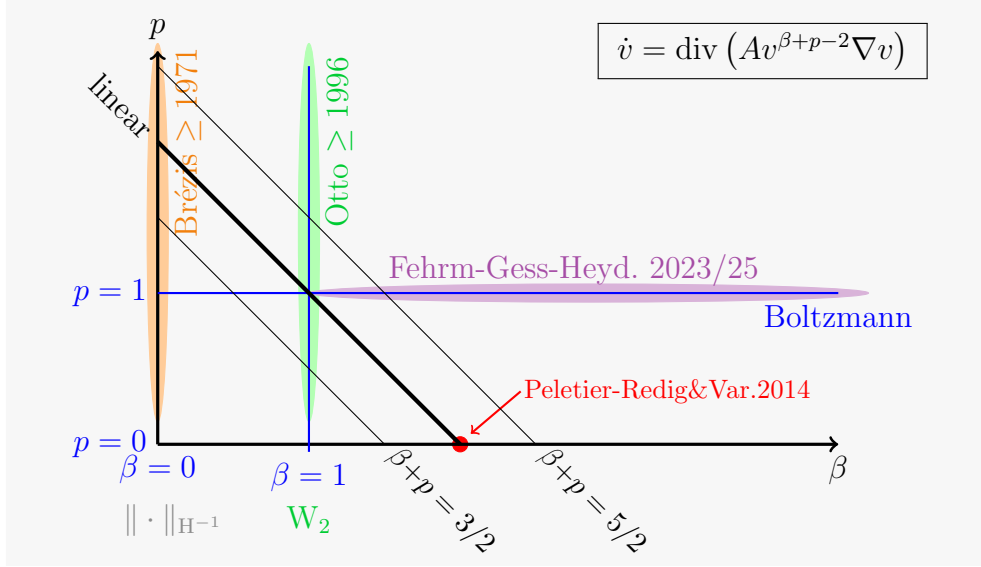
\begin{figure}
\centerline{\begin{tikzpicture}
\draw[gray!6, fill] (-2,-1.6) rectangle (11,5.9);
  \draw[fill,color=green!30] (2,2.8) ellipse (.14 and 2.5);
    \node[rotate=90, green!80!blue] at (2.4,4) {Otto $\geq 1996$};
  \draw[fill,color=violet!30] (5.7,2) ellipse (3.7 and 0.12);
    \node[violet!70] at (5.5,2.3) {Fehrm-Gess-Heyd.\ 2023/25};
  \draw[fill,color=orange!40] (0,2.8) ellipse (.14 and 2.5);
    \node[rotate=90, orange!80!brown] at (0.4,4) {Br\'ezis $\geq 1971$};

  \draw[red,fill] (4,0) circle (0.1);
  \draw[thick,red,->] (4.8,0.7) --(4.14,0.14);
  \node[right, red] at (4.7,0.7){{\smaller Peletier-Redig\&Var.2014}};


\node at (8,5.2) {\fbox{ $\dot v= \DIV\big(A v^{\beta+p-2}\nabla v\big) $ }};
\draw[very thick, ->] (0,0) -- (9,0) node[below] {$\beta$};
\draw[very thick, ->] (0,0) -- (0,5.2) node[above] {$p$};

\draw[thick, blue] (0,2)--(9,2) node[below]{Boltzmann};
\node[left, blue]  at (0,2) {$p=1$ }; 
\node[left, blue]  at (0,0) {$p=0$ }; 
\draw[thick, blue] (2,5)--(2,-0.1) node[below]{$\beta=1$};
\node[below, blue] at (0,0) {$\beta=0$};

\draw[ultra thick] (4,0)--(0,4) node[left,rotate=-45] {linear};
  \draw[thin] (0,3)--(3,0)  node[right,rotate=-45] {{\small$\beta{+}p=3/2$}};
  \draw[thin] (0,5)--(5,0)  node[right,rotate=-45] {{\small$\beta{+}p=5/2$}};

\node[gray!90] at (0.2,-1) {$\|\cdot\|_{\rmH^{-1}}$};
              \node[green!80!blue] at (2,-1) {$\rmW_2$};
\end{tikzpicture}}
\caption{In the $(\beta,p)$ plane for the mobility $\mob(v)=v^\beta$ and the
  free energy $E=\sfE_p$ we indicate several regions where the analysis of the
  porous medium equation was developed. For details see the main text.}
\label{fig:Intro} 
\end{figure}
\AAA In \cite{PeReVa14LDSH} the case $p=0$ and $\beta=2$ was shown to serve as
a model for heat transfer. All the different choices are mathematically valid,
but they correspond to different microscopic models or, what is the same, to
different physical applications.  The case $\beta=1$ seem to be the right
choice for diffusion processes, while modeling heat transfer needs an entropy
$S=S(u)=-E(u)$ that depends concavely and monotonously on the internal energy
$u$, such that the temperature $\theta=1/S'(u) =\rmd u/\rmd s$ is
non-negative. This can only be realized by $S(u)=u^p$ with $0< p<1$ or
$S(u)=\log u$, i.e.\ $p=0$. Thus, the word ``right'' in the above citation from
\cite[p.\,108]{Otto01GDEE} corresponds to the right one for a given physical
application.\EEE

The present work, which has its origin in \cite[Cha.\,4]{Fren19DEGS}, 
reconsiders the model from \cite[Sec.\,4]{LMPR17MOGG}
of a heterogeneous diffusion equation on the interval $\DDD={]{-}1,1[}$, 
where the mobility $A=A_\eps$ is of order $\eps$ on a small
region of length $\eps$ around $x=0$. This means we set   
\[
  A_\eps(x)= \left\{\ba{cl} A_-&\text{for }x \in {]{-}1,0[},\\
        b\eps & \text{for }x \in {]0,\eps[}, \\
         A_+& \text{for } x \in {]\eps,1[}. \ea \right.
\]
and obtain the one-dimensional parabolic equation
\[
\pl_t v = \pl_x\big( A_\eps(x) \pl_x v^\alpha\big)  \ \text{
  in } {]{-}1,1[}, \qquad \pl_xv^\alpha\Big|_{x=\pm 1} = 0. 
\]
Following the ideas in \cite{NeuJag07ETCR, GaNeKn17DETC,
  CGLP19DAEI,CiDaPo24EICP} and the references there, it is not difficult to
pass to the limit $\eps\to 0^+$ in this 
equation to receive a membrane model:
\begin{equation}
  \label{eq:I.PME.transm}
  \begin{aligned}
  &\pl_t v = \pl_x\big( A_- \pl_x v^\alpha\big)
    \ \text{ in } {]{-}1,0[},
  \\
  & A_-\pl_x v^\alpha\big|_{x=0^-} = A_+\pl_x
    v^\alpha\big|_{x=0^+} = \tdfrac b\alpha\,\big(  v^\alpha|_{x=0^+} -  v^\alpha|_{x=0^-}\big),
  \\
  &\pl_t v = \pl_x\big( A_+ \pl_x v^\alpha\big) \ \text{ in } {]0,1[},
\end{aligned}
\end{equation}
The middle equations are the \emph{transmission conditions} saying that the
mass flux $A_\pm \pl_x v^\alpha(0^\pm)$ through the membrane is given by
$\frac b\alpha \big(v(0^+)^\alpha{-}v(0^-)^\alpha\big)$. Note that the limiting
solutions are no longer continuous across the membrane. More generally, the
transmission condition reads 
\[
 A_\pm\pl_x v^\alpha\big|_{x=0^\pm}= b \big( G(v_+)- G(v_-)\big) \quad
 \text{with } G(u)=\int_0^u \mob(u)E''(u) \dd u,
\]
which \SSS only depends on \EEE the term $A\, \mob(v)E''(v)$ that is available
in the diffusion equation \eqref{eq:I.GenDiffEqn}; see Section \ref{su:Rescale}
for a simple derivation.

This paper is going beyond the mere convergence of the solutions. Instead,
we choose a gradient \AAA structures \EEE $(\calE, \calR^*_\eps)$ in the form
$\calR^*_\eps(v,\xi)= \frac12 \langle\xi, \bbK(v) \xi\rangle$ as above, but
$A$ replaced by $A_\eps$, and the question arises whether the family of
gradient structure $(\calE, \calR^*_\eps)_\eps$ has a limit and how it can be
obtained.  As explained above, there are many gradient structures for the
limiting system, see e.g.\ \cite{Miel13TMER, GliMie13GSSC} for a choice with a
linear operator $\bbK_\text{bulk-interf}$. Here, we derive the gradient
structure from the above Otto-type gradient structure defined via $\calE$ and
$\bbK_\eps$ in a systematic way by exploiting the theory of \emph{evolutionary
  $\Gamma$-convergence for gradient systems in the sense of the
  energy-dissipation principle}, shortly called EDP-convergence. For the
porous-medium equation with exponent $\alpha$, we will obtain a gradient
structure, where the dissipative process of transmission really depends truly
on the two parameters $\beta$ and $\varkappa$, while the transmission condition
for the PDE \eqref{eq:I.PME.transm} depends on the exponent $\alpha=
\beta+\varkappa-1$ only. 

The notion of EDP convergence was initiated in \cite{LMPR17MOGG} and formally
introduced in \cite{DoFrMi19GSWE, MiMoPe21EFED}, for further applications in
the area of reaction and diffusion systems, see
\cite{FreLie21EDTS, MiPeSt21EDPC, Step21CGED, PelSch23CGST, HePiSc24?GFMG,
  HeMiSt25?DCLN, MiScSt25?DFOD} and the references therein.  
In this theory we consider a family
of gradient system $(\bfQ,\calE_\eps,\calR_\eps)$, where the state space $\bfQ$
is a closed subset of a Banach space $\bfX$, the energy functionals
$\calE_\eps:\bfQ \to \R\cup\{\infty\}$ are sufficiently smooth, and
$\calR_\eps:\bfQ \ti \bfX\to [0,\infty]$ denotes the dissipation
potentials, i.e.\ for all $q\in \bfQ$, $\calR_\eps(q,\cdot):\bfX\to [0,\infty]$
is lower semi-continuous, convex, and satisfies $\calR_\eps(q,0)=0$. The dual
dissipation $\calR^*_\eps$ is given by Legendre-Fenchel conjugation
$  \calR^*_\eps( q,\xi) = \sup\bigset{ \langle \xi,v\rangle -
    \calR_\eps(q,v) }{ v\in \bfX}$.

The associated gradient-flow equation associated with
$(\bfQ,\calE_\eps,\calR_\eps)$ is given by
\begin{equation}
  \label{eq:I.GradFlow}
    0 \in \pl_{\dot q} \calR_\eps(q,\dot q) + \rmD \calE_\eps(q)  \quad
  \text{or equivalently by} \quad
  \dot q \in \pl_{\xi}\calR^*_\eps(q,{-}\rmD\calE_\eps(q)).
\end{equation}
The energy-dissipation principle states that, under suitable technical
assumptions, the gradient-flow equation \eqref{eq:I.GradFlow} is
equivalent to the energy-dissipation balance
\begin{equation}
  \label{eq:I.EnerDissBal}
  \calE_\eps(q(T)) + \mfD_\eps(q(\cdot)) = \calE_\eps(q(0)) \ \ 
  \text{with }\mfD_\eps(q(\cdot))= \int_0^T \!\!\big\{
  \calR_\eps(q,\dot q) + \calR^*_\eps(q,{-}\rmD\calE_\eps(q)) \big\}
  \dd t.  
\end{equation}
The simplest form of EDP-convergence of the family
$(\bfQ,\calE_\eps,\calR_\eps)$ to the limiting gradient system
$(\bfQ,\calE_0,\calR_\eff)$ is defined via the $\Gamma$-convergence
$\calE_\eps \overset{\Gamma}{\weak} \calE_0$ on $\bfQ$ and
$\mfD_\eps \overset{\Gamma_\rmE}{\weak} \mfD_0$ on $\rmL^1([0,T];\bfQ)$ and
$\mfD_0$ is defined via $\calE_0$ and $\calR_\eff$ as in
\eqref{eq:I.EnerDissBal}, see \cite{DoFrMi19GSWE, MiMoPe21EFED} for precise
definitions.  We emphasize that (i) $\mfD_\eps$ has \RRR to
$\Gamma$-converge on \EEE
the \emph{space of curves in} $\bfQ$ with bounded energy $\calE_\eps$ whereas
$\calE_\eps$ $\Gamma$-converges in the usual sense on the state space $\bfQ$
and that (ii) the effective dissipation potential is uniquely determined by
this procedure, but it may differ from a possibly existing limit $\calR_0$. In
particular, in our case we will see that all $\calR_\eps^*(q,\xi)$ are
quadratic in $\xi$, while the effective dissipation potential
$\calR_\eff^*(q,\cdot)$ contains a non-quadratic term for describing the
transmission through the membrane.

Moreover, we emphasize that the definition of EDP-convergence does not contain
any statement concerning the convergence of the solutions $q_\eps$ of the
gradient-flow. Nevertheless, under \AAA the same \EEE technical assumption as
for the Energy-Dissipation Principle one can show that solutions
$q_\eps:[0,T]\to \bfQ$ of the gradient-flow equation \eqref{eq:I.GradFlow}
converge (after extracting a suitable subsequence) to a solution of the
gradient-flow equation for the effective gradient system, see e.g.\
\cite[Lem.\,2.8]{MiMoPe21EFED}. \RRR For our specific model we provide the
corresponding convergence result in Corollary \ref{co:CvgSols}. \EEE This is in
close analogy of the convergence of minimizers in the case of classical
$\Gamma$-convergence for static minimization problems. However, \RRR we
emphasize that \EEE the purpose of EDP-convergence is the identification of
$\calR_\eff$, which is in fact a more difficult task than showing convergence
of solutions.

\medskip
Our main result is indeed the $\Gamma$-convergence $\mfD_\eps
\overset{\Gamma_\rmE}{\weak} \mfD_0$ (conditioned to bounded energies),
which is contained in Theorem \ref{th:MainResUnscaled} for a suitable class of
mobilities $\mob$ and energy densities $E$. It is obtained by
deriving a corresponding result for a suitably rescaled problem, where
the membrane layer is described by the stretched internal variable $ y =
x/\eps \in [0,1]$. This scaling highlights that there is a separation of time scales
because inside the small layer the density is adjusting to an optimal profile
on the time scale of order $\eps$. The relevant profiles $y \mapsto u(y) =
v(\eps y) $ are in fact
\emph{non-equilibrium steady states} (NESS) that are determined by the chemical
potentials on the two sides of the membrane layer, see \cite{Miel23NESS} for a
general theory of such fast-slow gradient systems. 

As a result we find the effective dual dissipation potential $\calR_\eff$
including the membrane term:
\begin{align}
\label{eq:I.R*eff}
  \calR^*_\eff(v,\xi)
  &= \int_{-1}^0 \!\frac{A_-}2 \mob(v)|\pl_x \xi|^2 \dd x
    +  \int^{1}_0 \!\frac{A_+}2 \mob(v) |\pl_x \xi|^2 \dd x +
    b\,\bfR^*_\mb \big(v(0^-), v(0^+), [\xi ]_0\big),
\end{align} 
where $[\xi]_0:=\xi(0^+){-}\xi(0^-)$ is the jump of the chemical potential
across the membrane. Denoting the flux through the membrane by $\kappa $, the
effective membrane potential $\bfR_\mb $ is characterized by the
following cell problem:
\begin{equation}
\label{eq:I.def.Rmemb}
\begin{aligned}  
  &\bfR_\mb \big(v_-,v_+, \kappa ):= \ol M (v_-,v_+, \kappa ) -
    \ol M (v_-,v_+, 0) \\
  & \text{with }\ol M (v_-,v_+, \kappa)
   =\inf \int_0^1\Big\{\frac{\kappa^2}{2 \mob(u(y))} +
    \frac{\mob(u(y))}2 \big| \pl_y \big(E'(u(y))\big)\big|^2 \Big\} \dd y.
\end{aligned}
\end{equation}
where the infimum is taken over all positive $u \in \rmH^1([0,1])$ with
boundary conditions $u(0)=v_-$ and $u(1)=v_+$. \AAA In
\cite[Eqn.\,(14)]{BodDer04CFNE} a closely related functional was obtained when
studying NESS via large deviation theory, see Remark \ref{rm:BodineauDerrida}. \EEE

Clearly, \AAA properties of the
mobility function $\mob$ as well as from the energy density $E$ migrate into
the membrane potential $\bfR_\mb $. \EEE To substantiate this claim, we observe
\begin{equation}
  \label{eq:I.def.H}
  \ol M(v_-,v_+,0)= \SSS \frac12 \EEE \big( H(v_+)- H(v_-)\big)^2 \quad \text{with } H(u)=\int_0^u
\sqrt{\mob(r)}\, E''(r) \dd r , 
\end{equation}
which depends on the choice of $E$ and $\mob$. 

While the occurrence of the function $\ol M $ in the $\Gamma$-limit $\mfD_0$ is
\RRR more or less straightforward, one of our key results (see Proposition
\ref{pr:Rmb.olM}) states \EEE that $\ol M $ satisfies the relation
\begin{equation}
  \label{eq:olM.Rmb.R*mb}
  \ol M (v_-,v_+,\kappa)= \bfR_\mb ( v_-,v_+,\kappa) +
\bfR_\mb ^* \big(v_-,v_+, E'(v_-){-}E'(v_+)\big).
\end{equation}
In \cite{LMPR17MOGG} this was shown by an explicit calculation, but for the
general case this follows via the abstract saddle-point theory for NESS
developed in \cite{Miel23NESS, Miel25?PGSN}, see Section
\ref{su:RedMembFormulas}.

Specifying to the case $\mob(v)=v^\beta$ and $E=\sfE_\varkappa$ one sees that
$\bfR_\mb = \wt\bfR_{\beta,\varkappa}$ depends on both, $\beta$ and
$\varkappa$, independently, while the transmission condition in
\eqref{eq:I.PME.transm} only depends on $\alpha=\beta+\varkappa-1$. Two cases
can be solved explicitly: 
\begin{align*}
\wt\bfR^*_{0,\varkappa}\big(v_-,v_+,[\xi]_0\big) = \frac12\, [\xi]_0^2 
\ \text{ and } \ 
\wt\bfR^*_{\beta,1}(v_-,v_+,[\xi]_0)= \frac1{\beta^2} \big(
v_-v_+\big)^{\beta/2} \CCC^*\big(\beta [\xi]_0\big),
\end{align*}
where the ``cosh function'' $\CCC^*(\zeta)=4 \cosh(\zeta/2) -4$ is the Legendre
transform of $\CCC$, namely
\begin{equation}
  \label{eq:CCC}
  \CCC(a)= \sup\bigset{a \zeta-\CCC^*(\zeta)}{\zeta \in \R} = 2 a \,
  \sinh^{-1}(a/2) -2 \sqrt{4{+}a^2} +4 . 
\end{equation}
In fact, the case of the linear
Fokker-Planck equation with $\alpha=1$ and $\beta = \varkappa =1 $ was derived in
\cite[Sec.\,4]{LMPR17MOGG} by formal calculations only. Here we give a rigorous
proof and generalize the result to the \AAA range 
\[
 \beta \in [0,1] \quad \text{and} \quad 2\varkappa+\beta \in [3,4],
\]
see Figure \ref{fig:beta.q}. Thus, for $\alpha \in [1/2,3/2]$ we can treat at
least one case. \EEE 

It is interesting to note that $\bfR^*_\mb (v_-,v_+,\cdot)$ is in general no
longer quadratic, as one might expect because all $\calR^*_\eps(v,\cdot)$ are
quadratic \AAA (see Remark \ref{rm:Quadr.bfRmb} for a quadratic
$\bfR^*_{\mb,\mafo{quadr}}$ that is simply fitted but seems to have less physical
relevance). \EEE  
For $\varkappa=1$ (the Boltzmann entropy) the explicit formula for
$\wt\calR_{\beta,1}$ gives an exponential behavior in $[\xi]_0$, which recovers
the classical exponential kinetics introduced by Ren\'e Marcelin in
\cite{Marc15CECP} for chemical reactions, transmission, absorption,
etc. Nowadays, it is often called Marcelin-De\;Donder kinetics, see
\cite{Fein72CKCC} and \cite[p.\,77]{Grme10MENT}.  Such exponential kinetic
relations arise also in large-deviation theory for jump processes, see
\cite[p.\,187-188]{DemZei87LDTA} and \cite{Leon95LDLR, MPPR17NETP}. For
$\varkappa>1$ we will show that
\[
  \wt\bfR^*_{\beta,\varkappa}(v_-,v_+,[\xi]_0) \sim \big( 
[\xi]_0\big)^{\gamma} \quad \text{with } 
 \gamma = \AAA 1+ \EEE  \alpha/(\varkappa{-}1) 
        = \AAA 2 \EEE  + \beta/(\varkappa{-}1), 
\]
see \eqref{eq:GrwothExpo}.
Of course, this is consistent with the fact that the transmission condition 
takes the form 
\[
A_\pm \pl_x(v^\alpha)\big|_{x=0^\pm} 
= \tdfrac b\alpha \big(v_+^\alpha - v_-^\alpha \big)
= b \,\pl_\delta \bfR^*_\mb (v_+,v_-; \xi_+{-}\xi_-) \ 
\text{ with } \xi_\pm = -E'(v_\pm),
\]
see \eqref{eq:GenTransCond}.  Thus, our result generalizes the approach in
\cite{LMPR17MOGG, PelSch23CGST} significantly by allowing general mobilities
$\mob$ and general energy densities $E$ (where $\mob(v)=v^\beta$ and
$E''(v)=v^{\varkappa-2}$ is just a special case) and provide the first rigorous
proof for the EDP-convergence.

In summary, we see that the notion of EDP-convergence is flexible enough to
study convergence of gradient systems in degenerate cases (loss of uniform
geodesic convexity, cf.\ \cite{Miel16EGCG}) where the effective dissipation
potential may completely change its structure, e.g.\ from quadratic to general
power laws or even exponential. Moreover, it becomes clear that EDP-convergence
is a convergence of the couple $(\calE_\eps,\calR_\eps)$, where by definition
the convergence of $\calE_\eps \overset{\Gamma}{\weak} \calE_0$ is independent
of $\calR_\eps$, but the convergence of $\calR_\eps$ cannot be seen
independently of $\calE_\eps$. Indeed, in general \AAA (microscopic) properties
of $\calE_\eps$ migrate into $\calR_\eff$, \EEE as is seen in the exponential 
behavior of $\delta\mapsto \wt\bfR^*_{\beta,1} (u_-,u_+,\delta)$, which is a
consequence of the logarithmic Boltzmann statistics $\xi= E'(v)=\log v$ in the
case $ \varkappa =1$.

\RRR At first sight. the migration of properties of $\calE_\eps$ into
$\calR_\eff$ may seem to be inconsistent with Otto's philosophy of ``separating
energetics and kinetics''. However, this separation has to be understood on one
given scale, and when singular limits are concerned, then the scales are mixed
again such that kinetics and energetics can mix on the microscale, via the cell
problem defined in \eqref{eq:I.def.Rmemb} and the associated NESS. Thus, the
effective (macroscopic) kinetic relation may indeed depend on both, the
microscopic energetics and the microscopic kinetics. In contrast, the effective
energetics only depends on the microscopic energetics. \EEE

The plan of the paper is as follows. In Section \ref{se:MembrModel} we describe
the  model in detail and introduce the notation for the scaling $
x=X_\eps(y)$. In Section \ref{su:Ass.GS} we present our precise assumptions on
the mobility $ u \mapsto \mob(u)$ and the energy density $u \mapsto E(u)$ and
discuss the range of validity for the power-law case $\mob(u)=u^\beta$ and
$E=\sfE_p$. Sections \ref{su:EDPcvg} to \ref{su:ProofMainRes} contain the main
results on the EDP-convergence in the unscaled, which is derived  directly from
the corresponding statements in the scaled setting. 

Sections \ref{se:APrioriEst} to \ref{se:Limsup} work in the scaled setting on a
sequence of functions $u_\eps(t,y) = v_\eps(t,X_\eps(y))$ with the
corresponding scaled dissipation functional $\wh\mfD_\eps$. We first provide a
priori estimates which \AAA lead \EEE to spatio-temporal compactness on the
bulk parts (see Proposition \ref{pr:SpaTimeComp}), whereas the degeneration
$X'_\eps(y)=\eps$ of the membrane part $y\in I_0={]0,1[}$ can be exploited in
the limiting continuity equation to show that the flux in membrane part is
spatially constant, see Section \ref{su:CharLimits}. The $\Gamma$-liminf
estimate for $\wh\mfD_\eps$, which is presented in Section \ref{se:Liminf},
tackles the missing compactness of $u_\eps$ in the membrane part by using weak
convergence methods and a convexification transformation, which is possible
only because of the scale separation and a subtle argument showing that the
weak* limit does not have a singular part. 
\AAA
Here we use ideas similar to \cite[Thm.\,3.2]{AMPSV12PLWG} and
\cite[Sec.\,4.3+4]{FreLie21EDTS}. 
\EEE
The construction of the recovery
sequences for $\wh\mfD_\eps$ for the $\Gamma$-limsup estimate are obtained only
under the more restrictive assumption that $\mfD_\eps$ is convex, which allows
for the application of Jensen's inequality when doing temporal smoothing via a
convolution kernel.

Section \ref{se:NESS.Rmemb} supplies the necessary background on NESS and their
characterization via so-called BER functions $\scrB_{\calE,\calR}$. This theory
is crucial for showing that the integrand of the $\Gamma$-limit $\mfD_0$ can
again be identified as $\calR_\eff(u,\dot u)+ \calR^*_\eff(u,-\rmD\calE(u))$
thus defining the effective dissipation potential $\calR_\eff$. 
Moreover, we derive estimates of the membrane dissipation potentials
$\wt\bfR_{\beta,p}$ arising for $\mob(v)=v^\beta$ and $E=\sfE_p$. In
particular, Section \ref{su:BoltzmannPower} provides an elegant and short way
to calculate $\wt\bfR_{\beta,1}$ explicitly.

\section{The membrane problem and the main results}
\label{se:MembrModel}

\subsection{Modeling using gradient structures}
\label{su:Setup}

We generalize the linear Fokker-Planck equation considered in
\cite[Sec.\,4]{LMPR17MOGG} in considering nonlinear diffusion on the
one-dimensional domain $\DDD ={]{-}1,1[}\subset  \R^1$, where we assume that the
mobility is depending on the space variable $ x\in \DDD$ in such a way
that it is of order $\eps$ in the interval $[0,\eps]$. The total mass of the
diffusion substance is assumed to be normalized to $1$ such that the density
distributions $v(t,\cdot)$ lie in 
\[
\calP(\DDD):= \bigset{ v \in \rmL^1(\DDD)}{ \int_\DDD v \dd x =1, \ v\geq
  0 \text{ a.e.\, } } . 
\]
Our diffusion equation for the density $v$ takes the form 
\begin{equation}
  \label{eq:PME1}
 \begin{aligned}
  \pl_t v = \DIV\Big( A_\eps(x)  
   \mob(v)  \nabla\big(\BBB E'(v  ) \EEE \big) \Big)
  &\quad \text{for }(t,x)\in {]0,\infty[}\ti \DDD,
\\
\nu \cdot \nabla\big( \BBB E'(v) \EEE \big) =0 
  & \quad  \text{for }(t,x)\in {]0,\infty[}\ti \pl\DDD.
\end{aligned}
\end{equation}
The associated gradient system is given by the energy $\calE$ and
the dual dissipation potential $\calR_\eps^*$ as defined in
\eqref{eq:I.calE.calR}. To write down an explicit expression of the dissipation
functional $\mfD_\eps$ we use the continuity equation
\begin{equation}
  \label{eq:ContSimp}
\pl_t v + \pl_x J = 0   
\end{equation}
which is always interpreted in the sense of distributions, i.e., 
\begin{equation}
  \label{eq:CE.weak} 
\iint_{[0,T]\ti\DDD} \big( \pl_t\phi(t,x) v(t,x) + \pl_x \phi(t,x) J(t,x) \big)
\dd x \dd t = 0 \quad  \text{for all } 
 \phi \in \rmC^\infty_\rmc({]0,T[}\ti \ol{\DDD }).
\end{equation}
Note that
$\phi(0,x)=\phi(T; x)=0$, whereas $\phi(t,\pm1)$ can be arbitrary. It will be
part of our analysis to provide conditions guaranteeing $J\in
\rmL^1([0,T]\ti\DDD) $ for all $v$ satisfying $\mfD_\eps(v)<\infty$. Indeed, by
Legendre transforming $\calR^*(v,\cdot)$ we have 
\[
\int_0^T \calR_\eps(v,\pl_t v)\dd t = \inf\Bigset{ \iint_{[0,T]\ti\DDD} 
  \frac{J^2}{2A_\eps\mob(v)}  \dd x \dd t }{ v\in
  \rmL^1( [0,T] \ti\DDD), \ \pl_t v +\pl_x J =0 }.  
\]
Thus, using a simple H\"older estimate and $\int_0^T\calR(v,\pl_t v)\dd
t<\infty$,  we find $J\in \rmL^q([0,T]\ti\DDD)$ with $q\in [1,2]$ whenever
$A_\eps\mob {\circ} v \in \SSS \rmL^{q/(2-q)} \EEE ([0,T]\ti\DDD)$, \AAA see
Proposition \ref{pr:SpaTimeComp}. \EEE  Subsequently, we will
always assume that $J$ is given by $v$ as the minimizer of the above definition of
$\calR_\eps$. 

\subsection{Rescaling the domain, the limit equation, and NESS}
\label{su:Rescale}

As suggested in \cite{AMPSV12PLWG,LMPR17MOGG} it is advantageous to
transform the variable $x \in \DDD $ in such a way that the
membrane has finite width equal to 1. We set 
$I :=[-1,2]$, where $I$ decomposes into the
three subintervals  
\[
I_-:=[-1,0],\quad I_0:={]0,1[}, \quad I_+:=[1,2].
\]
Here $y=x/\eps \in I_0$ corresponds to an adapted microscopic variable
describing the membrane.

The piecewise affine rescaling $Y_\eps: \DDD \to I; \: x \mapsto y=Y_\eps(x)$
and its inverse $X_\eps=Y_\eps^{-1}:I\to \DDD $ are explicitly given by
\[
Y_\eps(x)=
 \begin{cases} 
     x& \text{for } x\in [-1,0],\\
    x/\eps& \text{for } x\in [0,\eps],\\
    \frac{x+1{-}2\eps}{1{-}\eps}&\text{for } x\in [\eps,1],
 \end{cases}
\quad \text{and} \quad 
X_\eps(y)=
 \begin{cases} 
    y& \text{for }y\in I_-,\\
    \eps y& \text{for } y\in I_0,\\
    (1{-}\eps)y{-}1{+}2\eps&\text{for }y\in I_+.
 \end{cases}
\]
The construction was done such that the thin layer $[0,\eps] \subset \DDD$ is
magnified to the interval $[0,1]=I_0\subset I$. We will use that $X_0$ is
still Lipschitz continuous with $ X'_0(y) = 0$ on $I_0$. 

We transform the density $v$ into the new function
\[
u(t,y) = v(t,X_\eps(y)) \quad \text{or inversely} \quad v(t,x)=u(t,Y_\eps(x)).
\]
Then $u(t)$ is no longer a probability density, but
$\int_I u X'_\eps \dd y \equiv 1$. Moreover, the continuity equation (CE)
$\pl_t v+\pl_x J=0$ transforms into
\begin{subequations}
  \label{eq:mfFeps}
\begin{equation}
  \label{eq:CE.transformed}
  \pl_t ( X'_\eps\, u) + \pl_y Q=0 \quad \text{with } Q(t,y)=J(t,X_\eps(y)). 
\end{equation}
Using $(\pl_x v)(t,X_\eps(y))= \pl_y u(t,y)/X'_\eps(y)$ we define the
transformed energy and dissipation functional $\wh\calE_\eps$ and
$\wh\mfD_\eps$ for $\eps\geq 0$ via 
\begin{equation}
  \label{eq:def.mfFeps}
\begin{aligned}
&\wh\calE_\eps(u):= \calE( u{\circ} Y_\eps) = \int_I E(u(y)) X'_\eps(y) \dd y,
\\
& \wh\mfD_\eps(u):= \mfD_\eps(u{\circ}Y_\eps) 
 = \iint_{[0,T]\ti I}\!\!  \Big( \frac{Q^2}{2\bbA_\eps\mob(u)}  
   + \frac{\bbA_\eps}{2} \big| \pl_y H(u)\big|^2\Big) \dd y \dd t  
\end{aligned}
\end{equation}
\end{subequations}
with $\bbA_\eps(y): = A_\eps(y) / X'_\eps(y) $ and $H$ from \eqref{eq:I.def.H}. 
Here $Q$ is the solution of the scaled continuity equation
\eqref{eq:CE.transformed} minimizing the expression defining
$\wh\mfD_\eps$, see Lemma \ref{le:DifferentFlux} for a characterization.  

The main observation is that $\bbA_\eps$  is
piecewise constant taking only the values $A_-$, $b$, and $A_+/(1{-}\eps)$,
such that $\bbA_\eps$ and $1/\bbA_\eps$ are nicely bounded for $\eps\in
{]0,1/2[}$. \AAA In particular, we can define $\bbA_0$ by continuous extension with
$\bbA_0 (y) \in \{A_-, b, A_+\}$, which gives \EEE the uniform convergence  
$ \| \bbA_\eps {-} \bbA_0 \|_{\rmL^\infty} = O(\eps)$. Thus, the functional
$\wh\mfD_\eps$ has a simple dependence on $\eps$. The only nontrivial dependence
on $\eps$ stems from the scaled continuity equation \eqref{eq:CE.transformed},
where the prefactor satisfies $ X'_\eps(y)=\eps$ for $y \in I_0=[0,1]$. This
degeneration will lead to a scale separation such that $u(t,\cdot)|_{I_0}$ can
adjust instantaneously to the data given at $y=0$ and $y=1$ at time $t$. 

While the rescaled dissipation functional $\wh\mfD_\eps$ will be our main interest,
we can also look at the rescaled version of the diffusion equation
\eqref{eq:I.GenDiffEqn}, namely 
\begin{equation}
  \label{eq:RescalDiffEqn}
  X'_\eps(y) \pl_t u(t,y) = \pl_y\big( \bbA_0 \mob(u) E''(u)\pl_y u\big) \quad
  \text{for } (t,y) \in \R_>\ti I,
\end{equation}
where the layer thickness $\eps$ only occurs in the prefactor $X_\eps'$. 

We note that $X_\eps$ has a Lipschitz continuous limit $X_0$ satisfying
$X'_0(y)= 0$ in $I_0$ and $X'_0(y)=1$ otherwise. \AAA Using the ideas in 
\cite{NeuJag07ETCR, GaNeKn17DETC, CGLP19DAEI,CiDaPo24EICP} it follows that \EEE
the limiting  equation reads 
\begin{equation}
  \label{eq:LimitDiffEqn}
   X'_0 \pl_t u(t,y) = \pl_y\big( \bbA_0 \mob(u) E''(u)\pl_y u\big) \ \text{ in }
   \R_>\ti I, \quad \pl_y u=0 \text{ on }\R_>\ti \pl I, 
\end{equation}
which consists of two diffusion equations on $\R_>\ti I_-$ and $\R_>\ti I_+$,
and a quasistatic equation on the rescaled thin layer $\R_>\ti I_0$, in which
no time derivative appears any more. This means that inside the membrane the
profile is adjusted on a much faster time scale and can be consider as
quasistatic with respect to the time scale of the diffusion in $I_-$ and
$I_+$. These profiles are so-called Non-Equilibrium Steady States (NESS) and
are defined by the boundary values $u_-(t):= u(t,0) $ and $u_+(t)=u(t,1)$ and
the static equation
\begin{align}
  \label{eq:NESS.Eqns}
  0= b\,\pl_y \big(\mob(u) E''(u)\pl_yu \big) , \quad u(0)=u_-, \quad u(1)=u_+.
\end{align}
Using the monotone function $G(u)= \int_0^u \mob(r) E''(r) \dd r$, we easily
find the unique NESS 
\begin{equation}
  \label{eq:ExplicitNESS}
  u_\mafo{NESS}(y) = G^{-1} \Big( (1{-}y)\,G(u_-) + y\,G(u_+)\Big) \quad
  \text{for } y\in  I_0=[0,1] . 
\end{equation}

As \eqref{eq:LimitDiffEqn} can be rewritten as
\SSS $ X'_0 \pl_t u = \pl_y \big( \bbA_0 \pl_y G(u)\big)$ \EEE we see that
$y \mapsto \bbA_0\pl_y G(u)\big|_{(t,y)}$ has to be continuous. As
$u_\text{NESS}$ from \eqref{eq:ExplicitNESS} \AAA satisfies \EEE
$\pl_y ( G{\circ}u_\text{NESS})(y) \equiv G(u_+)-G(u_-)$ we obtain the desired
transmission conditions \AAA  \eqref{eq:DerivedTransCond}. Thus, the limit equation \eqref{eq:LimitDiffEqn} can be reduced to the smaller
domain $\R_> \ti I_\PM$ and we arrive at the limiting system
\begin{subequations}
\label{eq:LimitEqn}
\begin{align}
 \pl_t u(t,y) & = A_- \pl_y^2 G(u) \quad \text{in } \R_>0\ti I_-, \qquad \pl_y
 u(t,-1)=0 \text{ for } t>0,  
\\
  \label{eq:DerivedTransCond}
  &A_+\pl_y( G{\circ} u) (t,1) = A_- \pl_y (G {\circ} u)(t,0) = b \big(
  G(u(t,1)) - G(u(t,0))\big). 
\\
 \pl_t u(t,y) & = A_+ \pl_y^2 G(u) \quad \text{in } \R_>0\ti I_+, \qquad \pl_y
 u(t,2)=0 \text{ for } t>0,  
\end{align}
\end{subequations}

In the following remark we show that it is not difficult to equip the limiting
system with a gradient structure have a quadratic dissipation potential. This
approach was indeed taken in \cite{Miel13TMER}, but the present paper shows that
such quadratic potentials are physically relevant only in the case $\mob(u) =
\mafo{const}$.   

\begin{remark}[Linear kinetic relation for membrane transmission]
\label{rm:Quadr.bfRmb}
It is easy to see that the limiting equation has the gradient structure
$ \big( \bfQ_\mafo{lim},\calE_0, \calR_\mafo{quad} \big)$ with 
\begin{align*}
&\bfQ_\mafo{lim}=\rmL^1(I_-\cup I_+), \quad 
  \calE_0(u)=\int_{I_-\cup I_+} E(u(y))\dd y ,
\\
&\calR^*_\mafo{quad}(u,\xi) = \sum_{\sigma \in \{-,+\}} \int_{I_\sigma}
      \frac{ A_\sigma\mob(u)}2 \big(\pl_y \xi\big)^2 \dd y +  \frac b2\,
      \frac{G(u(1)){-}G(u(0))}{E'(u(1)){-}E'(u(0))} \,\big( \xi(1))-
      \xi(0)\big)^2.  
\end{align*}
Similar ad hoc constructions of quadratic dissipation potentials have been 
given in many contexts, see \cite{Maas11GFEF, CHLZ12FPEF, 
 Miel13GCRE, ErbMaa14GFSD} for Markov processes on a finite state
space and \cite{CarMaa14A2WM, MitMie17EGSL, CarMaa17GFEI} for Lindblad
equations for quantum systems, and for reaction diffusion systems including
transmission in \cite{ Miel11GSRD, GliMie13GSSC, Miel13TMER}. 
\end{remark}
\EEE

\subsection{Assumptions on $E$ and $\mob$}
\label{su:Ass.GS} 

We now formulate the specific assumptions on the \AAA coefficient function
$\bbA_0$ (cf.\ paragraph after \eqref{eq:def.mfFeps}), \EEE on
the energy function $E$, and the mobility function $\mob$.
\begin{subequations}
\label{eq:AllAssump}
\begin{align}
\label{eq:bbA}
& A_-,A_+, b>0, \quad \mob\in \rmC^0(\R_\geq), \ \mob(r) >0 \text{ for }r>0 , 
\\
& \label{eq:E.bounds}
\begin{aligned}  &E\in \rmC^0({[0,\infty[}) \cap \rmC^2({]0,\infty[}), \ \ E
  \text{ is strictly convex and superlinear}, \\
&\exists\,  q_E\geq 1, \ c_\rmE>0\ \ \forall\, u\geq 0: \quad   E(u)\geq c_\rmE
u^{q_\rmE} - 1/c_\rmE, 
 \end{aligned}
\\[0.3em]
&\label{eq:Ass.for.H}
  \int_0^1 \mob(r)^{1/2} E ''(r) \dd r < \infty.
\end{align}   
The last assumption allows us to define $H$ as in \eqref{eq:I.def.H} that
encodes combined information on $E$ and $\mob$. We impose the further
assumptions:
\begin{align}
& \label{eq:Cond.H.qH} 
  \exists\, \qsubH\geq \frac12,\  C>0\ \forall\,u>0: \quad  \frac{u^\qsubH}C -C
  \leq H(u) \leq C \big( 1{+} u^\qsubH\big),  
\\[0.3em]
&  \label{eq:mob.H.growth}
  \exists\, C_\mob>0 \ \forall\, u\geq 0:\qquad \qquad \qquad  
  \mob(u) \leq C_\mob \big(1{+} H(u)\big)^2 . 
\end{align}
\end{subequations}

Our main result on the EDP convergence relies on the restrictive assumption
that the two functions 
\begin{equation}
  \label{eq:ConcaveTwofold}
  u\mapsto \mob(u) \qand  u\mapsto \sfg_\mafo{id}(u):=\frac1
         {\mob(u)\,(E''(u))^2} \ \text{ are concave}.
\end{equation}
This condition guarantees that the functionals $\mfD_\eps$ and $\wh\mfD_\eps$
are convex, which is a very helpful property.  However, we believe that this
condition is not essential and may be relaxed by a technically more advanced
analysis. To substantiate this point, we introduce the following assumption
that is significantly weaker, see the discussion in Proposition
\ref{pr:PowerLawCvx}. This weaker condition will be enough to pass to the limit
in the membrane part, where there is no temporal compactness, such that the
full $\Gamma$-liminf estimate is valid.
\begin{equation}
  \label{eq:Cvx.Ass.psi}
 \begin{aligned}  
  \exists\: &\text{homeomorphism } \psi:\R_\geq \to \R_\geq 
  \ \ \exists\, C_\psi>0\ \forall\, w>0: 
 \\ 
  &\text{(i) } \  \ w \mapsto \sfn_\psi(w):=\mob(\psi(w))
    \text{ is concave}, 
 \\
  &\text{(ii) } \ w\mapsto \sfg_\psi(w):=\frac1
         {\mob(\psi(w))\big(\psi'(w)\,E''(\psi(w))\big)^2}
    \: \text{ is concave} ,
 \\
  &\text{(iii) } w  \leq C_\psi 
  \big(1{+}H(\psi(w))\big)^2,     
\\
 & \text{(iv) } \ \psi(0)=0, \quad \psi,\,\psi^{-1} \in \rmC^1(\R_>).
\end{aligned}
\end{equation}

The following result shows that the above assumptions can be met by looking at
mobilities and energies of power-law type. Figure \ref{fig:beta.q} displays the
relevant regions for different parts of the conditions. The linear
Fokker--Planck equation is obtained for $\beta+\varkappa =2$, where
$\varkappa =1$ and $\beta=1$ provides the Otto gradient-structure with the
Boltzmann entropy, see \cite{JoKiOt98VFFP}, whereas \cite{Otto01GDEE}
treats the porous medium equation with $\beta=1$ and $\varkappa \geq 1-1/N$ and
$\varkappa >N/(N{+}2)$ with the space dimension $N$. 
\begin{figure}
\centering 
\begin{tikzpicture}[scale=1.5]
\draw[gray!6, fill] (-2,-0.6) rectangle (5,3.5);

\draw[thick, ->] (0,0) -- (4.2,0) node[right]{$\beta$};
\draw[thick, ->] (0,0) -- (0,3) node[left]{$\varkappa $};

\draw[ultra thick, color=red] (0,2)-- node[pos=-.04, left, sloped]
{lin.\,FP} (2,0) ; 

\draw[color=white] (2,3)--node[sloped]
{\textcolor{red!60}{slow diffusion}} (3,2); 
\draw[color=white] (-1.5,2.9)--node[sloped]
{\textcolor{red!60}{fast diffusion}} (-0.5,1.9); 

\draw[opacity=0.2, fill =green, very thin] (0,0)--(1,0)
                         --(1,3.3) --(0,3.3)--cycle;
\node[right, rotate=90] at (0.5,1.85) {\textcolor{green!50!black}{$\mob$ concave}};

\draw[opacity=0.2, fill=blue, very thin] (0,1.5)--(3,0)--(4,0)--(0,2)--cycle;

\node[right] at (1.15,1.2) 
{\textcolor{blue!60!black}{$\sfg_\mathrm{id}$ concave}} ;

\draw[fill = black!30 ,very thin ] (0,0)-- (4,0) -- 
   (4,1) -- (1,1) -- (0,1.5) -- cycle;

\draw (1,0) node[below]{$1$} -- (1,1)-- (0,1)  
                node[left]{$1$}  ;

\draw (2,1) -- (2,0) node[below]{$2$}  ;

\end{tikzpicture}
\caption{The different regimes for $\mob(u)=u^\beta$ and
  $\sfE''_\varkappa (u)=u^{\varkappa -2}$ leading to the equation $\dot u = \pl_x\big(
  u^{\beta+\varkappa -2}\pl_x u\big)$. The weaker condition
  \eqref{eq:Cvx.Ass.psi} hold above the gray area, whereas the stronger condition
  \eqref{eq:ConcaveTwofold} holds in the closed parallelogram obtained as
  intersection of the light blue and the light green strips.}
\label{fig:beta.q} 
\end{figure}

\begin{proposition}[Power-law $E$ and $\mob$] 
\label{pr:PowerLawCvx} 
Assume that $E=\sfE_\varkappa $ and $\mob$ are given via
\begin{equation}
  \label{eq:PolyCase}
  \sfE_\varkappa''(u)=u^{\varkappa -2}, \ \sfE_\varkappa (1) = 
  \sfE'_\varkappa(1)=0,  \qand  \mob(u)=u^\beta.  
\end{equation}
Then, the conditions \eqref{eq:AllAssump} and \eqref{eq:Cvx.Ass.psi} hold if
and only if  
\begin{align}
\label{eq:Model.Param}
\beta\geq 0 \qand \varkappa\geq \max\{ 1, (3{-}\beta)/2\},
\end{align}
Then, the conditions \eqref{eq:AllAssump} and \eqref{eq:ConcaveTwofold} hold if
and only if  
\begin{align}
\label{eq:Model.Param2}
\beta \in  [0,1] \qand 2\varkappa +\beta \in [3,4].
\end{align}
\end{proposition}
\begin{proof}
Assuming \eqref{eq:PolyCase} we first observe $\qsubE = \max\{ \varkappa ,1\}$
and $H'(u) = u^{\varkappa +\beta/2-2}$. Hence, \eqref{eq:Ass.for.H} is satisfied
for $2\varkappa+\beta>2$, and then $H(u)= c u^{\varkappa
  +\beta/2-1}$ for all $u\geq 0$.
 
In particular, we have $\qsubH= p + \beta/2-1$ which leads to the
restriction $2\varkappa + \beta \geq 3$. The condition \eqref{eq:mob.H.growth}
means $\beta \leq 2 \qsubH$ and reduces to $\varkappa \geq 1$. This shows that 
\eqref{eq:AllAssump} holds if and only if \eqref{eq:Model.Param}.  
Next we show that in this case condition \eqref{eq:Cvx.Ass.psi} can always be
satisfied by a transformation $u = \psi(w)=w^a$ with $a>0$.
For the four conditions in \eqref{eq:Cvx.Ass.psi}  use the
transformation $u=\psi(w)= w^a$ with $a>0$ and find 
\[
\sfn (w)= w^{\beta\, a} \quad \text{and} \quad 
\sfg(w) = w^\gamma \text{ with } \gamma= 2 -a\,(2\varkappa {+}\beta{-}2).
\]
The concavities (i) and (ii) hold if the exponents lie in $[0,1]$, which means 
\[
a \in {]0,1/\beta]} \qand  a \in [\sigma/2, \sigma] \ \text{ with }
\sigma=\frac2{2\varkappa{+}\beta{-}2} > 0.
\]
Such $a>0$ can be found if and only if $2\varkappa +\beta > 2$ (for $\sigma>0$)
and $\varkappa \geq 1$ (for $1/\beta\geq \sigma/2$).  Using
$z \mapsto H(\psi(z))^2 \sim z^b$ with $b = 2 a/\sigma \geq 1$; we see that
(iii) and (iv) in \eqref{eq:Cvx.Ass.psi} hold.

The conditions \eqref{eq:Model.Param2} follow easily
from $\mob(u)=u^\beta$ and $\sfg_\text{id}(u)= u^{4-2\varkappa -\beta}$.
\end{proof}


\subsection{EDP-convergence and the effective kinetic relation} 
\label{su:EDPcvg}

The total dissipation functional is given as 
\begin{equation}
  \label{eq:Def.mfDeps}
  \mfD_\eps(v) = \iint_{[0,T]\ti \DDD}\! \Big( \frac{J^2}{2A_\eps\mob(v)} +
  \frac{A_\eps}2 \big| \pl_x H(v)\big|^2 \Big) \dd x \dd t . 
\end{equation}
This form of the dissipation functional shows the dependence on the small
parameter $\eps$ explicitly. The first term becomes singular \AAA because
$A_\eps$ \EEE is in the denominator, hence the flux through the membrane will
be \AAA relatively \EEE small. In the second term $A_\eps$ appears as a
prefactor, which allows the function $v$ to develop a large slope such that in
the limit a discontinuity will appear.

Before going into the details, we will provide the main result on the
$\Gamma$-convergence of $\mfD_\eps$. In the following 
result the convergence $\mfD_\eps  \overset{\Gamma_\rmE}{\weak} \mfD_0$  is
defined via: 
\begin{align*}
\text{(i) }\quad& \hspace{-0.5em}\left. \ba{@{}c} v_\eps \weak v_0 \text{ in } 
     \rmL^1([0,T]\ti\DDD) \text{ and} 
\\ 
   \sup_{\eps} \sup_{ t\in [0,T]} \calE(v_\eps(t)) < \infty \ea\right\}
 \ \Longrightarrow \ \liminf_{\eps\to 0} \mfD_\eps(v_\eps) \geq \mfD_0(v_0), 
\\[0.5em]
\text{(ii)}\quad & \forall\, \wt v_0\in \rmL^1([0,T]\ti\DDD)  \ \exists\, (\wt
v_\eps)_{\eps} \subset \rmL^1([0,T]\ti\DDD) \text{ such that } \\
& \qquad \wt v_\eps \weak \wt v_0 \text{ in } \rmL^1([0,T]\ti\DDD) \ \text{ and }
\  \limsup_{\eps\to 0} \mfD_\eps(\wt v_\eps) \leq \mfD_0(\wt v_0).
\end{align*} 
Moreover, we define the subintervals $\DDD^-:={]{-}1,0[}$ and $\DDD^+:={]0,1[}$.

\begin{theorem}[$\Gamma$-limit of $\mfD_\eps$]
\label{th:MainResUnscaled}
Let $\mfD_\eps$ be given as above and assume that $E$ and $\mob$ satisfy the
assumptions \eqref{eq:AllAssump} and \eqref{eq:ConcaveTwofold}. Then, we have 
$\mfD_\eps  \overset{\Gamma_\rmE}{\weak} \mfD_0$, where  $\mfD_0$ is
given in the form
\begin{align*}
\mfD_0(v) &= \iint_{[0,T]\ti \DDD^-} \Big(\frac{J^2}{2A_-
  \mob(v)}  + \frac{A_-}2 \big| \pl_x H(v)|^2 \Big)\dd x \dd t 
\\
&\quad + \int_0^T \!\Big(b\, \bfR_\mb \big(v_-,v_+, \kappa/b\big) +
b \, \bfR^*_\mb  \big( v_-,v_+, E'(v_-){-}E'(v_+)  \big) \Big) \dd t 
\\
&\quad +  \iint_{[0,T]\ti \DDD^+} \Big(\frac{J^2}{2A_+ \mob(v)}  
  + \frac{A_+}2 \big| \pl_x H(v)|^2 \Big)\dd x \dd t , 
\end{align*}
where $\bfR_\mb $ is defined in \eqref{eq:I.def.Rmemb},  the lower and
upper traces $v_-$ and $v_+$ are given by $v_\pm(t)=\lim_{h\searrow 0} v(t,\pm
h)$, and 
$\kappa$ is the flux through the membrane, namely $\kappa(t) =
\frac\rmd{\rmd t} \int_{\DDD^+}v(t,x)\dd x$.  
\end{theorem}

The proof of this result will be obtained through the corresponding
$\Gamma$-convergence of a rescaled version of $\mfD$ given in 
in Theorem \ref{th:Main.Gamma.Cvg} and then  \SSS completed in \EEE 
Section \ref{su:ProofMainRes}. In fact, the stronger condition
\eqref{eq:ConcaveTwofold} is only needed for the $\Gamma$-limsup estimate (ii),
whereas the $\Gamma$-liminf estimates works with the weaker assumptions
\eqref{eq:Cvx.Ass.psi}, see Theorem \ref{th:Main.Gamma.Cvg}.  

This allows us to state our result on EDP-convergence of the gradient systems 
$(\calP(\DDD),\calE,\calR_\eps)_\eps$  to the effective membrane system 
$(\calP(\DDD),\calE,\calR_\eff)$. According to \cite{DoFrMi19GSWE,
  MiMoPe21EFED, Miel23IAGS} convergence in the sense of the
Energy-Dissipation Principle (EDP) is defined as follows:
\[
(X,\calE_\eps,\calR_\eps) \overset{\,\text{EDP}\,}\longrightarrow
(X,\calE_0,\calR_\eff) \ 
\Longleftrightarrow \ \begin{cases} \calE_\eps \overset{\Gamma}\to \calE_0
  &\text{in } X, \\  
\mfD_\eps\overset{\Gamma_\rmE}{\weak} \mfD_0& \text{in }\rmL^1([0,T];X), 
\end{cases}
\] 
with $\mfD_\eps(u)=\int_0^T \! \big( \calR_\eps(u,\dot u) {+} 
\AAA \calR^*_\eps \EEE (u,{-}\rmD\calE(u))\big) \dd t$ and
$\AAA \mfD_0 (u) \EEE =\int_0^T\!\big( \calR_\eff(u,\dot u) {+} 
\calR^*_\eff(u,{-}\rmD\calE(u))\big) \dd t$. Note that we use the notation
$\calR_\eff$ rather than $\calR_0$, because $\calR_\eff$ may depend also on
$\calE$ (as it does in this paper) and there might be a ``natural limit''
$\calR_0$ of $\calR_\eps$ that is different from $\calR_\eff$. 

The explicit representation of $\mfD_0$ immediately tells us that it is indeed
given in terms of a dual pair $\calR_\eff$ and $\calR_\eff^*$, namely with 
\begin{align}
\label{eq:2.16b}
\calR^*_\eff(v,\xi) :=& \int_{\DDD^-}\! \frac{A_-}2 \mob(v) \SSS \big|\pl_x
\xi\big|^2 \EEE   \dd x + b \bfR^*_\mb \big( v_\PM , [\xi]_0 \big)  
 + \int_{\DDD^+} \frac{A_+}2 \mob(v) \SSS \big|\pl_x \xi\big|^2 \EEE \dd x,
\end{align}
where $v_\PM \SSS = \big( v \EEE (0_-),v(0_+) \big) $ and $[\xi]:=\xi_+ - \xi_-$ with
$\xi_\pm=\xi(0_\pm)$. Hence, $\calR^*_\eff $ has two bulk
parts for $\DDD^-:={]{-}1,0[}$ and $\DDD^+:={]0,1[}$ and the membrane part in
$\bfR^*_\mb$, which only depends on the jump $[\xi]_0$ of the chemical
potential $\xi$ at the membrane.

We continue to use the symbol ``$\pm$'' in the
usual way to indicate two cases, whereas ``$\PM$'' defines a pair via 
$a_\PM:=(a_-,a_+)$ or a union as in $I_\PM :=I_-\cup I_+$.

\begin{corollary}[EDP-convergence]
\label{co:EDPcvg} Under the assumptions of Theorem \ref{th:MainResUnscaled} we
have 
\[
(\calP(\DDD),\calE,\calR_\eps)_\eps \ \overset{\text{EDP}}\longrightarrow \
(\calP(\DDD),\calE,\calR_\eff) 
\]
with the effective dissipation potential given in \eqref{eq:2.16b} or \eqref{eq:I.R*eff}. 
\end{corollary}

We emphasize that the definition of EDP-convergence provides a \emph{unique}
effective dissipation potential that is extracted from the (unique)
$\Gamma$-limit $\mfD_0$. This is in contrast to the approach in
\cite{FreLie21EDTS} where only the $\Gamma$-liminf estimate and a chain rule is
established. In approaches using $\Gamma$-liminf estimates only (see also in
\cite{SanSer04GCGF}) the limiting equation is rigorously derived, but the
additional information on the effective gradient structure is not uniquely
determined, see \cite[Rem.\,5.24]{Miel23IAGS} for a counterexample.

For the limiting gradient system $(\calP(\DDD),\calE,\calR_\eff)$ we can derive
the associated gradient-flow equation, which reads in the abstract form
$\dot q= \rmD_\xi\calR^*_\eff\big(q,{-}\rmD\calE(q)\big)$, by using the 
weak form $\frac\rmd{\rmd t} \langle \phi, q\rangle = \rmD_\xi
\calR^*(q, - \rmD\calE(q))[\phi]$ for a suitable set of dense test functions
$\phi$. 

We choose test functions $\phi_\PM:=(\phi_-,\phi_+)\in
\rmC^\infty(\ol\DDD^-)\ti \rmC^\infty (\ol\DDD^+)$  (which do not need to have the
same trace on $x=0$). With this we obtain 
\begin{align}
\label{eq:D.calR*.phi}
\begin{aligned}
\rmD_\xi \calR^*_\eff(u,\xi)[\phi_\PM] &
= \int_{\DDD^-} \!\!\frac{A_-\mob(v)}2 \,\pl_x\xi \, \pl_x\phi_-\dd x 
  + b \,\pl_\delta \bfR^*_\mb( v_\PM, [\xi]_0) [\phi]_0
\\ &
\quad  + \int_{\DDD^+} \!\! \frac{A_+\mob(v)}2 \,\pl_x\xi \, \pl_x\phi_+\dd x ,
\end{aligned}
\end{align}
where $[\phi]_0=\phi_+(0){-}\phi_-(0)$. For treating the time derivative, it is
important to take into account the flux $\kappa$ through the membrane. 
The continuity equation takes the form  (see Proposition \ref{pr:Prop.Q0}) 
\begin{equation}
  \label{eq:ContWith.kappa}
  \frac\rmd{\rmd t} \Big( \int_{\DDD^-}\!\! \phi_- v\dd x + 
     \int_{\DDD^+}\!\! \SSS \phi_+ v \dd x \EEE \Big) 
 =  -\int_{\DDD^-} \!\!\pl_x \phi_- J_- \dd x + \kappa\,[\phi]_0
    - \int_{\DDD^+} \!\!\pl_x \phi_+ J_+\dd x,
\end{equation}
which implies $\kappa(t)= \frac\rmd{\rmd t} \int_{\DDD^+}v(t,y)\dd y $ (to see
this, choose $\phi_\PM \equiv (0,1)$). 

Comparing \eqref{eq:D.calR*.phi} evaluated at $\xi=-\rmD\calE(v)=-E'(v)$ and
\eqref{eq:ContWith.kappa} we find the constitutive equations
\begin{equation}
  \label{eq:Constitutive}
  J_\pm= -A_\pm \mob(v)\pl_x\big(E'(v)\big) \qand 
\kappa = b \, \rmD_\delta \bfR^*_\mb\big( v_\PM,E'(v_-){-}E'(v_+) \big).
\end{equation}

Moreover, to obtain the strong form \eqref{eq:I.PME.transm} one assumes
sufficient smoothness and integrates by parts on the right-hand side of
\eqref{eq:ContWith.kappa} to 
find the additional boundary condition
\[
J_-(t,-1)=0, \quad J_-(t,0)=-\kappa(t)=J_+(t,0), \quad J_+(t,1)=0.
\]
The middle conditions and the following result provide the transmission
conditions \eqref{eq:DerivedTransCond}.

\begin{proposition}[Transmission condition]
\label{pr:TransmCond} Under the assumptions of Theorem
\ref{th:MainResUnscaled} we have 
\[
\kappa = b \, \rmD_\delta \bfR^*_\mb\big( v_\PM,E'(v_-){-}E'(v_+) \big)
= b \,\big( G(v_+) - G(v_-)\big).
\]
\end{proposition} 
\begin{proof} 
This result will be established  in Corollary \ref{ex:Exist.Rmb}, see
\eqref{eq:GenTransCond}.    
\end{proof} 

We reiterate our main focus in deriving a consistent kinetic relation $\kappa =
b \rmD_\delta \bfR_\mb\big(v_\PM, [\xi]_0\big)$ \BBB which contains more
information \AAA than \EEE the transmission condition $\kappa =  b \,\big( G(v_+) -
G(v_-)\big)$. Thus, the importance is that the effective gradient system
$(\calP(\DDD), \calE,\calR_\eff)$ contains the additional information of the
energy functional $\calE$ and the effective dissipation potential $\calR_\eff$
that incorporates the kinetic relation for the membrane through $\bfR_\mb$.

\subsection{Statement of the $\Gamma$-convergence result for $\wh\mfD_\eps$}
\label{su:MainResult}

For defining the $\Gamma$-limit $\wh\mfD_0$ of the family
$\wh\mfD_\eps$ we need to distinguish the membrane part in
$\Omega_0=[0,T]\ti I_0$ from the bulk parts $\Omega_\PM = [0,T]\ti
I_\PM$ where $I_\PM:=I_+\cup I_-$. In the latter case we will derive
strong convergence results for the $u_\eps$ while
for $u_\eps|_{\Omega_0}$ we have to rely on weak-convergence, because only
$X'_\eps u_\eps$ is controlled (in the strong norm).

We first define a set of $u$ that contains the domain of $\wh\mfD_0$. Recalling 
$X'_0=\bm1_{I_\PM}$  we set
\begin{align}
   \label{eq:def.bfL1}
\begin{aligned} 
   \bfL_1:= \Big\{ \: u \in \rmL^1(\Omega) \; \Big|\;  u \in
     \rmL^1\big([0,T];\rmC(I)\big), \  X'_0u(t,\cdot)\in \calP(I) \text{ a.e.\
       on } [0,T],& \\  
        \exists\, Q \in \rmL^1(\Omega)   
   \text{ such that \eqref{eq:LimitCE.3} holds} &\:\Big\}, 
\end{aligned}
\end{align}
where the limiting continuity equation has the form 
\begin{equation}
  \label{eq:LimitCE.3} 
   \pl_t(X'_0 u) + \pl_y Q =0  \quad\text{in the sense of distributions}. 
\end{equation}

Because of $X'_0(x)=0$ for $y \in I_0$ a flux $Q$ associated with $u\in \bfL_1$ 
satisfies $\pl_y Q(t,y) =0$ for $y\in I_0$. Hence, there exists a function
$\kappa \in \rmL^1(0,T)$ such 
that 
\[
 Q(t,y)= \kappa(t) \quad \text{for a.a.\ } (t,y)\in \Omega_0. 
\]
In other words, for a.a.\ $t\in [0,T]$ the flux $Q(t,\cdot)$ is
constant with respect to the microscopic membrane variable $y \in
I_0$. Moreover, we can test the limit continuity equation
\eqref{eq:LimitCE.3} with 
functions $\varphi$ satisfying $\varphi(t,y)= \phi_\pm(t)$ for $y \in
I_\pm$. Defining the masses $M_\pm(t)$ in the upper and
lower part of the bulk domains $I_\pm$, respectively, via 
\[
M_\pm(t):= \int_{I_\pm} u(t,y) \dd y, \quad \text{giving } 
M_+(t)+M_-(t)=1 \ \text{ a.e.\ on }
[0,T], 
\]
leads to $\int_0^T \big(\dot \phi_+ M_+ + \dot\phi_-
M_- + \kappa(\phi_+{-}\phi_- ) \big) \dd t =0$ for all
$\phi_\pm\in \rmC_\rmc^1({]0,T[})$.  This identity reveals the role of
$\kappa$ as the flux through the membrane, because it shows that
\begin{equation}
  \label{eq:flux.kappa.3}
   \frac\rmd{\rmd t} M_+(t) = \kappa(t) = -   \frac\rmd{\rmd t}
   M_-(t)   \qquad \text{a.e.\ in } [0,T].  
\end{equation}

With the above we can define three linear maps on $\bfL_1$
providing the flux $\kappa$ through the membrane and the two traces
$U_\PM=(u(\cdot,0),u(\cdot,1))$ of the relative densities, namely
\begin{equation}
  \label{eq:FluxTrace.3}
  \sfF(u):= \kappa \in \rmL^1(0,T) 
  \qand 
  \sfT(u):= (u(\cdot,0),u(\cdot,1))=:U_\PM \in \rmL^1(0,T)^2.
\end{equation}
Now we are able to write down the limit functionals $\wh\mfD_0:
\rmL^1(\Omega) \to [0,\infty]$ as follows:
\begin{subequations}
 \label{eq:mfD0limit}
\begin{align}
  \label{eq:mfD0.3A}
&  \wh\mfD_0(u):=  
    \begin{cases} 
     \wh\mfD^\text{bulk}_0\big(u\big) +
     \wh\mfD^\mb _0\big(\sfT(u),\sfF(u)\big)
   & \text{for } u \in \bfL_1,\\    \infty& \text{otherwise,} 
   \end{cases} 
\\
  \label{eq:mfD0.3B}
&\text{with } \wh\mfD^\text{bulk}_0(u)= \iint_{\Omega_\PM} \!\!
\Big(\frac{\SSS Q^2_\pm \EEE }{2\bbA_0\mob(u)} + \frac{\bbA_0}{2} \big|
  \pl_y H(u) \big|^2  \Big) \SSS \dd y \EEE \dd t 
\\
  \label{eq:mfD0.3C}
&\text{and }  \wh\mfD^\mb _0(U_\PM,\kappa) = \int_0^T  b\, \ol M \big(U_\pm(t),
\kappa(t)/b \big)  \dd t 
\end{align}
\end{subequations}
with $\ol M$ is defined in \eqref{eq:I.def.Rmemb} (or \eqref{eq:RMJ} below) and
can be expressed via $\bfR_\mb$ and $\bfR^*_\mb$ as in \eqref{eq:olM.Rmb.R*mb}.
Here $Q_\pm$ should be extended by $\kappa$ to a function $Q$ such that
\eqref{eq:LimitCE.3} holds and that the integral in \eqref{eq:mfD0.3B} is
minimized.

We emphasize that $\wh\mfD_0$ \AAA completely ignores \EEE the values of $u$ in
the membrane part $\Omega_0=[0,T]\ti {]0,1[}$. This will be essential in the
proofs for the liminf and the limsup estimate.

The membrane functional $\ol M:\R_\geq^2\ti\R \to \R_\geq $ is defined via
\begin{subequations}
\label{eq:RMJ}
\begin{align}
\label{eq:RMJ.a}
& \ol M(u_\PM,\kappa):= 
 \inf\bigset{ \bfM \big(u,\kappa \bm 1_{I_0}\big) }{ u \in \bfP(u_\PM)  },
\\
\label{eq:RMJ.c}
&\text{with } 
 \bfM(u,Q) :=\int_0^1  \!\!\Big(  \frac{Q(y)^2}{2 \mob(u(y))} 
  + \frac12 \big| \pl_y H(u(y))\big|^2 \Big) \dd y 
\\
\label{eq:RMJ.d}
&\text{and } \bfP(u_\PM ) := \bigset{u \in \rmW^{1,1}([0,1])}{\ u\geq 0, \
  u(0)=u_-,\ u(1)=u_+ }. 
\end{align}
\end{subequations}
Here $\ol M$ may attain the value $+\infty$ if $\bfM(u,\kappa) = \infty$ for
all competitors. However, for $\kappa=0$ we always have
$\ol M(U_\PM,0)= \AAA \frac12 \EEE 
 \big( H(U_+){-}H(U_-)\big)^2<\infty$, such that $\bfR_\mb$ in
\eqref{eq:I.def.Rmemb} is well-defined. We note that $\bfM$ is defined for
general functions $Q\in \rmL^1(0,1)$ for later purposes, but for the definition
of $\ol M$  it is evaluated only for constant functions $Q=\kappa \bm1_{I_0}$.

We see that $\wh\mfD^\mb _0$ depends via $\bfR_\mb$ on $u|_{\Omega_0}$ in such
a way that the minimization of the dissipation functional $\bfM$ is local in
the time $t\in [0,T]$ for a given constant flux $\kappa(t)$ through the
membrane and given relative densities $u_-(t)$ and $u_+(t)$. This manifests
a separation of time scales: the convergence to a spatially constant flux
inside the membrane and to the density profile $u(y)=u_\mafo{NESS}(y)$ 
is much faster than the diffusive processes in the bulk.

We are now ready to formulate our main result on the scaled domain
$\Omega$. For the liminf estimate we can use the weaker \AAA concavity \EEE
assumption \eqref{eq:Cvx.Ass.psi}, whereas for the limsup estimate involving
the construction of a recovery sequence we need the \AAA stronger concavity
\EEE assumption \eqref{eq:ConcaveTwofold}.  However, we expect that the usage
of the stronger assumption is technical and conjecture that the weaker
assumption is sufficient.

\begin{theorem}[$\wh\mfD_\eps \overset{\Gamma_\rmE}{\weak} \wh\mfD_0$] 
\label{th:Main.Gamma.Cvg}
Let $\wh\mfD_\eps$ and $\wh\mfD_0$ be defined as in \eqref{eq:def.mfFeps} and
\eqref{eq:mfD0limit}, respectively, and assume that \eqref{eq:AllAssump} holds.   

(A) If additionally \eqref{eq:Cvx.Ass.psi} holds, then for every
sequence $(u_\eps)_{\eps\in {]0,1[}} $ satisfying the energy bound
\eqref{eq:MainAss.E} and the convergence $X'_\eps u_\eps  \weaks \mu $ in
$\mfM(\ol\Omega)$, we have
\begin{equation}
    \label{eq:MainLiminf}
  \mu= u_0\LEB^2 \text{ with } u_0 \in \bfL_1 \qand
     \liminf_{\eps\to 0} \wh\mfD_\eps(u_\eps) \geq \wh\mfD_0(u_0). 
\end{equation}

(B) If additionally \eqref{eq:ConcaveTwofold} holds, then for all $\wh u_0 \in
\bfL_1$ with $\sup_{t\in [0,T]} \calE_0(\wh u_0(t,\cdot)) < \infty$ and
$\wh\mfD_0(\wh u_0)<\infty$, there exists a (recovery) sequence
$(\wh u_\eps)_{\eps \in {]0,1[}}$ satisfying $ X'_\eps \wh u_\eps \weak 
X'_0 \wh u_0$ in $\rmL^1(\Omega)$, the energy bound \eqref{eq:MainAss.E}, and the
dissipation convergence $\lim_{\eps\to 0} \wh\mfD_\eps(\wh u_\eps) =
\wh\mfD_0(\wh u_0)$.
\end{theorem}

The proof of this result will be the content of Sections \ref{se:APrioriEst} to
\ref{se:Limsup}. In particular, the liminf estimate in part A follows from 
Propositions \ref{pr:liminf.bulk} and \ref{pr:liminf.memb}, see the end of
Section \ref{se:Liminf}. The limsup estimate is established in Section
\ref{se:Limsup} based on the stronger assumption \eqref{eq:ConcaveTwofold}.

\subsection{Proof of the Main Result in Theorem \ref{th:MainResUnscaled} } 
\label{su:ProofMainRes}

We now return to the original unscaled gradient systems
$(\bfQ, \calE, \calR^*_\eps)$ generating the
thin-layer porous-medium equation \eqref{eq:PME1}. The energy
$\calE_\eps$ and the dual dissipation potentials
$\calR^*_\eps$ are given in \eqref{eq:I.calE.calR} with $A=A_\eps$. 
The corresponding dissipation functional
$\mfD_\eps$ defined in the usual form from  $\calE$ and
$\calR^*_\eps$ is given in \eqref{eq:Def.mfDeps}.

We now deduce the main result in Theorem \ref{th:MainResUnscaled} from the 
scaled result in Theorem \ref{th:Main.Gamma.Cvg}. For this we recall the
fundamental transformations 
\[
u_\eps (t,y)= v_\eps(t,X_\eps(y)) \qand v_\eps(t,x) = u_\eps (t , X_\eps(y)),
\]
here it is important that $v_\eps(t,\cdot) \in \calP(\DDD)$ and $X'_\eps u_\eps
\in \calP(I)$.

The main connection between the two results stems from the weak convergence
properties. We recall that by the criterion of de la Vall\'ee Poussin a
sequence $\big(w_k\big)_{k\in \N}$ is weakly pre-compact in
$\rmL^1(\Sigma,\mu)$ if and only if there exists a superlinear function
$\Phi:\R_\geq \to \R_\geq$ such that $\sup_{k\in \N}\int_\Sigma \Phi(|w_k|) \dd
\mu \leq C < \infty$. In particular, the weak convergence $w_k \weak w$ in
$\rmL^1(\Sigma,\mu)$ implies 
\[
\sup_{k \in \N} \int_{\Sigma'} |w_k|\dd \mu \ \longrightarrow \ 0 \quad
\text{for } \mu(\Sigma') \to 0. 
\] 

In our situation we use that $r \mapsto E(r)$ is
superlinear by \eqref{eq:E.bounds} and that we always assume $\sup_{[0,T]}
\calE(v_\eps(t)) \leq \sfCE <\infty$, which implies $\int_0^T\!\!\int_\DDD
E(v_\eps(t,x)) \dd x \dd t\leq T   \sfCE <\infty$. We easily obtain the
following equivalence $u_\eps= v_\eps {\circ}X_\eps$, here assumed for
$\eps\in [0,1]$.
\begin{align}
\label{eq:WeakImpliesWeak}
& v_\eps \weak  v_0 \ \text{ in } \rmL^1([0,T]\ti \DDD) \quad \Longleftrightarrow
\quad 
X'_\eps\,u_\eps \weak X'_0 \,u_0  \text{ in } \rmL^1([0,T] \ti I).
\end{align}
The main point here is that there cannot be any concentration of $(v_\eps)$
near the membrane $[0,T]\ti \{0\}$, because 
\[
\sup_{\eps\in {]0,1]}} \int_0^T\!\!\int_{[-\delta,\delta]} v_\eps(t,x)  \dd x
  \dd t \ \longrightarrow \ 0 \quad \text{for } \delta \searrow 0. 
\]
\medskip

\noindent
\begin{proof}[Proof of Theorem \ref{th:MainResUnscaled}]
We first deduce the liminf estimate for $\mfD_\eps$ from that of
$\wh\mfD_\eps$.

First, we observe that for all functions $u\in \bfL_1$ we can define $v(t,x)=
u(t,Y_0(x))$ for all $x \in \DDD\setminus\{0\}$, because $Y_0$ is discontinuous
at $x=0$. Nevertheless the limits $v(t,0_\pm)$ exists and are equal to $u(t,0)$ and
$u(t,1)$, respectively. Hence, comparing the definitions of $\mfD_0$ and
$\wh\mfD_0$ we obtain
\[
\mfD_0( u {\circ} Y_0) = \wh\mfD_0( u) \quad \text{for all } u \in \bfL_1.
\]

Secondly, we consider a sequence $v_\eps \weak v$ in $\rmL^1(\Omega)$ with
$\sup_{t\in [0,T]} \calE(v_\eps(t)) \leq \sfCE<\infty $ and $\mfD_\eps( v_\eps)
\leq \sfC_\mfD < \infty$. Because of $X'_\eps u_\eps(t) \in \calP(I)$ we can
choose a subsequence such that $X'_\eps u_\eps \weaks \mu $ in $\mfM(\Omega)$.
By \eqref{eq:WeakImpliesWeak} we also know that 
$u_\eps=v_\eps {\circ} X_\eps$ satisfies $X'_\eps u_\eps \weak X'_0 u_0$. 

Thirdly, we observe that $\wh\calE_\eps(u_\eps(t))=\calE(v_\eps(t)) \leq \sfCE$
and $\wh\mfD_\eps( u_\eps)=\mfD_\eps( v_\eps) \leq \sfC_\mfD $. Thus, the
desired liminf inequality follows from \eqref{eq:MainLiminf} in Part A of
Theorem \ref{th:Main.Gamma.Cvg}: 
\[
\liminf_{\eps\to 0} \mfD_\eps(v_\eps)= \liminf_{\eps\to 0}
\wh\mfD_\eps(u_\eps) \geq \wh\mfD_0(u) = \mfD_0(v).
\]

For the limsup estimate we use Part B of Theorem
\ref{th:Main.Gamma.Cvg}. Given $\wt v_0$ we define $\wh u_0= \wh v_0{\circ}
Y_0$ for $(t,y) \not\in \Omega_0$, and for $(t,y)\in \Omega_0$ we fill in the
values of $\wh u_0$ via $\wh u_0(t,y) = u_\mafo{min}{v_\PM(t),\kappa(t)}(y)$,
where $u_\mafo{min}^{u_\PM,\kappa}$ is the unique minimizer of $\bfM(\cdot,\kappa)$ in
$\bfP(u_\PM)$. With this we obtain $\wh u_0\in \bfL_1$ as well as 
$\wh \calE_0(\wh u_0(t))=\calE(\wt v_0(t)) \leq \sfCE$ and $\wh\mfD_0( 
\wh u_0)=\mfD_0(\wt v_0) < \infty$. Thus, Part B of Theorem
\ref{th:Main.Gamma.Cvg} provides the recovery sequence $\wh u_\eps$ such that 
$X'_\eps \wh u_\eps \weak X'_0 \wh u_0$ in $\rmL^1(\Omega)$ and $\wh\mfD_\eps( 
\wh u_\eps) \to \wh \mfD_0(\wh u_0)$. Setting $\wt v_\eps = \wh u_\eps{\circ}
X_\eps$ for $\eps \in [0,1]$, relation \eqref{eq:WeakImpliesWeak} gives 
$ \wt v_\eps \to \wt v_0$. Moreover, we have  
\[
\mfD_\eps(\wt v_\eps) = \wh\mfD_\eps(\wt u_\eps) \to \wh\mfD_0(\wt u_0) 
= \mfD_0(\wt v_0),
\]
which is the desired limsup estimate. 
\end{proof}

\RRR 

Finally, we want to show that EDP-convergence automatically implies that 
accumulation points $u_0$ of a family $\big(v_\eps)_{\eps\in (0,1)}$ of EDB
solution for the gradient systems
$\big(\calP(\DDD),\calE,\calR_\eps\big)$ are indeed EDB solutions for the
limiting gradient system $\big(\calP(\DDD),\calE,\calR_\eff\big)$. Here, a
curve $v_\eps:[0,T] \to \calP(\DDD)$ is called an EDB solution if the 
function $[0,T]\ni t \mapsto \calE(v_\eps(t)) \in \R$ is absolutely continuous and
satisfies 
\[
\calE(v_\eps(t)) + \int_r^t \!\! \Big(\calR_\eps(v_\eps,\dot v_\eps) +
\calR^*_\eps(v_\eps,{-}E'(v_\eps) \big) \Big) \dd s = \calE(v_\eps(r)) \quad 
\text{for all }r,t \text{ with } 0 \leq r< t\leq T. 
\]  
The following result is established under the additional assumption that the
limiting system $\big(\calP(\DDD),\calE,\calR_\eff\big)$ 
satisfies chain-rule inequality in the following sense. 
\begin{align}
\nonumber
&\text{If }v:[0,T]\to \calP(\calD) \text{ satisfies } 
\\
\label{eq:ChainRule}
&\sup\nolimits_{t\in [0,T]} \calE(v(t)) \ < \infty \quad \text{and} \quad
\mfD_0(v) \ < \infty,
\\
\nonumber
&\text{then, we have }  \calE(v(T)) +  \mfD_0(v) \geq \calE(v(0)). 
\end{align}
We expect that this chain rule can be established under our assumptions because
we can exploit the convexity of $\calE$ and $\mfD_0$ for temporal smoothing,
see \cite{MiScSt25?DFOD, HeMiSt25?DCLN, MaMiZi26?NACR} similar approaches.   

\begin{corollary}[Convergence of solutions]
\label{co:CvgSols}
Let the assumptions \eqref{eq:AllAssump} and \eqref{eq:Cvx.Ass.psi} as well as
the chain-rule inequality \eqref{eq:ChainRule} hold. Then, for every family
$\big(v_\eps)_{\eps\in (0,1)}$  of EDB solutions for 
$\big(\calP(\DDD),\calE,\calR_\eps\big)_{\eps\in (0,1)}$ with well-prepared
initial conditions, viz.\ 
\[
v_\eps(0) \to v^0 \text{ in }\rmL^1(\DDD) \quad \text{and} \quad
\calE(v_\eps(0)) \to \calE(v^0) < \infty.
\]
then there exists a subsequence (not relabeled) and a limit curve $v_0 :[0,T]
\to \calP(\DDD)$ such that (i) $v_0(0)=v^0$, (ii) we have the convergence 
\[
v_\eps(t) \to v_0(t) \text{ in } \rmL^1([0,T]) \quad \text{and} \quad
\calE(v_\eps(t)) \to \calE(v_0(t)) \qquad \text{for all } t\in [0,T],
\]
and (iii) $v_0$ is an EDB solution for the
effective gradient system $\big(\calP(\DDD),\calE,\calR_\eff\big)$ . 
\end{corollary} 
\begin{proof}
We base our arguments on the a priori estimates for
$u_\eps(t,y)=v_\eps(t,X_\eps(y))$ derived in Proposition \ref{pr:SpaTimeComp}.  
From $\wh\calE_\eps(u_\eps(t))=\calE(v_\eps(t))\leq \calE(v_\eps(0)) \leq
\calE(v^0) {+} 1< \infty $ we see that \eqref{eq:MainAss.E} holds. Moreover, we have
\begin{align*}
\wh\mfD_\eps(u_\eps)= \mfD_\eps (v_\eps)=\calE(v_\eps(0))-\calE(v_\eps(T)) \leq
\calE(v_\eps(0)) \leq \calE(v^0) +1< \infty,
\end{align*}
such that \eqref{eq:MainAss.P} and \eqref{eq:MainAss.D} hold as well. Moreover,
the continuity equation \eqref{eq:rho.eps.CE} for the pairs $(u_\eps,Q_\eps)$
holds by definition.  Proposition \ref{pr:SpaTimeComp} implies the a priori
estimates
\[
\big\| u_\eps \big\|_{\rmL^{2\qsubE+2\qsubH}(\Omega_\PM)} + 
  \big\| Q_\eps\big\|_{\rmL^{r_*} (\Omega_\PM)} \leq C_\rma \quad \text{with }
  r_*=\tfrac{2 \qsubE+2\qsubH}{\qsubE+2\qsubH}>1 . 
\]
Exploiting the continuity equation we obtain that $u_\eps|_{I_\PM}$ is
equi-continuous, namely 
\[
 \| u_\eps(t)-u_\eps(r)\|_{\rmW^{-1,r_*}(I_\PM)} \leq \int_r^t \|
 Q_\eps(s)\|_{\rmL^{r_*}(I_\PM)} \dd s \leq C_\rma (t{-}r)^{1-1/r_*} \quad \text{for }
   0\leq r<t\leq T.
\]
Thus, the Arzel\`a-Ascoli theorem provides a limit function $\ol u_0:[0,T] \ti
\rmW^{-1,r_*}(I_\PM) $ and a subsequence (not relabeled) with
$u_\eps(t)|_{I_\PM} \weak \ol u_0(t)$ in $\rmW^{-1,r_*}(I_\PM) $ for all $t\in
[0,T]$. With $\wh\calE_\eps(u_\eps(t)) \leq \calE(v^0) {+} 1< \infty $
and superlinearity of $E$ 
we obtain $ u_\eps(t)|_{I_\PM} \weak \ol u_0(t)$ in $\rmL^1(I_\PM)$. 
In particular, we have (i) $v_0(0)=v^0$. 

By the superlinearity of $E$ we also know that $v_\eps(t)$ cannot concentrate
near $0$ and we conclude $ v_\eps(t) \weak v_0(t)$ in $\rmL^1(\DDD)$, where
$v_0(t,x):= \ol u_0\big(t,Y_0^{-1}(x)\big)$.  To show that $v_0$ is an EDB
solution, we can start from $\calE(v_\eps(T))+ \mfD_\eps(v_\eps) = \calE(v_\eps(0))$
because $v_\eps$ is an EDB solution. We can pass to the limit $\eps\to 0$ along
the chosen subsequence because $v_\eps \weak v_0$ in $\rmL^1([0,T]\ti \DDD)$,
and Theorem \ref{th:MainResUnscaled} gives $\calE(v_0(T))+ \mfD_0(v_0) \leq
\calE(v_\eps(0))$.

As the assumed chain rule \eqref{eq:ChainRule} implies the opposite
  inequality we conclude $\calE(v_0(T))=\lim_{\eps\to 0} \calE(v_\eps(T))$ and 
superlinearity and strict convexity of $ E$ imply the strong convergence
$v_\eps(T) \to v_0(T)$ in $\rmL^1(\Omega)$. 

We can repeat the above argument on the intervals $[0,r]$ for all $r\in
{]0,T[}$ and thus obtain the convergences (ii). This also implies 
$\calE(v_0(r))+ \int_0^r (\calR_\eff+\calR^*_\eff)\dd s=\calE(v_0(0))$ for all
$r\in [0,T]$. Subtracting this relation with the choice $r=t$ from the relation
for $r$, shows that $v_0$ is indeed an EDB solution, i.e.\ (iii) is established
as well.  
\end{proof}

\subsection{Potential  generalizations}

The following remarks highlights that the theory developed in this paper is simple
enough to allow for several generalizations (see e.g.\ \cite{LMPR17MOGG,
  Fren19DEGS, Miel23NESS, PelSch23CGST}). He will only indicate possible
directions that seem to be promising and leave the details for future
research.\medskip

First, we may
allow for a suitable dependence on $x$ for the diffusion factor $A_\eps$, the
energy $E$, and the mobility. This would be best formulated in the stretched
(microscopic) variable $y=Y_\eps(x)$, e.g.\ in the form 
\begin{align*}
&A_\eps(x) = \wt  A_\pm(Y_\eps(x)) \text{ in }I_\PM, \quad A_\eps(x)= \eps
   \wt b(x/\eps) \text{ in }I_0, \\
& E_\eps(x,u)= \wt E(Y_\eps(x),u) ,
\quad \mob_\eps(x,u)=\wt\mob(Y_\eps(x),u),
\end{align*}
where the functions $\wt A_\pm$, $\wt b$, $\wt E$, and $\wt\mob$ are piecewise
continuous in $y\in I$ and satisfy the assumption of Section \ref{su:Ass.GS}
uniformly for $y \in I$. The essential arguments of the proofs remain the
same, because we do not 
need to differentiate in $y \in I$. Of course, the cell problem \eqref{eq:RMJ}
will now involve the functional 
\[
\wt\bfM(u,Q)= \int_0^1 \Big( \frac{Q(y)^2}{2\,\wt\mob(y, u(y))} +
\frac{\wt\mob(y,u(y))}{2} \big| \pl_y\big(\wt E(y,u(y))\big)\big|^2 \big) \dd y.
\]
We refer to \cite[Sec.\,4]{LMPR17MOGG} and \cite[Sec.\,4]{Fren19DEGS} more
details.\medskip

Secondly, it seems possible to include an absorption/desorption reaction to the
diffusion equation in the form 
\begin{align*}
\pl_t v &= \DIV\!\big( A_\eps\mob(v) \nabla E'(v)\big) {+} C_\eps\pl_\xi
\Psi_\mafo{react}^* \big(v, {-}E'(v)\big) E'(v)\\
& =- \Big(\bbK_\eps(v){+} \pl_\xi\bfPsi^*_\mafo{react}(v, \,\cdot\,)\Big) E'(v),
\end{align*}
where $\bbK$ is the Onsager operator for diffusion as before, whereas the dual
dissipation potential
$\bfPsi^*_\mafo{react}(v,\cdot): \xi \mapsto \int_I C_\eps(x)
\Psi^*_\mafo{react}(v,\xi) \dd y$ for the reaction may be nonquadratic and the
space-dependent prefactor $C_\eps$ scales as $A_\eps$. We refer to
\cite[Thm.\,V.2]{Miel23NESS} for a formal calculation showing that, for a linear
reaction-diffusion equation with Boltzmann entropy ($p=1$) and cosh-type
gradient structure for the reaction, the effective gradient structure for the
membrane limit can be obtained explicitly. However, a rigorous analysis would
be more difficult than the one in this paper, for instance the continuity
equation and the primal dissipation potential have to include a growth term $G$
(see e.g.\ \cite{HeMiSt25?DCLN} for EDP-convergence in a similar case):
\[
  X'_\eps(y) \pl_t u + \pl_y Q = X'_\eps(y)G \quad \text{and} \quad
  \wh\mfD_\eps^\mafo{flux} (u) = \int_0^T\!\!\int_I \Big( \frac{Q^2}{2 \ \bbA_\eps
    \mob(u) } + \bbC_\eps \Psi\big( G/ \bbC_\eps\big) \Big)\dd y \dd t.\medskip
\]

Finally, it can be expected that also higher-dimensional settings can be
treated where the membrane is a smooth hypersurface. After some localization it
will be possible to restrict to a case where the domain is a cylinder
$[-1,1] \ti \Sigma$ with a bounded cross-section $\Sigma \subset \R^{d-1}$
having Lipschitz boundary. By stretching $ x\in [-1,1]$ as before to
$y = Y_\eps(x) \in I =[-1,2]$ we obtain the scaled variables
$(y,z)\in I\ti \Sigma=:\Omega_\mafo{sca}$. The unscaled continuity equation
$\pl_t v  + \pl_x J_1 + \nabla_z (J_2,...,J_d)=0$ transforms as follows:
\begin{align*}
&X'_\eps(y) \pl_t u + \pl_y Q_1 + X'_\eps(y) \nabla_z(Q_2,...,Q_d) =0 
\\
&
\text{with }u(t,y,z)=v(t,X_\eps(y),z) \text{ and } Q(t,y,z)=J(t,X_\eps(y),z).
\end{align*}
Using $ \rmd x = X'_\eps\rmd y$ and $A_\eps = X'_\eps \bbA_\eps$, the scaled
dissipation functional takes the form
\[
\wh\mfD_\eps(u) = \int_0^T\!\!\! \int_\Sigma \!\int_I \Big(  \frac{|Q|^2}{2 \bbA_\eps
    \mob(u) } + \frac{\bbA_\eps\mob(u)}2  \big( (\pl_y E'(u))^2 {+} (X'_\eps)^2
    \big| \nabla_z E'(u)\big|^2 \big) \Big) \dd y \dd z \dd t.
\]
In the formal limit $\eps\to 0$, we again see a scale separation inside of the
membrane region $[0,T]\ti \Sigma \ti I_0$: On the one hand we obtain
$Q_1(t,y,z)= \kappa (t,z)$ for $y \in I_0$ because the continuity equation
reduces to $\pl_y Q_1=0$, i.e.\ we have constant flux in the scaled normal
direction. On the other hand, because of $X'_0=0$ in $I_0$ the functional
$\wh\mfD_0$  has neither couplings in time nor in $z \in \Sigma$, hence we can
look for minimizers $u(t,\cdot,z) \in \rmC^0(I_0)$ for each $(t,z)\in [0,T]\ti
\Sigma$.  Thus, the same cell problem via the functional $\bfM$ appears as in
the one-dimensional case. 
 
However, it remains a challenge to fill in all the details for a rigorous
analysis of multi-dimensional case.    
\EEE

\section{A priori estimates}
\label{se:APrioriEst}

We now work in the scaled variables $(t,y)\in \Omega= [0,T]\ti I$.  

\subsection{Energy and dissipation assumptions on $u_\eps$}
\label{su:E.D.Ass.mu}
 
To structure the a priori estimates used for the EDP-convergence we
enumerate the standard a priori bounds in terms of the energy and the
two parts of dissipation functional as follows. Throughout
we consider an $\eps$-dependent family $u_\eps:[0,T]\to \rmL^1(I)$,
which satisfies the three bounds for all
$\eps\in {]0,\eps_0]}$:
\begin{subequations}
 \label{eq:MainAss.all} 
\begin{align}
\label{eq:MainAss.E}
  & \!\!
   \sup_{t\in [0,T]} \wh\calE_\eps(u_\eps(t)) \leq \sfCE,
\\
\label{eq:MainAss.P} & 
 \wh\mfD^\text{flux}_\eps(u_\eps):=
 \int_0^T\!\! \int_I  \frac{Q_\eps^2}{2\bbA_\eps\mob(u_\eps)}  
   \dd y \dd t  \leq \sfCF <\infty,
\\
\label{eq:MainAss.D} 
 & \wh\mfD^\text{slope}_\eps( u_\eps ):=
 \int_0^T\!\! \int_I  \frac{\bbA_\eps}{2} \big| \pl_y H(u_\eps)\big|^2 \dd y \dd t 
 \leq \sfCS < \infty,
\\
  \label{eq:rho.eps.CE}
&\text{where } \quad   X'_\eps\,\pl_t u_\eps  + \pl_y Q_\eps=0 . 
\end{align}
\end{subequations}
We recall that $Q_\eps$ in \eqref{eq:MainAss.P} is always assumed to be the
minimizer in $ \wh\mfD_\eps^\text{flux}(u_\eps)$ subject to
\eqref{eq:rho.eps.CE}, see Lemma \ref{le:DifferentFlux} for details. 
Clearly, definition \eqref{eq:Def.mfDeps} implies $\wh\mfD_\eps(u)=
 \wh\mfD_\eps^\text{flux}(u) + \wh\mfD_\eps^\text{slope}(u)$, and we will be able to
derive a priori estimates by looking at the two parts independently,
in particular, \eqref{eq:MainAss.E} and \eqref{eq:MainAss.D} provide
spatial compactness for $u_\eps$, while
\eqref{eq:MainAss.E} and \eqref{eq:MainAss.P} provide temporal
compactness for $X'_\eps u_\eps$.

However, for deriving the $\Gamma$-limit of $\wh\mfD_\eps$ we need to
look at the sum, as taking the $\Gamma$-limits of the two parts
separately, we would obtain two limits $\wh\mfD_0^\text{flux} $ and
$\wh\mfD_0^\text{slope}$ such that the sum $\wh\mfD_0^\text{flux} +
\wh\mfD_0^\text{slope}$ would be lower than our $\wh\mfD_0$. Nevertheless,
$\wh\mfD_0$ will be again a sum of a flux term
$\wh\mfD_\eff^\text{flux}$ and a slope term $\wh\mfD_\eff^\text{slope}$, but
they are not the $\Gamma$-limits of $ \wh\mfD^\text{flux}_\eps$ and
$\wh\mfD^\text{slope}_\eps $, respectively.

Subsequently, we will use the following shorthand notation for stating that a
sequence is bounded in a function space $\bbX$:
\[
f_\eps \bddin \bbX \quad \overset{\text{def}}{\Longleftrightarrow} \quad 
\exists\, C, \eps_0>0\ \forall\, \eps \in {]0,\eps_0[}: \ 
\| f_\eps\|_\bbX \leq C.
\]
For example, \eqref{eq:MainAss.all} implies $Q_\eps^2/\mob(u_\eps) \bddin
\rmL^1(\Omega)$ and $\pl_y H(u_\eps) \bddin \rmL^2(\Omega)$. This is especially
useful in intermediate calculations, if we are not interested in the specific
constant.

\subsection{A priori estimates for \TOS{$u_\eps$}{the density}
  based on \TOS{$\wh\mfD_\eps$}{the dissipation functional}} 
\label{su:Apriori}

The first result provides a simple a priori estimate of $u_\eps$ in terms of
the energy. Combining \eqref{eq:E.bounds} and \eqref{eq:MainAss.E} we find
\begin{equation}
  \label{eq:ueps.Linfty.qE}
  \big\| u_\eps \big\|_{\rmL^\infty\left([0,T],\rmL^{\qsubE}(I_\PM)\right)}^\qsubE
  \leq \big(c_\rmE \sfCE {+} 1 \big)/ c_\rmE^2 . 
\end{equation}
Note that we only have a good estimate on the bulk parts $I_\PM$ because of the
factor $X'_\eps$ in the definition of $\wh\calE_\eps$. 
The following bound will be valid on all of $I$, because the slope estimate
\eqref{eq:MainAss.D} does not degenerate on $I_0$.  We use 
\eqref{eq:ueps.Linfty.qE} on $I_\PM$ and the growth bound \eqref{eq:Cond.H.qH}
for $H$.

\begin{proposition} 
\label{pr:H.u.eps} 
Let $ E $ and $H$ satisfy \eqref{eq:E.bounds} and \eqref{eq:Cond.H.qH}, 
and let $(u_\eps)_\eps$ satisfy \eqref{eq:MainAss.all}. Then there exist
constants $\sfC_H$ and $\wt\sfC_H$ such that 
\begin{equation}
  \label{eq:H.ueps.bound}
  \int_0^T \big\| H(u_\eps(t,\cdot)) \big\|_{\rmH^1(I)}^2 \dd t \leq \sfC_H
  \qand  \iint_\Omega u_\eps(t,y)^{2\qsubH} \dd y \dd t \leq \wt\sfC_H. 
\end{equation}
\end{proposition}
\begin{proof} To simplify notation we write $u$ for $u_\eps$ and 
$z(t,y)=H(u(t,y))$. 

By \eqref{eq:MainAss.D} the $\rmL^2(\Omega)$ norm of $\pl_y z$ is bounded by
$(2\sfCS)^{1/2}$, so we only need to estimate $z$ in \AAA the \EEE
$\rmL^2(\Omega)$ norm.  For a.a.\ $t\in [0,T]$ we have
$\int_I \pl_y z(t,y)^2\dd y <\infty$, hence $z(t,\cdot)$ is continuous with
\begin{equation}
  \label{eq:z.Hoelder}
    |z(t,x)-z(t,y)| \leq |x{-}y|^{1/2} \|\pl_y z(t)\|_{\rmL^2(I)}.
\end{equation}

Moreover, defining $g(z)= \big(H^{-1}(z)\big)^{\qsubE}$ and using
$u(t,y)=H^{-1}(z(t,y))$ we have 
\[
\int_{I_\PM} g(z(t,y))\dd y = \int_{I_\PM} u(t,y)^{\qsubE} \dd y \leq
\big(c_\rmE \sfCE {+}1\big) /c_\rmE^2 =: C_2,
\]
where we used \eqref{eq:ueps.Linfty.qE}. Hence, for each $t$ there exists $y_t$
such that $g(z(t,y_t))\leq C_2$, or equivalently
$z(t,y_t)\leq H\big( C_2^{1/\qsubE}\big)=:C_3$.  Together with
\eqref{eq:z.Hoelder} we conclude
$ z(t,y) \leq C_3 + \sqrt3 \|\pl_y z(t)\|_{\rmL^2(I)}$, which implies
$\| z\|^2_{\rmL^2(\Omega)} \leq 6 C_3T + 6 \|\pl_y
z\|^2_{\rmL^2(\Omega)}$. Hence, the first estimate in \eqref{eq:H.ueps.bound}
is established with $\sfC_H= 6C_3T+ 14 \sfCS$.

Observing that \eqref{eq:Cond.H.qH} implies $u^\qsubH \leq
C(1{+}H(u))$ the desired second estimate follows immediately from the first one.  
\end{proof}

The importance of the above result is that we are able to control $u_\eps$
through $z_\eps=H(u_\eps)$ also in the layer domain $\Omega_0=[0,T]\ti I_0$. As
we will see now, it is much simpler to obtain $\rmL^q$ bounds for $u_\eps$ over
the bulk domain $\Omega_\PM = [0,T]\ti I_\PM$.
 
Next we show that the a priori estimates in the bulk domain
$\Omega_\PM$ can be improved by interpolating the last bound and the
uniform energy bound \eqref{eq:MainAss.E} with $\sfCE$.

\begin{proposition}[Improved bounds in the bulk] 
\label{pr:ImprovBound} 
Let $ E $ and $H$ satisfy \eqref{eq:E.bounds} and \eqref{eq:Cond.H.qH}, and
let $(u_\eps)_\eps$ satisfy \eqref{eq:MainAss.all}. Then there exist a constant 
$\sfC_\mathrm{bulk}$ such that 
\begin{equation}
  \label{eq:L.rstar}
  \int_{\Omega_\PM}   u_\eps(t,y)^{ 2\qsubE {+} 2\qsubH } \dd y \dd t \leq
\sfC_\mathrm{bulk}  .
\end{equation}
\end{proposition}
\begin{proof}  
Combining \eqref{eq:E.bounds} and \eqref{eq:MainAss.E} we find 
$u_\eps \bddin \rmL^\infty \big( [0,T] ; \rmL^{\qsubE }(I_\PM) \big)$. 

For $\qsubE\geq \qsubH$ we have $z_\eps=H(u_\eps)\leq C(1+u^\qsubH) $ by
\eqref{eq:Cond.H.qH}. Hence, $z_\eps \bddin \rmL^2\big([0,T];
\rmH^1(I_\PM)\big)$ and $z_\eps \bddin
\rmL^\infty\big([0,T];\rmL^{ \qsubE/\qsubH }(I_\PM)\big) $. Using the
\SSS Gagliardo-Nirenberg \EEE estimate   
and space-time interpolation (cf.\ \cite[Lem.\,4.1]{MieNau22EGTW}) we see that
$z_\eps \bddin \rmL^{q_*}(\Omega_\PM)$ with
$q_*=2 {+} 2\qsubE/\qsubH$.  Exploiting \eqref{eq:Cond.H.qH} once again, we
have established \eqref{eq:L.rstar} for the case $\qsubE\geq \qsubH$.

For $\gamma:= \qsubE/\qsubH <1$ we argue as follows. For $q\geq 1$ we have
$\| z\|_{\rmL^q} \leq \|z^\gamma\|_{\rmL^1}^{1/q}
\|z\|_{\rmL^\infty}^{1-\gamma/q}$, which is easily established by substituting
$z=w^{1/\gamma}$. Inserting the \AAA Gagliardo-Nirenberg interpolation \EEE 
$\|z\|_{\rmL^\infty} \leq C\| z\|_{\rmL^q}^\theta \|z\|_{\rmH^1}^{1-\theta}$
with optimal $\theta$, we find
\[
\|z\|_{\rmL^q}^q \leq C \| z^\gamma\|_{\rmL^1}^{(q+2)/(2+\gamma)} 
\| z\|_{\rmH^1}^{2(q-\gamma)/(2+\gamma)}.
\]
Choosing $q=2+2\gamma>2$ we obtain the estimate 
\[
\iint_{\Omega_\PM} z^{2+2\gamma} \dd y \dd t \leq C\int_0^T
\|z^\gamma\|_{\rmL^1}^2 \| z\|_{\rmH^1}^2 \dd t = C \sup_{[0,T]} \|
z(t)^\gamma\|_{\rmL^1}^2 \sfC_H. 
\]
Again using \eqref{eq:E.bounds} and \eqref{eq:Cond.H.qH} 
we have $ \int_{I_\PM} z_\eps(t)^\gamma \dd y \leq C(1{+}\sfCE)$. Now the estimate
for $u_\eps$ follows by using \eqref{eq:Cond.H.qH} again. 
\end{proof}

\subsection{Spatial and temporal compactness}
\label{su:SpaTemCompact}

To derive the liminf estimates will pass to the limit in cases where
$\mfD_\eps$ is non-convex. For this we need sufficient information on the 
strong convergence of $u_\eps$. To obtain this we will rely on a nonlinear
version of the classical Aubin-Lions-Simon lemma. We refer to
\cite[Thm.\,1]{Mous16VCAL} for a version fitting well to the PDE context. Here
we will rely on the theory developed in \cite{RosSav03TIEC} that is well suited
for problems in the calculus of variations, because it relies on scalar
functionals having compact sublevels. 

\begin{lemma}
\label{le:CompSublevels}
Let $ E $ and $H$ satisfy \eqref{eq:E.bounds} and
  \eqref{eq:Cond.H.qH} and define $\calF:\rmL^1_\geq (I)\to [0,\infty]$ via
\[
\calF(u)= \int_{I} \Big( H(u(y))^2 + \big| \pl_y
H(u(y))\big|^2 \Big) \dd y.  
\]
Then $\calF$ has compact sublevels in $\rmC^0(I)$. 
\end{lemma}
\begin{proof}
We fix $S>0$ and consider $U_S:=\bigset{u \in \rmL^1_\geq (I)}{ \calF(u)\leq S}$.

Using the continuous embedding $\rmH^1 \subset  \rmC^{1/2}$ there exists a
constant $C_S>0$ such that 
\[
\forall \, u\in U_S\ \forall \, y,\wt y \in I: \quad  H(u(y))\leq C_S, \quad  
\big|H(u(y))-H(u(\wt y))\big| \leq \sqrt{|y{-}\wt y|}\, C_S.
\]
As $H$ is continuous with $H'(r)=\sqrt{\mob(r)}E''(r)>0$, the inverse
$H^{-1}:[0,C_S] \to [0,\ol C_S] $ with $\ol C_S=H^{-1}(C_S)$ is uniformly
continuous. We denote the modulus of continuity by $\ol\omega$ and find 
\[
\forall \, u\in U_S\ \forall \, y,\wt y \in I: \quad  u(y)\in [0, \ol C_S], \quad  
\big|u(y){-}u(\wt y)\big| \leq \ol \omega\big( \sqrt{|y{-}\wt y|}\, C_S\big).
\]
Thus, compactness follows from the Arzel\`a-Ascoli criterion. 
\end{proof}

For the temporal compactness we restrict to the bulk $\Omega_\PM$, because the
continuity equation contains the term $X'_\eps\,\pl_t u_\eps$ with
$X'_\eps=\eps$ on $\Omega_0$. On $\rmL^1_\geq (I_\PM)$ we define the distance
function
\[
\calG(u,\wt u) := \sup\Bigset{\int_{I_\PM} (u{-}\wt u) \varphi \dd y }{ \varphi 
   \in \rmC^1_0(\ol I_\PM), \ \| \pl_y \varphi\|_{\rmC^0} \leq 1},
\]
where the subscript $_0$ indicates the boundary condition
$\varphi(-1)=\varphi(0)=\varphi(1)=\varphi(2)=0$. Clearly, $\calG(u,\wt u)=0$
implies $u=\wt u$ in $\rmL^1(I_\PM)$. With this, \cite[Thm.\,2]{RosSav03TIEC} 
will be applicable.

\begin{proposition}[Space-time compactness in $\Omega_\PM$]
\label{pr:SpaTimeComp}
Let $ E $ and $H$ satisfy \eqref{eq:E.bounds} and \eqref{eq:Cond.H.qH} and
assume that $(u_\eps)_\eps$ satisfies \eqref{eq:MainAss.all}. Then, there
exists a subsequence (not relabeled) and a limit function $u_0$ such that
\begin{subequations}
\label{eq:Subseq}
\begin{align}
&\label{eq:Subseq.a}
u_\eps \weak u_0 \ \text{ (weakly) in } \rmL^{2\qsubE+2\qsubH}(\Omega_\PM),  
\\
&\label{eq:Subseq.b}
u_\eps \to u_0 \ \text{ (strongly) in } \rmL^q(\Omega_\PM)
\text{ for all } q\in {[1,2\qsubE{+}2\qsubH[},  
\\
&\label{eq:Subseq.b1}
\mob(u_\eps) \bddin \rmL^{1+\qsubE/\qsubH}(\Omega_\PM) \ \text{and }\ 
Q_\eps \bddin \rmL^{r_*}(\Omega_\PM) \text{ with } 
 r_*= \tdfrac{2\qsubE{+}2\qsubH}{\qsubE{+}2\qsubH},
\\
&\label{eq:Subseq.c0}
u_\eps \bddin \rmL^{\qsubE{+}2\qsubH}\big([0,T];\rmC^0(I_\PM) \big),
\\
&\label{eq:Subseq.c}
u_\eps \to  u_0 \ \text{ (strongly) in } \rmL^s\big([0,T];\rmC^0(I_\PM) \big)
 \text{ for all } s \in {[1,\qsubE{+}2\qsubH[}.
\end{align}
\end{subequations}
\end{proposition}
\begin{proof} 
\STEP{1. Weak convergence:} By Proposition \ref{pr:ImprovBound} we can select a
subsequence and a $u_0 \in \rmL^{2(\qsubE+\qsubH)}(\Omega_\PM)$ such that
\eqref{eq:Subseq.a} holds. 
 
\STEP{2. Bounds on $\mob(u_\eps)$ and $Q_\eps$}: We first observe that
\eqref{eq:mob.H.growth} and \eqref{eq:L.rstar} imply
$\mob(u_\eps)|_{\Omega_\PM} \bddin \rmL^{\gamma_*}(\Omega_\PM)$ with
$\gamma_*=1{+}\qsubE/\qsubH$. If $Q_\eps$ is the optimal flux in the definition
of $\mfD_\eps^\mafo{flux}(u_\eps)$, then with
$r_*= 2\gamma_*/(1{+}\gamma_*)= 2(\qsubE{+}\qsubH)/(\qsubE{+}2\qsubH)\in
{]1,2[}$ we have
\begin{align*}
&\iint_{\Omega_\PM} |Q_\eps|^{r_*} \dd y \dd t \leq \Big(\iint_{\Omega_\PM}
\tdfrac{Q^2_\eps}{\mob(u_\eps)} \dd y\dd t \Big)^{r_*/2} \Big(
\iint_{\Omega_\PM} \mob(u_\eps)^{r_*/(2{-}r_*)} \dd y \dd t\Big)^{1-r_*/2}
\\
& \leq \sfCF^{(\qsubE+\qsubH)/(\qsubE+2\qsubH)} \| \mob(u_\eps)
\|_{\rmL^{\gamma_*}(\Omega_\PM)}^{\qsubH/(\qsubE+2\qsubH)}=: C_Q< \infty.  
\end{align*} 
Thus, \eqref{eq:Subseq.b1} is established.

\STEP{3. Temporal compactness:} 
With this we obtain a uniform control on the temporal behavior. Testing
the continuity equation with $\varphi$ and integrating in time we obtain 
\begin{align*}
\int_{I_\PM} &\big(u_\eps(s{+}h){-}u_\eps(s)\big) \varphi\dd y = \int_{I_\PM}
\int_s^{s+h} X'_\eps \pl_t u_\eps(t,y) \dd t \,\frac1{X'_\eps} \varphi(y)\dd y  
\\
&= \int_{I_\PM} \int_s^{s+h} Q_\eps(t,y) \dd t \frac1{X'_\eps} \pl_y \varphi(y)
\dd y ,
\end{align*}
where we use that $X'_\eps$ is constant in $I_+$ and in $I_-$. 
Taking the supremum over test functions with $\|\pl_y \varphi\|_{\rmC^0} \leq
1$, we obtain, using Step 2, 
\[
\calG(u_\eps(t{+}h),u_\eps(t)) \leq \frac1{1{-}\eps} \int_s^{s+h} \!\!\int_{I_\PM}
| Q_\eps| \dd y \dd t \leq \frac{(2h)^{1-1/r_*}}{1{-}\eps} \:C_Q^{1/r_*}.  
\]

\STEP{4. Convergence in measure:} Using $r_*>1$ and Lemma
\ref{le:CompSublevels} we are now able to apply the result from
\cite[Thm.\,2]{RosSav03TIEC}  with $\calF$ and $\calG$ using the Banach space
$B:=\rmL^{2\qsubE+2\qsubH}(I_\PM)$. Thus, we are able to select a further subsequence
such that $u_\eps$ converges to $u_0$ in measure, namely
\[
\forall\ \delta>0 : \quad \lim_{\eps \searrow 0} \LEB^1( A(\eps,\delta) )  =
0, \ \text{ where }  A(\eps,\delta):= \bigset{t\in
  [0,T] }{ \| u_\eps(t){-}u_0(t)\|_B \geq \delta} . 
\]

We set $q_*:= 2\qsubE{+}2\qsubH$ and use 
$\| u_\eps\|_{\rmL^{q_*}(\Omega_\PM)} \leq C_U$ as in Step 1. Hence, the estimates
\begin{align*}
&\iint_{\Omega_\PM} \big| u_\eps(t,y) - u_0(t,y)\big| \dd y \dd t \leq 
\int_0^T 2  \| u_\eps(t){-}u_0(t)\|_B \dd t 
\\
&\leq 2\int_{A(\eps,\delta)}   \| u_\eps(t){-}u_0(t)\|_B \dd t
  + 2 \int_{[0,T]\setminus A(\eps,\delta)} \delta \dd t 
\leq 2 \big(\LEB^1(A(\eps,\delta))\big)^{1-1/q_*} \:C_U^{1/q_*} + 2\delta T
\end{align*}
give convergence in $\rmL^1(\Omega_\PM)$. 
Interpolating between $\rmL^1(\Omega_\PM)$ and $\rmL^{q_*}(\Omega_\PM)$ shows
\eqref{eq:Subseq.b}.

\STEP{5. Strong convergence in $\rmL^s \big([0,T];\rmC^0(I_\PM)\big)$}: From
above $ z_\eps =H(u_\eps) \bddin  \rmL^{2\gamma_*}(\Omega_\PM) $ and
$z_\eps \bddin \rmL^2 \big( [0,T];\rmH^1(I_\PM)
\big) $. Space-time interpolation implies 
$z_\eps \bddin \rmL^{\eta}\big([0,T];\rmC^0(I_\PM)\big)$ with
$\eta=1+\gamma_*=2 + \qsubE/\qsubH$. With \eqref{eq:Cond.H.qH} we conclude
$u_\eps \bddin \rmL^{s_*}\big([0,T];\rmC^0(I_\PM)\big)$ with $s_*=\eta q_H=
\qsubE{+}2\qsubH$. Thus, \eqref{eq:Subseq.c0} is established. 

Applying \cite[Thm.\,2]{RosSav03TIEC} once again, but now with the smaller
space $ B= \rmC^0(I_\PM)$, and repeating the arguments of Step 4 with the
convergence in measure leads to the desired convergence \eqref{eq:Subseq.c}. 
\end{proof}

The above results are sufficient to pass to the limit in the bulk part of
$\wh\mfD_\eps(u_\eps)$ but not for the membrane part, because in $\Omega_0$ we
do not have temporal compactness because of the small prefactor $\eps =X'_\eps$
in the continuity equation \eqref{eq:rho.eps.CE}. To handle this case, we use
that we can estimate $u_\eps$ and $\mob(u_\eps)$ in $\rmL^1(\Omega_0)$. This
provides weak* convergence in the space of measures, which is the best we can
hope for, even in the case of the linear diffusion equation with the Otto
gradient structure, i.e.\ $\beta=\varkappa=1$. We first observe that the
transformed density $(t,y)\mapsto  X'_\eps(y) u_\eps(t,y)$, which appears in the
transformed continuity equation \eqref{eq:rho.eps.CE}, vanishes strongly in
$\rmL^2(\Omega_0)$. Moreover, on $\Omega_0$ we exploit the transformation
$u_\eps(t,y) = \psi(w_\eps(t,y))$,  which is assumed to exist in condition
\eqref{eq:Cvx.Ass.psi} and provide a priori bounds for $u_\eps$,
$\mob(u_\eps)$, and $w_\eps$ in $\rmL^1(\Omega_0)$.

\begin{lemma}[Bounds in the membrane part]
\label{le:Memb.Bounds}
Let the assumption \eqref{eq:AllAssump} and \eqref{eq:MainAss.all} hold. Then
there exists a constant $C_*>0$ such that 
\begin{subequations}
\begin{align}
  \label{eq:StrCvgMemb}
 & \big\|X'_\eps u_\eps \big\|_{\rmL^2(\Omega_0)} \leq C_* \eps^{1/2}
  \quad\text{for }\eps\in {[0,1/2]},
\\
\label{eq:L1.Bounds}
&u_\eps \bddin \rmL^1(\Omega_0), \quad 
\mob(u_\eps)  \bddin \rmL^1(\Omega_0),  \quad 
Q_\eps  \bddin \rmL^1(\Omega_0), \quad 
w_\eps=\psi^{-1}(u_\eps)   \bddin \rmL^1(\Omega_0).
\end{align}
\end{subequations}
\end{lemma}
\begin{proof}
From \eqref{eq:MainAss.E} and \eqref{eq:E.bounds} we immediately find a
constant $C_1$ such that $\eps \|u_\eps(t)\|_{\rmL^1(I_0)}  = \|X'_\eps
u_\eps(t)\|_{\rmL^1(I_0)} \leq C_1$ for all 
$t\in [0,T]$ and all $\eps$. Moreover, combining \eqref{eq:Cond.H.qH} with
$\qsubH\geq 1/2$ and \eqref{eq:H.ueps.bound} we have 
\begin{equation}
  \label{eq:L1.ueps.new}
  \int_0^T \| u_\eps(t)\|_{\rmL^{\infty}}\dd t \leq C+C\int_0^T \|
H(u_\eps)\|^2_{\rmL^\infty(I)} \leq C_2.
\end{equation}

With this we can interpolate in space and time as follows:
\[
\|X'_\eps u_\eps\|_{\rmL^2(\Omega_0)}^2
 = \eps \iint_{\Omega_0} \!\! X'_\eps u_\eps\,u_\eps \dd y \dd t 
\leq \eps \| X'_\eps u_\eps \|_{\rmL^\infty(0,T;\rmL^1(I_0)) }
  \| u_\eps\|_{\rmL^1(0,T;\rmL^\infty(I_0)) } \leq \eps C_1C_2,
\]
which is the desired estimate \eqref{eq:StrCvgMemb}.

The first estimate in \eqref{eq:L1.Bounds} follows from \eqref{eq:L1.ueps.new};
and the second is obtained by combining \eqref{eq:mob.H.growth} and
\eqref{eq:L1.ueps.new}. For the third estimate we use \eqref{eq:MainAss.P} and
the second estimate, namely 
\[
\| Q_\eps\|_{\rmL^1} = \iint_{\Omega_0} \frac{|Q_\eps|}{\sqrt{\mob(u_\eps)}} \,
\sqrt{\mob(u_\eps)} \dd y \dd t \leq \Big(\iint_{\Omega_0}
\frac{Q_\eps^2}{\mob(u_\eps)} \dd y \dd t\Big)^{1/2} \|
\mob(u_\eps)\|_{\rmL^1(\Omega_0)}^{1/2} \leq \sfCF^{1/2} C.
\] 
The fourth estimate follows because of $w_\eps \leq C_\psi (1{+}H(u_\eps))^2$
from \eqref{eq:Cvx.Ass.psi}(iii) and \eqref{eq:L1.ueps.new}.%
\end{proof}

\subsection{Characterization of the limiting flux $Q_0$} 
\label{su:CharLimits}

As a first step in the limit passage for $\wh\mfD_\eps(u_\eps)$ we pass to the
limit in the continuity equation \eqref{eq:rho.eps.CE}. Upon choosing a further
subsequence (not relabeled) we can assume that 
\
\begin{equation}
  \label{eq:RhoFluxCvg}
  X'_\eps u_\eps \to X'_0 u_0 \ \text{ in } \rmL^2(\Omega) \qand 
Q_\eps \weaks \ol Q_0 \ \text{ in } \mfS\mfM(\ol\Omega),
\end{equation}
where $\mfS\mfM$ denotes signed measures. In Step 2 of the proof of
Proposition \ref{pr:SpaTimeComp} we have shown that $Q_\eps$ is bounded in
$ \rmL^{r_*} (\Omega_\PM) $, but for the membrane part we only have the bound in
$ \rmL^1(\Omega_0) $ from \eqref{eq:L1.Bounds}. Nevertheless, this is sufficient
to pass to the limit $\eps\to 0$ leading to the limiting continuity equation
\begin{equation}
  \label{eq:LimitCE}
  \iint_{\Omega} \pl_t\varphi\,  X'_0 u_0  \dd y \dd t + 
 \iint_{\ol\Omega} \pl_y\varphi \dd \ol Q_0(t,y) \SSS = 0 \EEE 
  \ \text{ for all }\varphi\in
   \rmC^\infty_\rmc({]0,T[}\ti\ol I).
\end{equation}
The key point of the following result is that $X'_0(y)=0$ in $\Omega_0$, which
implies that $\ol Q_0|_{\Omega_0}$ must be independent of $y\in I_0$, i.e.\ it
may only depend on $t \in [0,T]$.

\begin{proposition}[Properties of $\ol Q_0$] 
\label{pr:Prop.Q0}
Let $\ol Q_0$ be obtained as in \eqref{eq:RhoFluxCvg}, then we have 
\begin{equation}
  \label{eq:Char.Q0}
  \ol Q_0 = Q_0 \LEB^2|_{\Omega_\PM} + \LEB|_{I_0}\otimes \ol\kappa
   \quad \text{with } Q_0 \in \rmL^{r_*}(\Omega_\PM) \text{ and } 
   \ol\kappa \in \mfS\mfM([0,T]),
\end{equation}
with $r_*=  2(\qsubE{+}\qsubH)/(\qsubE{+}2\qsubH)>1$. In particular, the continuity
equation \eqref{eq:LimitCE} reduces to 
\begin{equation}
  \label{eq:RedLimCE}
\iint_{\Omega_\pm}\!\!\Big(u_0\, \pl_t\varphi +
   Q_0 \pl_y \varphi\Big) \dd y \dd t + \int_{[0,T]} \!\! 
  \big( \varphi(t,1){-}\varphi(t,0)\big) \dd \ol\kappa(t)=0 
  \ \text{ for }\varphi\in \rmC^\infty_\rmc({]0,T[}\ti\ol I).
\end{equation}
\end{proposition}
\begin{proof} The absolute continuity of $\ol Q_0$ in $\Omega_\PM$ follows 
from the bounds for $Q_\eps$ derived in Proposition \ref{pr:SpaTimeComp}. 
Testing \eqref{eq:LimitCE} with  $\varphi$ having support in
$\Omega_0={]0,T[}\ti I_0$ yields 
\[
\iint_{\ol\Omega}  \pl_y\varphi(t,y) \dd \ol Q_0  =0 \quad \text{for all }\varphi\in
\rmC^\infty_\rmc( \Omega_0).
\]
Thus, there exists a signed measure $\ol\kappa \in \mathfrak
S\mfM([0,T])$ such that $\ol Q_0|_{\Omega_0}$ is given by $\ol\kappa$, i.e.\
\[
  \iint_{\Omega_0}\! \pl_y\varphi \dd \ol Q_0 = \iint_{\Omega_0} \pl_y\varphi(t,y)
         \dd y \dd\ol\kappa(t)
  = \int_{[0,T]}\!\! \big(\varphi(t,1){-} \varphi(t,0)\big)
         \dd \ol\kappa(t) \ 
\text{ for all } \varphi \in \rmC^0_\rmc(\Omega_0).
\]

By now we have characterized $Q_0$ when restricted to $\Omega_\PM$ and
to $\Omega_0$, however there might be signed measures with support on
the interfaces $[0,T]\ti\{0\}$ and $[0,T]\ti\{1\}$, which we call
$\ol\beta{}^-$ and $\ol\beta{}^+$, respectively. Using the simple form of the
limit $X'_0 u_0 \in \rmL^2(\Omega)$ in \eqref{eq:RhoFluxCvg} with $X'_0(y)=0$
in ${]0,1[}$ and $X'_0(y)=1$ otherwise,  
the continuity equation \eqref{eq:LimitCE} takes the form
\begin{align}
 \nonumber
 \iint_{\Omega_\PM}\!\Big( u_0 \pl_t \varphi + Q_0 \pl_y \varphi \dd y \dd t
  +\int_{[0,T]}\! \big(\varphi(t,1){-} \varphi(t,0)\big) \dd \ol\kappa(t) &
\\
\label{eq:CE.beta.pm}
  + \int_{[0,T]}\! \pl_y \varphi(t,1) \dd \ol\beta{}^+(t) 
 + \int_{[0,T]} \!  \pl_y \varphi(t,0) \dd \ol\beta{}^-(t) &=0
\end{align} 
for all $\varphi \in \rmC^\infty_\rmc({]0,T[}\ti \ol
I)$. Choosing $\psi\in \rmC^\infty_\rmc(\R)$ with $\psi'(0)\neq 0$ and
$\alpha^\pm\in \rmC^\infty_\rmc({]0,T[})$ we set 
\[
\varphi_n(t,y)= \frac1n\big( \alpha^-(t) \psi(n y)+ \alpha^+(t)
\psi(n(y{-}1)) \big)
\]
and insert this into \eqref{eq:CE.beta.pm}. Taking the limit $n\to
\infty$ we find 
\[
\psi'(0)\int_0^T \alpha^-\dd\ol\beta{}^- + \psi'(0) \int_0^T \alpha^+
\dd \ol\beta{}^+=0.
\]
Since $\alpha^\pm$ were arbitrary, we conclude $\ol\beta{}^\pm=0$, and the
result is established.  
\end{proof} 

As in Section \ref{su:MainResult} the masses $M_\pm(t)$ in the left and right
bulk parts are given as $M_\pm(t):=\int_{I_\pm} u_0(t,y) \dd y$. As $\int_I
X'_\eps(y) u_\eps(t,y) \dd y =1$ for a.a.\ $t\in [0,T]$ we have
$M_+(t)+M_-(t)=1$ a.e.\ in $[0,T]$. 
We can test the reduced limit continuity equation \eqref{eq:RedLimCE} by
$\varphi$ satisfying $\varphi(t,y)=\phi_\pm (t)$ on $\Omega_\pm$ and obtain 
\[
\int_0^T\!\big( \phi_+(t)M_+(t)+\phi_-(t)M_-(t)\big) \dd t
 + \int_{[0,T]}\!\big( \phi_+(t)-\phi_-(t) \big)\dd\ol\kappa(t) = 0
\]
for all $\phi_\PM\in \rmC_\rmc({]0,T[})^2$, which means that
\eqref{eq:flux.kappa.3} extends to the more general case if we interpret it in
the sense of distributions.

It will be a major step below to show that under the above assumptions
$\ol\kappa$ is again absolutely continuous, namely $\ol\kappa= \kappa \LEB|_{[0,T]} $
with $\int_0^T |\kappa| \log(1{+}|\kappa|)\dd t < \infty$, see
Proposition \ref{pr:kappa.int}. Then, $\ol Q_0$ is
absolutely continuous, i.e.\ $ \ol Q_0 = Q_0\,\LEB^2|_\Omega$ with $Q_0 \in
\rmL^1(\Omega)$. This will then show that $u_0$ lies in the space $\bfL_1$
defined in \eqref{eq:def.bfL1} which is essential for the definition of
$\wh\mfD_0$.

\begin{remark}\slshape
\label{re:mob.HighInt}
Indeed, strengthening the assumption \eqref{eq:mob.H.growth} to
$ \mob(u)^{1+\delta} \leq C(1{+}H(u)^2)$ for some $\delta>0$, one
easily derives $ \kappa \in \rmL^1(0,T)$ by mimicking the
estimate on $\Omega_\PM$: First we obtain $\mob(u_\eps) \weak  m_0$ in
$\rmL^{1+\delta}(\Omega_0)$ and, as in the beginning of the proof of
Proposition \ref{pr:Prop.Q0}, we find $Q_0|_{\Omega_0} \in
\rmL^{r_\delta}(\Omega_0)$ with $r_\delta=(2{+}\delta)/(1{+}\delta)>1$ and obtain
$\kappa \in \rmL^{r_\delta}([0,T])$ as well.   

However, in our example in Proposition \ref{pr:PowerLawCvx} the stronger
assumption $\mob(u)^{1+\delta} \leq C(1{+}H(u)^2)$ leads to the restriction
$\varkappa \geq 1+ \delta \gamma/2$, which would exclude the important
case of the Boltzmann entropy with $\varkappa =\qsubE=1$.
\end{remark}

\section{The \TOS{$\bm{\Gamma}$}{Gamma}-liminf estimate for
  \TOS{$\wh\mfD_\eps$}{the transformed dissipation functional}}
\label{se:Liminf} 

The liminf estimate for the bulk domain $\Omega_\PM$ will follow by standard
lower semi-continuity arguments using the space-time compactness derived in
Proposition \ref{pr:SpaTimeComp}. The main difficulties of the limit passage
\AAA arise \EEE from the membrane part, where the temporal compactness is lost
and we have to resort to weak-convergence methods. Moreover, an additional
problem arises because the a priori estimates in Lemma \ref{le:Memb.Bounds} are
so weak that we have to work with weak* convergence in the space of measures.
Our strategy is similar to that in \cite[Sec.\,4.3+4]{FreLie21EDTS}:\\
\mbox{}\ \AAA (i) \ convexify \EEE the functional using $\psi$ from
\eqref{eq:Cvx.Ass.psi}, \\
\mbox{}\:(ii)\: extend the functional to the set of measures,\\
\mbox{}\,(iii) use sequential weak* lower semi-continuity of the extended
functional, \\
\mbox{}\,(iv)\, show that the limit is absolutely continuous.

We emphasize that the convexification is only used for the membrane part, where
the continuity equation degenerates because of the temporal scale separation in
the membrane. Applying this nonlinear transformation in the bulk would not
work, because it destroys the linear structure of the continuity equation.

We first treat the bulk part, which is relatively straightforward because of
the favorable a priori estimates and the strong convergence of
$u_\eps|_{\Omega_\PM}$ provided in Proposition \ref{pr:SpaTimeComp}.  We recall
that the bulk part is given by 
\[
\wh\mfD^\text{bulk}_\eps(u) =\int_{\Omega_\PM} \Big( 
\frac{Q^2}{2\bbA_\eps \mob(u)} + \frac{\bbA_\eps}2 \big|\pl_y H(u)\big|^2 
\Big) \dd y \dd t
\]
with the harmless dependence on $\eps$ through $\bbA_\eps$. Moreover, the
singular dependence on $\eps$ in the scaled continuity equation \eqref{eq:rho.eps.CE}
does not appear in the bulk domain $\Omega_\PM$.

\begin{proposition}[Liminf estimate in the bulk]
\label{pr:liminf.bulk}
Let assumption \eqref{eq:AllAssump} be valid and consider a sequence
$(u_\eps)_\eps$ in $\rmL^1(\Omega)$ with $X'_\eps u_\eps \weak X'_0 u_0 $ in
$\rmL^1(\Omega_\PM)$ such that \eqref{eq:MainAss.all} holds. Then, we have
\[ 
 \liminf_{\eps\to 0} \wh\mfD^\mafo{bulk}_\eps(u_\eps) \geq
     \wh\mfD^\mafo{bulk}_0(u_0).
\]
\end{proposition}
\begin{proof} We recall the convergences in \eqref{eq:Subseq} and observe that
$z_\eps= H(u_\eps)$ satisfies $z_\eps \weak z_0\in \rmL^2 \big([0,T]; 
\rmH^1(I_\PM)\big)$. By \eqref{eq:Subseq.b1} we may assume 
$\mob(u_\eps) \weak m_0$ in $\rmL^{1+ \qsubE/\qsubH}(\Omega_\PM)$ and 
$ Q_\eps \weak \wt Q_0$ in $\rmL^{r_*}(\Omega_\PM)$ with $r_*>1$.
Using the strong convergence in \eqref{eq:Subseq.b} we conclude 
$z_0=H(u_0)$ and $m_0=\mob(u_0)$. In particular, we also have $g_\eps: = 
\SSS \pl_y z_\eps =\pl_y H(u_\eps) \weak \pl_y z_0 \EEE =: g_0$ in $\rmL^2(\Omega_\PM)$. 

The integrand of $\wh\mfD^\text{bulk}_\eps$ can be rewritten in the 
form $ Q_\eps^2/(2\bbA_\eps m_\eps) + \frac{\bbA_\eps}2|g_\eps|^2$, which is
non-negative, convex in $(Q_\eps,m_\eps,g_\eps)$, and lower semi-continuous, if
the quotient $Q^2/m$ is set to $0$ for $(Q,m)=(0,0)$ and equal to $+\infty$ for
$m=0$ and $Q\neq 0$. 
Now standard liminf estimates  (see e.g.\
\cite[Thm.\,5.14]{FonLeo07MMCV}) imply 
\[
 \liminf_{\eps\to 0} \wh\mfD^\mafo{bulk}_\eps(u_\eps)  \geq \int_{\Omega_\PM} 
\Big(  \frac{\wt Q_0^2}{2\bbA_0\mob(u_0)} + \frac{\bbA_0}2 
\big|\pl_y H(u_0)\big|^2  \Big) \dd y \dd t \geq    \wh\mfD^\mafo{bulk}_0(u_0).
\]
Here the last estimate follows because $\wt Q_0$ satisfies the limiting
continuity equation \eqref{eq:LimitCE}, and $\wh\mfD^\mafo{bulk}_0(u_0)$
contains the minimizing $ Q_0$.  
\end{proof}

For treating the integral over the membrane part $\Omega_0$, we exploit the
transformation \AAA $w_\eps = \psi^{-1}(u_\eps)$ \EEE where $\psi$ is the 
function satisfying the conditions \eqref{eq:Cvx.Ass.psi}. In particular, we
will use that $w \mapsto \sfn(w)=\mob(\psi(w))$ and $w  \mapsto \sfg(w)= 1/ 
\big(\sfn(w)(\psi'(w)E''(\psi(w)))^2\big)$ are concave, see
\eqref{eq:Cvx.Ass.psi}(i)+(ii). In particular, the limits  
\[
\sfn_\infty:= \lim_{w\to \infty} \frac{\sfn(w)}w \in \R_\geq  \qand  
\sfg_\infty:= \lim_{w\to \infty} \frac{\sfg(w)}w \in \R_\geq 
\]
exist.  With this we define the functional $\wt\mfD_\mfM$ on
$\bfK:= \mfS\mfM(\ol\Omega{}^0)\ti \mfM(\ol\Omega{}^0) \ti
\mfS\mfM(\ol\Omega{}^0)$ by using the decompositions
\[
\ol Q=q \LEB^2 + Q^\rmS \in \mfS\mfM(\ol\Omega{}^0), \quad \ol W   = w\LEB^2 +
W^\rmS,  \qand \ol Z = z \LEB^2 + Z^\rmS,
\]
where $q, w, z\in \rmL^1(\Omega_0)$ and $Q^\rmS,\,W^\rmS, \,Z^\rmS \perp
\calL^2$ (singular with respect to $\LEB^2|_{\Omega_0}$): 
\begin{equation}
  \label{eq:def.wt.mfD0}
\begin{aligned}  
  \wt\mfD_\mfM(\ol Q,\ol W, \ol Z)&:=  \AAA \wt\mfD_\mafo{rg}(q,w,z) 
  + \wt\mfD_\mafo{sg}( Q^\rmS,\,W^\rmS, \,Z^\rmS  ) \quad \text{where }
\\
 \wt\mfD_\mafo{rg}(q,w,z) \EEE &:= \frac12\iint_{\Omega_0}\!
  \Big(\frac{q^2}{\sfn(w)} +\frac{z^2}{\sfg(w)} \Big) \dd y \dd t \quad \AAA
\text{and} \\
 \wt\mfD_\mafo{sg}(
  Q^\rmS,\,W^\rmS, \,Z^\rmS  ) \EEE &:=   \frac12\iint_{\ol\Omega{}^0} \! \Big(
    \frac{(\rmd Q^\rmS\!/\rmd N)^2}{\sfn_\infty \rmd W^\rmS\!/\rmd N}
   +\frac{(\rmd Z^\rmS\!/\rmd N)^2}{\sfg_\infty \rmd W^\rmS\!/\rmd N} 
 \Big) \dd N(y,t),
\end{aligned}
\end{equation}
where $N\in \mfS\mfM(\ol\Omega{}^0)$ is any measure such that
$|Q^\rmS| {+} W^\rmS {+} |Z^\rmS| \ll N $, and $\rmd A^\rmS\!/\rmd N$ denotes
the Radon-Nikodym derivative.  In the case of
$\sfn_\infty\rmd W^\rmS/\rmd\rmN =0$ or $\sfg_\infty\rmd W^\rmS/\rmd\rmN =0$
the terms in the second integral have to be interpreted in the sense that the
nominator has to vanish $N$-a.e.\ in $\ol\Omega{}^0$.  By assumption
\eqref{eq:Cvx.Ass.psi} the functional $\wt\mfD_\mfM$ is convex, and the general
theory of the calculus of variations shows that it is sequentially weak* lower
semi-continuous on $\bfK$, see Section \ref{suA:CvxFunctional} relying on
\cite{FonLeo07MMCV}, because the second integral is 1-homogeneous and is
obtained from the recession function of the integrand of the first integral.

Considering only the integral over the membrane part of $\wh\mfD_\eps$ in
\eqref{eq:def.mfFeps} we obtain
\[
\ol\mfD{}_\eps^{\Omega_0}\!(u)= \iint_{\Omega_0}\Big( \frac{Q_\eps^2}{2b\,\mob(u)} +
\frac{b\,\mob(u)}2 \big|\pl_y E'(u)\big|^2 \Big) \dd y \dd t,
\] 
where $Q_\eps$ is the minimizer subject to \eqref{eq:rho.eps.CE}. Of
course, the construction of $\wt\mfD_\mfM$ is such that 
\begin{equation}
  \label{eq:TrafoCvx}
  \ol\mfD{}_\eps^{\Omega_0}\!(u_\eps) = b\, \wt\mfD_\mfM\big( (Q_\eps/b)\LEB^2,
(w_\eps) \LEB^2 , (\pl_y w_\eps)\LEB^2 \big)  , \quad 
\text{where } w_\eps=\psi^{-1}(u_\eps).
\end{equation}
Thus, it will be easy to derive the liminf estimate in
terms of the weak limits 
\begin{equation}
  \label{eq:WeaksCvg}
  \ol Q_\eps = Q_\eps\LEB^2 \weaks \ol Q_0=\LEB|_{I_0}{\otimes} \ol\kappa , \quad 
\ol W_\eps = w_\eps\LEB^2 \weaks \ol W_0, \quad 
\ol Q_\eps = \pl_y w_\eps \LEB^2 \weaks \ol Z_0.
\end{equation}
Note that $\pl_y w_\eps \bddin \rmL^1(\Omega_0)$ because $0\leq 
\sfg(w_\eps)\leq C(1{+}w_\eps) \bddin \rmL^1(\Omega_0)$ and
$\wh\mfD_\eps^\text{slope}(u_\eps) \leq \sfCS$. 

The nontrivial point is to gain enough control on the limits
$\ol Q_0,\ \ol W_0$, and $\ol Z_0$. For $\ol W_0$ and $\ol Z_0$ we will use the
fact that $u_\eps \in \rmL^1([0,T];\rmC^0(I)) $ has well-defined traces
$\sfT(u_\eps)= (u_\eps(\cdot,0), u_\eps(\cdot,1))$ converging strongly in
$\rmL^s([0,T])^2$ for $s \in {[1,\qsubE{+}2\qsubH[}$, see \eqref{eq:Subseq.c}.
Using $w_\eps=\psi^{-1}(u_\eps)$ we obtain a corresponding control for the
traces of $w_\eps$ and of the limit $ \ol W_0$. With this, we are able to
show that the singular parts $Q^\rmS,\ W^\rmS$, and $Z^\rmS$ have to
vanish. Then, we can minimize for $w_0(t,\cdot)$ for given traces $w_0(t,0)$
and $w_0(t,1)$, thus obtaining the desired membrane functional $\wh\mfD^\mb_0$
defined in \eqref{eq:mfD0.3C}.

Finally, for the membrane flux we obtain $\ol\kappa = \kappa \LEB|_{[0,T]}$
with $\int_0^T \CCC(\kappa(t)) \dd t < \infty $ where $\CCC$ is given in
\eqref{eq:CCC} and has the lower bound $\CCC(a) \geq \frac12
|a|\log(1{+}|a|)$. For this, it will be essential to use the fact that the
limiting flux $Q_0$ satisfies $\pl_y Q_0 \equiv 0$ on $\Omega_0$, which is a
property that was inferred indirectly from the continuity equation
\eqref{eq:LimitCE} and $X'_0=0$ on $\Omega_0$ (i.e.\ the time-scale
separation).  Only this special property allows us to derive the superlinearity
for showing the absolute continuity $\ol\kappa = \kappa \LEB$, see also Remark
\ref{rem:kappa}.  \AAA The following proof relies on the ideas in
\cite[Thm.\,3.2]{AMPSV12PLWG} and \cite[Sec.\,4.3+4]{FreLie21EDTS}.  \EEE

\begin{proposition}[Liminf in the membrane]
\label{pr:liminf.memb}
Let assumptions \eqref{eq:AllAssump} and \eqref{eq:Cvx.Ass.psi}
hold and consider a sequence $(u_\eps)_\eps$ such that the bounds
\eqref{eq:MainAss.all} and the convergences \eqref{eq:Subseq} hold.
Then, for the functional $\ol\mfD{}_\eps^{\Omega_0}$  we have
the  liminf estimate  
\[
\liminf_{\eps \to 0} \ol\mfD{}_\eps^{\Omega_0} \! (u_\eps)   \geq 
\wh\mfD^\mb _0( \sfT(u_0),\kappa),
\]
where  $\kappa=\sfF(u_0) \in \rmL^1([0,T])$ is the membrane flux as in
\eqref{eq:FluxTrace.3}.  
\end{proposition}
\begin{proof}
\AAA We note that establishing the liminf estimate will be the easier part (see
Step 5) of this proof. The main efforts go into showing that the flux
$\ol\kappa$ has a Lebesgue density $\kappa \in \rmL^1(0,T])$.\EEE 
 
  \STEP{1. Convexification by using $w_\eps=\psi^{-1}(u_\eps)$.} 
  Using the relation \eqref{eq:TrafoCvx} we obtain the convex functional
  $\wt\mfD_\mfM$ in terms of $(Q_\eps, w_\eps,\pl_y w_\eps)$ considered as
  (signed) measures. 

\STEP{2. Liminf estimate.} Using the weak* convergences and the lower
semi-continuity of $\wt\mfD_\mfM$ (cf.\ Theorem \ref{th:Relaxation}) we obtain 
\[
\liminf_{\eps\to 0} \ol\mfD{}_\eps^{\Omega_0}\!(u_\eps) = \liminf_{\eps\to 0} 
b\,\wt\mfD_\mfM(Q_\eps/b, w_\eps,\pl_y w_\eps)\geq b\,\wt\mfD_\mfM( 
\ol Q_0/b, \ol W_0,\ol Z_0). 
\]

\STEP{3. Disintegration on $\ol\Omega{}^0=[0,T]\ti \ol I{}^0$.} We now only
consider the second integral in $\wt\mfD_\mfM$ which is given in terms of the
singular measure $N^\rmS$.  As $\wt\mfD_\mfM(\ol Q_0,\ol W_0,\ol Z_0) < \infty$
and $\sfn^\infty,\sfg^\infty< \infty$, we can assume $N^\rmS = W_0^\rmS$.
Exploiting the disintegration of measures on
$\ol\Omega{}^0 = [0,T]\ti \ol I{}^0$ (see \cite[Sec.\,5.3]{AmGiSa05GFMS} and
more specifically \cite[Lem.\,4.6]{FreLie21EDTS}) together with
$\pl_y \ol Q{}^\rmS_0=0$ and $\pl_y \ol W_0=\ol Z_0$ (following from taking the 
limit in $\pl_y \ol W_\eps= \ol Z_\eps$) we find a singular measure
$\ol \mu^\rmS\in \mfM([0,T])$ and functions
$\wt\kappa \in \rmL^1([0,T];\ol\mu^\rmS)$ and
$\wt w\in \rmL^1(\ol\Omega{}^0; \ol\mu^\rmS {\otimes} \LEB|_{I_0})$ such that
\[
\ol Q{}^\rmS_0= \wt\kappa \:\ol\mu^\rmS, \quad 
\ol W{}^\rmS_0= \wt w \:\ol\mu^\rmS, \qand 
\ol Z{}^\rmS_0= \pl_y\wt w \:\ol\mu^\rmS. 
\] 
 With  this \AAA $\wt\mfD_\mafo{sg}\big (\ol Q{}^\rmS_0,  \ol W{}^\rmS_0, \ol
 Z{}^\rmS_0\big)$ from \eqref{eq:def.wt.mfD0} \EEE takes the form 
\[
\AAA \wt\mfD_\mafo{sg} \EEE (\wt\kappa, \wt w, \ol\mu^\rmS) := \frac12\int_{[0,T]}\int_{I_0} \!
\Big(\frac{\wt\kappa^2}{\sfn^\infty\wt w} + \frac{(\pl_y \wt w)^2} 
  {\sfg^\infty \wt w} \Big) \dd y \dd\ol\mu^\rmS.  
\]
The second term in $\wt\mfD_\mafo{sg}(\wt\kappa, \wt w, \ol\mu^\rmS) < \infty$
implies that $I_0\ni y \mapsto \wt w(t,y)$ is continuous for
$\ol\mu^\rmS$-a.a.\ $t\in [0,T]$. Hence, we can now minimize for
$\ol\mu^\rmS$-a.a.\ $t \in [0,T]$ with respect to $\wt w(t,\cdot)$ in terms of
the traces $\wt w_-(t)= \wt w(t,0)$ and $ \wt w_+(t) = \wt w(t,1)$.

The result is explicit as the integrand of
$\wt\mfD_\mafo{sg}$ is a simple rescaling of the functional $\bfM_{1,1}= \bfN_1$
as introduced in Section \ref{su:PowerLaw}.  Using
$\ol N_1(u_\PM,\kappa)= \inf \bigset{\bfN_1(u,\kappa)} { u\in \bfP(u_\PM)}$
(see \eqref{eq:olN.1.expl} for the explicit form of $\ol N_1(u_\PM,\kappa)$) we
obtain
\begin{align}
\nonumber
 \wt\mfD_\mafo{sg}(\wt\kappa, \wt w, \ol\mu^\rmS)& \geq \int_{[0,T]}
 \frac1{\sfg^\infty} \,\ol N_1\big( \wt w_\PM(t), s_\infty \wt\kappa(t)
  \big) \dd \ol\mu^\rmS \qquad \text{with
   } s_\infty= ( \sfg^\infty / \sfn^\infty )^{1/2}\\
\label{eq:LowBd.mfDII}
& = \int_{[0,T]} \Big( \frac{\sqrt{\wt w_-\wt w_+}}{\sfg^\infty} \,\CCC\big(
\tdfrac{s_\infty \wt\kappa}{\sqrt{\wt w_-\wt w_+}} \big) + \frac2{\sfg^\infty}
\big( \sqrt{\wt w_+} - \sqrt{\wt w_-}\big)^2\Big) \dd \ol\mu^\rmS .
\end{align}
In the last integral the integrand is understood to be extended to $(\wt
w_\PM,\wt \kappa) \in \R_\geq^2\ti \R$ by the unique  lower semi-continuous
hull. Note $\ol N_1(0,0,k)=+\infty$ for $k\neq 0$.
 
\STEP{4. Vanishing of the singular parts $\ol W^\rmS$ and $\ol Q{}^\rmS_0$.} 
We now want to prove $\ol Q{}^\rmS_0=0$ by showing $\wt \kappa\equiv 0$,
which will be deduced from $\ol N_1 (0,0,\wt\kappa)=\infty$ for $\wt\kappa\neq 0$. 

We use that the traces $U_\eps^\PM = \sfT(u_\eps)$ are well-defined, satisfy
$U^\PM_\eps \bddin \rmL^{s_*}([0,T])^2$ with $s_*=\qsubE+ 2\qsubH\geq 2$, and
converge to a limit $U^\PM_0=\sfT(u_0)$ strongly in $\rmL^s([0,T])^2$ for all
$s\in {[1,s_*[}$, see \eqref{eq:Subseq.c0} and \eqref{eq:Subseq.c}.  We now
combine \eqref{eq:Cond.H.qH} with condition \eqref{eq:Cvx.Ass.psi}(iii) which
provides the bound $w=\psi^{-1}(u)\leq C_\psi(1{+}u)^{2 \qsubH}$. Thus, we
conclude $W^\PM_\eps = \psi^{-1} \AAA \big( U^\PM_\eps\big) \EEE \bddin
\rmL^{r_0}([0,T])^2$ with 
$r_0=1+\qsubE/(2\qsubH)>1$ and the strong converge
$W^\PM_\eps \to W_0^\PM =\psi^{-1} \big( U^\PM_0 \big) $ in $\rmL^r([0,T])^2$ for all
$r\in {[1,r_0[}$.  Thus, the traces cannot develop a singular measure, viz.
\begin{align}
\nonumber
 &\wt w^\PM(t)= (0,0) \quad \ol\mu^\rmS\text{-a.e.\ in } [0,T] \qand 
\\
\label{eq:trace.z.u}
 &\sfT(w_0)=(w_0(t,0),w_0(t,1))=W^\PM_0=\psi^{-1}(U_0^\PM)) 
  \quad  \LEB\text{-a.e.\ in } [0,T],
\end{align}
where $w_0$ is the density of the Lebesgue part of $\ol W_0$. 

Using $ \wt\mfD_\mafo{sg}(\wt\kappa,\wt w,\ol\mu^\rmS) \leq \frac1b
\liminf \wh\mfD_\eps(u_\eps)$ $\leq (\sfCF{+}\sfCS)/b< \infty$ and inserting
$\wt w^\PM(t)= (0,0)$ into the lower bound \eqref{eq:LowBd.mfDII}, we can exploit
the property that $\ol N_1(0,0,k)=+\infty$ for $k\neq 0$. We conclude that 
$\wt\kappa(t)=0 $ for $\ol\mu^\rmS$-a.a.\ $t\in [0,T]$, which gives 
 $\ol Q_0^\rmS = \wt\kappa \ol\mu^\rmS =0$, i.e.\ $\ol Q_0$
is absolutely continuous and  has the form $\ol Q_0 = \kappa
\LEB^2|_{\Omega_0}$ with $\kappa \in \rmL^1([0,T])$. 

\underline{\em Step 5. Final membrane part.} We now return to the first
integral in $\wt\mfD_\mfM$ as defined in \eqref{eq:def.wt.mfD0} which depends
on the Lebesgue parts of $\ol Q_0= \kappa_0 \LEB^2$ and
$\ol W_0 = w_0 \LEB^2 + \wt w \ol\mu^\rmS$. Using the trace relation
\eqref{eq:trace.z.u} we can undo the nonlinear $\psi$-transformation by
defining $u_0:= \psi(w_0)$ and obtain
\begin{align*}
\liminf_{\eps \to 0}\ol\mfD_\eps^{\Omega_0} (u_\eps )
&  \geq \frac b2 \iint_{\Omega_0} \!\Big(\frac{(\kappa_0/b)^2}{\sfn(w_0)} +
  \frac{(\pl_y w_0)^2}{\sfg(w_0)} \Big) \dd y \dd t 
\\ &
\overset{u=\psi(w)}= \frac b2 \iint_{\Omega_0}
\!\Big(\frac{(\kappa_0/b)^2}{\mob(u_0))} + 
  \big(\pl_y H(u_0)\big)^2 \Big) \dd y \dd t
\\
& \geq \int_0^T b \,\ol M(U_0^\PM(t),\kappa_0(t)/b) \dd t\  
  = \ \wh\mfD^\mb _0(U_0^\PM,\kappa_0), 
\end{align*}
where \AAA the first estimate uses $\wt\mfD_\mafo{sg}\geq 0$ and \EEE 
the last estimate uses $\bfM$ and $\ol M$ as defined in
\eqref{eq:RMJ}. Proposition \ref{pr:liminf.memb} is proved.
\end{proof}

A crucial point in the above proof is to show that the singular part
of $\ol Q_0$ vanishes, i.e.\ $\ol\kappa= \kappa
\,\!\LEB|_{[0,T]}$. Unfortunately, the proof is rather indirect and
relies on the limiting characterization of $\ol Q_0$ via $\pl_y\ol Q_0\equiv 0$
and the control of the traces $U^\PM_\eps=\sfT (u_\eps)$ in $\rmL^s([0,T])$ for
$s\in {[1,2\qsubE{+}2\qsubH[}$.   

\begin{remark} \label{rem:kappa}
It would be desirable to derive equi-integrability of the fluxes
$(Q_\eps)_{\eps>0} $ directly. This is easily possible in the case
$\mob(u)^{1+\delta}\leq C_\mob (1{+}H(u)^2)$ with $\delta>0$, see
\cite[Sec.\,5]{Fren19DEGS}.  

In the critical case $\delta=0$, the bottleneck is the functional
$\bfM_{1,1}(Q,u)$ which is quadratic in $Q$, but when minimizing with
respect to $u\in \bfP(u_\PM)$ we only retain a linear lower bound.
Indeed, in Section \ref{su:Bound.calM} we show that $Q\mapsto
\calJ(u_\PM,Q)=\inf\bigset{\bfM_{1,1}(Q,u)}{ u\in \bfP(u_\PM)}$ satisfies
$\calJ(u^\pm,Q)\leq C_\eta\|Q\|_{\rmL^2(I_0)}$ for all $Q$ with
$\mafo{sppt} (Q)\subset[\eta,1{-}\eta] \subset {]0,1[}$. Nevertheless $\kappa
\mapsto \calJ(u^\PM,\kappa \bm1_{I_0})= \ol M_{1,1}(u^\PM,\kappa)=\ol
N_1(u^\PM,\kappa ) $ grows superlinearly, which was crucial in Step 4 of the
above proof.  
\end{remark}

Finally, we want to provide a constructive estimate on the superlinear
integrability of the limiting flux $\kappa \in \rmL^1([0,T])$ obtained in
Proposition \ref{pr:liminf.memb}.  For this we use the bounds of
$U^\PM=\sfT(\rho)$ in $\rmL^s([0,T])^2$ and the fact that condition
\eqref{eq:mob.H.growth} provides a lower bound for 
$\ol M_{1,1}(u_\PM,\kappa)$. We refer to \cite[p.\,415]{FreLie21EDTS} for a similar
Orlicz bound and to Section \ref{su:PowerLaw} for more and 
better bounds for the power-law functions $\ol M_{\beta,p}$ treated in Proposition
\ref{pr:PowerLawCvx}.

\begin{lemma}[Lower bound on $\ol M$]
\label{pr:LowBound.bmM}
Let the assumption \eqref{eq:AllAssump} hold, then $\ol M$ defined in
\eqref{eq:I.def.Rmemb} satisfies the lower bound
\begin{equation}
  \label{eq:bmM.LowBound}
  \ol M(u^\PM,\kappa) \geq \frac 14\, \ol N_1\big((1{+}H(u_\PM))^2,
  \frac{\kappa}{2\sqrt{C_\mob}} \big)
\end{equation}
with $\ol N_1$ given explicitly in \eqref{eq:olN.1.expl}. 
\end{lemma}
\begin{proof}
Using \eqref{eq:mob.H.growth} we estimate $\bfM(Q,u)$ from below:
\begin{align*}
\bfM(Q,u)
&= \int_0^1\Big(\frac{Q^2}{2 \mob(u)}
   + \frac{\mob(u)}2 \big|\pl_y E'(u)\big|^2 \Big) \rmd y
\geq   \int_0^1\!\!\Big(\frac{Q^2}{2 C_\mob(1{+}H(u)^2)} + \frac12
  \big|\pl_y H(u)\big|^2 \Big) \dd y
\\
&\geq   \int_0^1\!\!\Big(\frac{Q^2/C_\mob}{2 z} + \frac14\,
  \frac{(\pl_y z)^2}{2z}\Big) \dd y \quad
\text{with } z=(1{+}H(u))^2. 
\end{align*}
Inserting $Q=\kappa \bm1_{I_0}$, using the definition of $\ol M_{1,1}=\ol N_1$ 
in Section \ref{su:PowerLaw}, and minimizing with respect to $z$ with the
boundary conditions $z_\pm= (1{+}H(u_\pm))^2 $ we obtain the desired lower
bound.%
\end{proof} 

The final bound on $\kappa$ now follows by using the explicit form of 
$\ol N_1(u_\PM,\kappa)$ in terms of $\CCC$ (cf.\ \eqref{eq:olN.1.expl}),
the control of $u_\PM$ in $\rmL^s([0,T])$ derived in 
\eqref{eq:Subseq}, and the relative-entropy type bound 
\[
 \int_0^T \CCC(\kappa(t)) \dd t
  \leq \frac q{q{-}1} \int_0^T a(t) \CCC\big(\kappa(t)/a(t)\big) \dd t  
   +  \frac2{q{-}1} \|a\|_{\rmL^q(0,T)}^q,
\]
which is established in Proposition \ref{pr:calC.estimate}.

\begin{proposition} 
\label{pr:kappa.int}
Let the assumptions \eqref{eq:AllAssump} and \eqref{eq:Cvx.Ass.psi} be
satisfied. Set $r_*=1+\qsubE/(2\qsubH)>1$ and
$s_*=q_H p_*= \qsubE{+}2\qsubH \geq 2$. Then, there exists $C_*>0$ such that
the following holds. If $ u_0\in \rmL^1(\Omega)$ and $Q_0\in \rmL^1(\Omega)$
are limits of a sequence $(u_\eps)_{\eps>0}$ satisfying \eqref{eq:MainAss.all},
then $Q_0|_{\Omega_0}= \kappa \LEB^2$ with $\kappa\in \rmL^1([0,T])$ satisfying
\[
\int_0^T \!\!\CCC(\kappa) \dd t \leq \frac {8r_*}{r_*{-}1} \sqrt{C_\mob}
\big(\sfCF{+}\sfCS\big) + C_* \|1{+}U^-\|_{\rmL^{s_*}(0,T)}^{\qsubH}
\|1{+}U^+\|_{\rmL^{s_*}(0,T)}^{\qsubH} .
\] 
\end{proposition}
\begin{proof} Combining $\int_0^T \ol M(U_0^\PM,\kappa) \dd t \leq
  \wh\mfD^\mb _0(U^\PM,\kappa) \leq \sfCF{+}\sfCS<\infty$, the
  estimate \eqref{eq:bmM.LowBound}, and  $\ol N_1(u_\PM,\kappa) \geq
  \sqrt{u_-u_+}\,\CCC\big(\kappa/\sqrt{u_-u_+}\big)$    
from \eqref{eq:olN.1.expl} we find
\[
\int_0^T a \,\CCC(\kappa/a)\dd t \leq
 8\sqrt{C_\mob} \big(\sfCF{+}\sfCS\big), \quad \text{where } 
a= 2\sqrt{C_\mob}\, (1{+}H(U_0^-))\,(1{+}H(U_0^+)). 
\]

We next show that $a$ lies in $\rmL^{p_*}(0,T)$. For
this we exploit condition \eqref{eq:Cond.H.qH} and the trace
bound $ U^\PM_0 \in \rmL^{\qsubE+2\qsubH}([0,T])^2$ following from
\eqref{eq:Subseq.c0}. Hence, we have 
\begin{align*}
\|a\|_{\rmL^{p_*}}^{p_*} &\leq \big(4C_\mob\big)^{p_*/2}
\|1{+}H(U_0^-)\|_{\rmL^{2p_*}} \|1{+}H(U_0^+)\|_{\rmL^{2p_*}} \\
&\leq \big(4C_\mob\big)^{p_*/2}
     \big\|1+C_H(1{+}U_0^-)^{q_H}\big\|_{\rmL^{2p_*}}  
     \big\|1+C_H(1{+}U_0^+)^{q_H}\big\|_{\rmL^{2p_*}}  \\
& \leq \big(4C_\mob\big)^{p_*/2} \ol C_H^2
    \|1{+}U_0^-\|_{\rmL^{s_*}}^{q_H} \|1{+}U_0^+\|_{\rmL^{s_*}}^{q_H} <\infty,
   \ \text{ where } \ol C_H=\max\{1,C_H\}. 
\end{align*}
Hence, Proposition \ref{pr:calC.estimate} is
applicable and the result follows.
\end{proof}

We are now ready to prove the first part of Theorem \ref{th:Main.Gamma.Cvg}.
\vspace{0.3em}

\noindent
\begin{proof}[{\bfseries\upshape Proof of Part A of Theorem \ref{th:Main.Gamma.Cvg}}] We consider 
a general sequence $(u_\eps)_\eps$ such that
$X'_\eps u_\eps(t,\cdot) \in \calP(I)$ and that \eqref{eq:MainAss.all}
holds. We set $\alpha_0 =\liminf_{\eps \to 0} \wh\mfD_\eps(u_\eps) \in \R_\geq$.
Hence, we can choose a subsequence (not relabeled) with
$ \wh\mfD_\eps(u_\eps) \to \alpha_0$. The above theory provided us with a further
subsequence and a function $u_0\in \bfL_1$ such that
$X'_\eps u_\eps \to X'_0 u_0$ and that Propositions
\ref{pr:liminf.bulk} and \ref{pr:liminf.memb} hold. Hence, using
$\sfF(u_0)=\kappa$ we have 
\[
\alpha_0 \geq  \liminf_{\eps\to 0} \wh\mfD^\text{bulk}_\eps(u_\eps) 
 + \liminf_{\eps\to 0} \ol\mfD{}^{\Omega_0}_\eps(u_\eps) 
\geq  \wh\mfD^\text{bulk}_0(u_0) +  \wh\mfD^\mb_0(\sfT(u_0),\sfF(u_0))
= \wh\mfD_0(u_0). 
\]
Thus,  Part A of Theorem \ref{th:Main.Gamma.Cvg} is proved. 
\end{proof}

\section{The \TOS{$\bm{\Gamma}$}{Gamma}-limsup estimate  of
  \TOS{$\wh\mfD_\eps$}{whDeps}} 
\label{se:Limsup} 

This section is devoted to the proof of Part B of Theorem
\ref{th:Main.Gamma.Cvg}, which is restricted to the case where $\wh\mfD_\eps$
is a convex functional because we have to rely on the stronger concavity assumption 
\eqref{eq:ConcaveTwofold}. However, we expect that the more restrictive
assumption \eqref{eq:ConcaveTwofold} is of technical nature such that the
result is also true in more general situations. 

The advantage of resorting to convex functionals like $\wt\mfD_\mfM$ in
\eqref{eq:def.wt.mfD0} was evidently the weak lower semi-continuity which
implied the desired liminf estimates. Here we will use another advantage
arising from Jensen's inequality, which will allow us to do smoothing of general
functions $\wh u_0 \in \bfL_1$. The general structure involves an integral
functional $\mfI$ over a domain $\bbT \ti \Sigma$, where $\bbT=(\bbR/\bbZ)^d$
is a $d$-dimensional torus on which an averaging operator $\calT$ is defined
via a probability measure $\nu\in \calP(\bbT)$ via
\[
(\calT f)(x,y) = \int_{\bbT} f(x{-}\wt x,y) \dd \nu(\wt x) \quad \text{for } f\in
\rmL^1(\bbT\ti \Sigma). 
\] 
If the integrand $F$ of the integral  functional 
\[
\mfI:\rmL^1(\bbT\ti \Sigma;\R^m) \to [0,\infty], \quad z \mapsto 
\iint_{\bbT \ti \Sigma} F\big(y,z(x,y)\big)\dd x\dd y 
\]
is a normal integrand that is lower semi-continuous and convex in $z\in \R^m$ 
and \emph{does not depend on} $x \in \bbT$, then Jensen's inequality implies
$\mfI(\calT z) \leq \mfI(z)$. 

To apply the above argument, we extend our functions $\wh u_0\in \bfL_1$ to even,
periodic functions $t\mapsto \wt u_0(t)=\wh u_0(|t|) $ on $\bbT:= [-T,T]$,
where the two endpoints are identified. Moreover, we choose a
non-negative mollifier $\chi_1 \in \rmC^1_\rmc(\R;[0,1])$  with $\int_\R \chi_1(s)
\dd s=1$ and $\mafo{sppt}(\chi_1)\subset [-1,1]$ and set
$\chi_\delta(t)=\chi_1(t/\delta) /\delta $. Recalling $\Omega =[0,T] \ti I$, we
define the mollifiers 
\[
\calT_\delta:  \left\{  \ba{ccc} 
 \rmL^1(\Omega)&\to &\rmL^1(\Omega),\\
 \wh u_0&\mapsto & \big(\chi_\delta * \wt u_0\big) \big|_{\Omega}. 
\ea \right. 
\]

Clearly, our dissipation functionals $\wh\mfD_\eps$ do not explicitly depend on
$t\in [0,T]$, and they are symmetric under the time reflection
$u_0(\cdot) \mapsto u_0(T{-}\,\cdot\,)$, because the flux $Q$ appears quadratic
in $\wh\mfD_\eps^\text{flux}$ and linearly in the continuity
equation. \AAA As the limiting continuity equation is time-independent, \EEE 
applying the convolution $\calT_\delta$ to $\wh u_0$ and the 
associated flux $\wh Q_0$ we see that the pair
$\big(\calT_\delta \wh u_0, \calT_\delta \wh Q_0\big) $ still satisfies the
limiting continuity equation. \AAA Moreover, from $ \wh u_0\in
\rmL^1\big([0,T];\rmC^1(I)\big)$ we deduce $\calT_\delta \wh u_0 \in \rmC^1
\big( [0,T];\rmC^0(I)\big)$, and after, exploiting the limiting continuity
equation, we find $ \calT_\delta \wh Q_0 \in \rmC^0 \big(
[0,T];\rmC^1(I)\big)$. Additionally we also have $m_0(t):=\int_I X'_0(y)
\calT_\delta \wh u_0(t,y) \dd y =1$ for all $t \in [0,T]$, because we first observe
$\iint_\Omega   X'_0(y) \calT_\delta \wh u_0\dd y \dd t = T$ as convolution
conserves the mass, and secondly we know that $m_0$ is constant
along all solutions of the limiting  continuity equation. \EEE

With this and the stronger concavity condition
\eqref{eq:ConcaveTwofold}, we obtain a first tool for our proof, namely
Jensen's inequality for $\wh\mfD_0$:
\begin{equation}
  \label{eq:Jensen.mfD0}
  \wh\mfD_0(\calT_\delta \wt u_0)\leq \wh\mfD_0(\wt u_0) \quad \text{for all } u
  \in \bfL_1.
\end{equation}

\subsection{Approximation and flux corrections}
\label{su:ApproxFluxCorr}

We now have to construct $\wh u_\eps$ for a given $\wh u_0\in \bfL_1$ with
$\wh\mfD_0(\wh u_0)<\infty$. The construction has to be such that we can insert
it into $\wh\mfD_\eps$ where the membrane part is explicitly given as an
integral, in contrast to $\wh\mfD_0$, where only the traces
$U^\PM=\sfT(\wh u_0)$ and the flux $\kappa= \sfF(\wh u_0)$ are relevant. Even
under the weaker concavity assumption \eqref{eq:Cvx.Ass.psi} we know that there
exists a unique minimizer $u_\mafo{min}^{u_\PM,\kappa} \in \bfP(u_\PM)$ for the
functional $u\mapsto \bfM(u,\kappa) $. 

\SSS Recall also that the $\Gamma$-convergence in Theorem \ref{th:Main.Gamma.Cvg}
only asks for $X'_\eps \wh u_\eps \weak X'_0 \wh u_0$,  thus we are free to
change $\wh u_0$ on $\Omega_0 = [0,T]\ti I_0$. Thus, from now on, we \EEE
assume that $\wh u_0$ is such that
\begin{equation}
  \label{eq:u.min.choice}
  \Big(\ \wh u_0(t,y) = u_\mafo{min}^{U^\PM(t),\kappa(t)} (y) \ \text{ for } y \in
I_0\ \Big) \quad \text{for a.a.\ } t\in [0,T].
\end{equation}
\SSS This choice is crucial for the construction of the recovery sequence,
because we do not need to recover an arbitrary $\wh u_0$ but only those $\wh
u_0$ that are already optimally chosen inside of the membrane part $\Omega_0=
[0,T] \ti I_0$. \EEE  

With this choice, we now define our recovering sequences for \AAA
$\eps\in [0,1/2]$ \EEE using small regularizing parameters
$\delta, \,\theta \in [0,1]$ via
\begin{equation}
  \label{eq:RecoverSeq}
  \wt u^\eps_{\delta,\theta} (t,y) = \AAA c_\eps(t) \EEE \Big(
  \theta\,\tdfrac12 +(1{-}\theta)\,     (\calT_\delta \wt u_0)(t,y) \Big).   
\end{equation}
 Later, the parameters $\delta$ and
$\theta$ will be chosen as functions of $\eps$ such that
$(\delta,\theta)=(\delta(\eps),\theta(\eps)) \to (0,0)$ for $\eps \to 0$,
which then implies $ X'_\eps \wt u^\eps_{\delta(\eps),\theta(\eps)} \weaks
X'_0\wt u_0$ in $\mfM(\ol\Omega)$ as desired.

The parameter $\theta>0$ lifts the function away from $u=0$ providing the
strictly positive lower bound $ \wt u_{\delta,\theta} (t,y)\geq \theta/2$, and
$\delta>0$ provides temporal smoothness. \AAA The prefactor $c_\eps(t)$ is
chosen such that 
\begin{equation}
  \label{eq:wtu.eps.normalized}
  1= \int_I X'_\eps(y) \wt u^\eps_{\delta,\theta}(t,y) \dd y = c_\eps(t) \int_I
     X'_\eps(y) \wt u^0_{\delta,\theta}(t,y) \dd y 
   \quad \text{for all } t \in [0,T].
\end{equation}
Note that $c_\eps$ also depends on $(\delta,\theta)$ but
$c_\eps:=c_{\eps.\delta,\theta}$ lightens the notation. Clearly, we have
$c_0=1$ and using $X'_0(y) -X'_\eps(y) = \eps g(y)$ with $g(y)= 0$ for $y<0$, 
$g(y)=-1$ for $y\in {]0,1[}$, and $g(y)=1$ for $y>1$  as well as $\wt
u^0_{\delta,\theta} \geq 0$, we find 
\[
\frac1{1+\eps\| \wt u^0_{\delta,0}\|_{\rmL^\infty}} \leq c_\eps(t) \leq
\frac1{1-\eps}   \quad \text{for } t \in [0,T]. 
\]
Thus, we have $| c_\eps(t) -1|\leq C_\delta \eps$ and differentiating
\eqref{eq:wtu.eps.normalized} with respect to $t\in [0,T]$, the temporal
smoothness of $\wt u_{\delta,\theta}$ gives  $|\dot c_\eps(t) |\leq C_\delta
\eps$. 
 
The following result shows that the associated flux $Q^\eps_{\delta,\theta}$
can be controlled as a perturbation of the flux $\wh Q_0$ for $\wh u_0$.  
\EEE

\begin{lemma}[Recovered continuity equation]
\label{le:RecovCE} 
Let $  \wt u^\eps_{\delta,\theta}$ be as above,
let $\wh Q_0$ be the optimal flux associated with $\wh u_0$,  \AAA and set
\[
\wt Q^0_{\delta,\theta}:= (1{-}\theta) \calT_\delta \wh u_0 \ \text{ and } \
\wt Q^\eps_{\delta,\theta} (t,y) = \wt Q^0_{\delta,\theta}(t,2) + \int_y^2\!   
X'_\eps(z) \big( \dot c_\eps(t) \wt u^0_{\delta,\theta}(t,z) {+} c_\eps(t)\pl_t
\wt u^0_{\delta,\theta}(t,z)\big) \dd z .
\]
Then, for $\eps,\delta,\theta \in [0,1/2]$, the couple $(\wt
u^\eps_{\delta,\theta}, \wt Q^\eps_{\delta,\theta})$ satisfies the continuity 
equation \eqref{eq:rho.eps.CE} in the sense of distributions:
\begin{equation}
  \label{eq:Recov.CE}
 \iint_\Omega \!\! \Big( X'_\eps \wt u^\eps_{\delta,\theta}\,\pl_t \phi  +
Q^\eps_{\delta,\theta} \pl_y \phi\Big) \dd y \dd t =0 \quad 
\text{ for all } \phi \in \rmC_\rmc^1\big( {]0,T[}\ti [-1,2]\big). 
\end{equation}
Moreover, for all $\delta>0$ there exists $C_\delta$ such that
$\|  \wt Q^\eps_{\delta,\theta} - \wt Q^0_{\delta,\theta}
\|_{\rmL^\infty(\Omega)}\leq \eps C_\delta$ for all $\eps,\theta \in
[0,1/2]$. \EEE
\end{lemma}
\begin{proof}
  \AAA Because of the smoothness of $\wt u^0_{\delta,\theta}$ in $t \in [0,T]$,
  equation \eqref{eq:wtu.eps.normalized} can be differentiated in $t$. The
  obtained relation implies
  $ \wt Q^\eps_{\delta,\theta}(t,-1) = \wt Q^\eps_{\delta,\theta}(t,2) $. Thus,
  we can integrate by parts with respect to $y$ in
\begin{align*}
&  \iint_\Omega \!\! \Big( X'_\eps \wt u^\eps_{\delta,\theta}\,\pl_t \phi  {+}
\wt Q^\eps_{\delta,\theta} \pl_y \phi\Big) \rmd y \dd t = 
 \iint_\Omega \!X'_\eps\Big( \wt u^\eps_{\delta,\theta}\,\pl_t \phi  {+}
 \big( \dot c_\eps \wt u^0_{\delta,\theta} {+} c_\eps\pl_t
\wt u^0_{\delta,\theta}\big) \,\phi  \Big) \rmd y \dd t 
\\
& =  \iint_\Omega \!X'_\eps \pl_t \big( c_\eps \wt
u^0_{\delta,\theta}\,\phi\big) \dd y \dd t = 0 ,
\end{align*}
where for the last step one uses integration by parts in $t$ and
$\phi(0)=\phi(T)=0$. Thus, \eqref{eq:Recov.CE} is established. 

For the last statement we observe that the continuity equation for 
$\big(  \wt u^0_{\delta,\theta}, \wt Q^0_{\delta,\theta}\big) $ implies 
$\wt Q^0_{\delta,\theta} (t,2) + \int_y^2 X'_0 \pl_t \wt u^0_{\delta,\theta}(t,z) \dd
z = \wt Q^0_{\delta,\theta} (t,y)$. Thus, we have
\[
\wt Q^\eps_{\delta,\theta}(t,y) - \wt Q^0_{\delta,\theta}(t,y) =\int_y^2 \Big( X'_\eps
\dot c_\eps \wt u^0_{\delta,\theta} + (X'_\eps{-}X'_0)  c_\eps \pl_t\wt
u^0_{\delta,\theta} +X'_0 \,(c_\eps{-}1) \pl_t\wt u^0_{\delta,\theta} \Big) \dd z  . 
\] 
Using $|\dot c_\eps|+| c_\eps{-}1| \leq C\eps$, $ X'_\eps{-}X'_0=\eps g$ with
$\|g\|_{\rmL^\infty}=1$, and $\wt u^0_{\delta,\theta}\in
\rmC^1\big([0,T];\rmC^0(I)\big) $, the final assertion  
$\|  \wt Q^\eps_{\delta,\theta} - \wt Q^0_{\delta,\theta}
\|_{\rmL^\infty(\Omega)} \leq \eps C_\delta  $ follows. \EEE
\end{proof}

For showing convergence of $\wh\mfD_\eps(\wh u_\eps)$ we have to control the
error in $\wh\mfD_\eps^\text{flux}$ if the optimal flux $\wt Q$ is replaced by 
another flux $Q$ also satisfying the continuity equation. 

\begin{lemma}[Influence of different fluxes]
\label{le:DifferentFlux}
Consider two pairs $(u,Q)$ and $(u,\wt Q)$ solving the continuity equation
\eqref{eq:rho.eps.CE} for $\eps\in [0,1/2]$ with optimal flux $\wt Q$, i.e.\
the  minimizer  
in $\wt\mfD_\eps(u)= \inf_Q \SSS \iint_\Omega \EEE Q^2/(2\bbA_\eps\mob(u)) \dd y \dd t$.  
Then, the optimal flux $\wt Q$ is given by $\wt Q(t) = \AAA \bbP_{u(t)}^\eps
\EEE  Q(t)$  a.e.\ on $[0,T]$, where
\[
 \bbP^\eps_u Q:= \left\{ \ba{cl}\ds Q- \mathsf H_\eps(u)\int_I \tdfrac{Q  \dd
     y}{\bbA_\eps\mob(u)} 
  & \text{if } \int_I1/\mob(u) \dd y <\infty, \\[0.8em]
 Q &  \text{if } \int_I1/\mob(u) \dd y =\infty,  \ea \right. \text{ with }
\mathsf H_\eps(u):=\Big(\int_I\tdfrac{\rmd y}{\bbA_\eps \mob(u)}\Big)^{-1} . 
\]
In particular, if $\int_I1/\mob(u(t)) \dd y=\infty$, then we automatically have
$Q(t)=\wt Q(t)$ a.e.  
\end{lemma}
\begin{proof}
Subtracting the two continuity equations, we immediately see that
$\pl_y(Q{-}\wt Q)=0$, i.e.\ there exists a function $a:[0,T]\to \R$ such that
$\wt Q(t,y)=Q(t,y)-a(t)$ a.e.\ on $\Omega$. 

Thus, the minimization is only over the scalar function $t\mapsto a(t)$ in the
functional 
\[
\calF(a):= \iint_\Omega \frac{( \SSS Q(t,y){-} a(t) \EEE )^2}
    {2m(t,y)}  \dd y \dd t, \quad \text{where }m= \bbA_\eps(y)\mob(u).
\] 
In the case $\calF(a)<\infty$ we have
$\int_I(Q(t,y){-}a(t))^2/(2m(t,y))\dd y < \infty$ a.e.\ and can minimize
in $\alpha= a(t)$ for each $t\in [0,T]$ separately.

For the quadratic function $\alpha \mapsto
F(\alpha)=\int_I(Q{-}\alpha)^2/(2m)\dd y\in [0,\infty] $ there are three
cases: 
\\
(i) $F\equiv$ is identically $+ \infty$, \\
(ii)  there exists a
unique $\alpha_*\in \R $ such that $F(\alpha_*)< \infty$ and
$F(\alpha)=+\infty$ otherwise, and 
\\
(iii) $F(\alpha) < \infty$ for all $\alpha \in \R$. \medskip

Because of $\calF_\eps(a)< \infty$, case (i) can only occur on a null set, and
hence can be ignored. Case (ii) corresponds to the case $ \int_I1/m \dd
y=\infty$ which implies $Q(t)=\wt Q(t)$. In case (iii) the unique minimizer
$\alpha_\mafo{min}$ is characterized by $\int_I(Q{-}\alpha_\mafo{min})/m \dd y
\SSS =0$ which defines $\bbP_u^\eps$.
\end{proof}

\subsection{Construction of a recovery sequence} 
\label{su:Limsup.Recov}

We are now ready construct the recovery sequence $\wt u_\eps$.
We start with a general $\wh u_0\in \bfL_1$ with $\wh\mfD_0(\wh u_0) < \infty$
and proceed in four steps:

\STEP{ I. Positivity}: Set $\wt u_{0,\theta} := \theta \frac12 +
(1{-}\theta)\wh u_0$ and show $\wh\mfD_0(\wt u_{0,\theta})\to \wh\mfD_0(\wh
u_0)$ as $\theta\to 0$. 

\STEP{II. Smoothing}: Fix $\theta>0$, set $\wt u_{\delta,\theta} := \calT_\delta
\wt u_{0,\theta}$, and show $\wh\mfD_0(\wt u_{\delta,\theta})\to \wh\mfD_0(\wt
u_{0,\theta})$ as $\delta\to 0$.

\STEP{III. Positive thickness $\eps$}: Fix $\theta,\delta>0$ and show
$\wh\mfD_\eps(\wt u_{\delta,\theta})\to \wh\mfD_0(\wt 
u_{\delta,\theta})$ as $\eps\to 0$.

\STEP{IV. Recovery sequence}: Show $\wh\mfD_\eps(\wt
u_{\delta(\eps),\theta(\eps)}) \to \wh\mfD_0(\wh u_0)$ for suitable $
(\delta(\eps),\theta(\eps))\to (0,0)$.\bigskip 

For the proof of these steps we use the convex functional 
\[
\sfD_\eps(u,Q)= \iint_\Omega \Big( \frac{Q^2}{2\bbA_\eps\mob(u)} +
\frac{\bbA_\eps}2 \big|\pl_y H(u)\big|^2 \Big) \dd y,
\]
such that
$\wh\mfD_\eps(\wt u)= \sfD_\eps(\wt u, \wt Q) \leq \sfD_\eps(\wt u, Q)$ if
$(\wt u,\wt Q)$ and $(\wt u,Q)$ solve the continuity equation
\eqref{eq:rho.eps.CE} and $\wt Q$ is the minimizer. \AAA The convergence
$\sfD_\eps \to \sfD_0$ is trivial because of
$\| \bbA_\eps{-}\bbA_0\|_{\rmL^\infty(I)} =O(\eps)$, all the difficulties arise
through the degeneracy of $X'_\eps$ in the continuity equation. \SSS Moreover,
note that $\wh\mfD_\eps$ as well as $\sfD_\eps$ contain the membrane part
$\Omega_0= [0,T] \ti I_0 $. This is particularly important for $\eps=0$ where
$\wh\mfD_0$ is defined in \eqref{eq:mfD0limit} via $\ol M$ that matches the
choice of $\wh u_0$ made in \eqref{eq:u.min.choice}.\medskip \EEE

\noindent
\begin{proof}[Proof of Step I] To lighten the notation we set
  $u_\theta := \wt u_{0,\theta}$ and recall that
  $(u_\theta,(1{-}\theta)\wh Q_0)$ solves the limiting continuity equation
  \eqref{eq:LimitCE.3} and that Jensen inequality can be applied, which gives
\[
\wh\mfD_0(\wt u_\theta)  = \sfD_0(u_\theta, \wt Q_\theta)
\leq  \sfD_0(u_\theta, (1{-}\theta) \wh Q_0)
\overset{\text{Jensen}}\leq  \sfD_0 (\wh u_0, \wh Q_0) = \wh\mfD_0(\wh u_0),
\]
where $\wt Q_\theta$ are the optimal fluxes of $\wt u_{0,\theta}$. 

We clearly have $(u_\theta,(1{-}\theta)\wh Q_0) \to (\wh u_0,\wh Q_0)$ in
$\rmL^1(\Omega)$. Hence, the lower semi-continuity of the convex functional
$\sfD_0$ implies $\liminf_{\theta\to 0} \AAA \sfD_0 \EEE ( u_\theta, 
(1{-}\theta) \wh Q_0) \geq \sfD_0 (\wh u_0, \wh Q_0)= \wh\mfD_0(\wh u_0)$. 
Thus, we have established
$\sfD_0(u_\theta, (1{-}\theta) \wh Q_0) \to \wh\mfD_0(\wh u_0)$, and it remains
to show that
\[
0\leq g_\theta:=\sfD_0(u_\theta,(1{-}\theta) \wh Q_0) - \sfD_0(u_\theta,\wt Q_\theta)
 = \iint_\Omega \Big( \frac{(1{-}\theta)^2\wh Q_0^2
  -  \wt Q_\theta^2}{2\bbA_0\mob(u_\theta)} \Big) \dd y \dd t \ \to \  0.  
\]
By Lemma \ref{le:DifferentFlux} we have $\wt Q_\theta =(1{-}\theta)
\bbP^0_{u_\theta} \wh Q_0$, and using $\wh Q_0=\bbP^0_{\wh u_0} \wh Q_0$
and $ u_\theta \to \wh u_0$ we obtain $\wt Q_\theta\to \wh Q_0$ a.e.\ in $\Omega$. 

Writing $m_\theta:=\bbA_0 \mob(u_\theta)$ and $h_\theta(t) := 1/\int_I
1/m_\theta(t,x)\dd y$ we can rewrite $g_\theta$ in the form 
\begin{align*}
\frac{g_\theta}{(1{-}\theta)^2} & 
= \iint_\Omega \frac{\wh Q_0^2- (\bbP^0_{u_\theta}\wh Q_0)^2}{2 \,m_\theta} 
  \dd y \dd t = \int_0^T \frac{h_\theta(t)}2 
\Big(\int_I \wh Q_0/m_\theta \dd y\Big)^2 \dd t
\\
& \leq \iint_\Omega \frac{\wh Q_0^2}{m_\theta} \dd y \dd t \leq
 \frac1{1{-}\theta} \iint_\Omega \frac{\wh Q_0^2}{\bbA_0\mob(\wh u_0)} \dd y
 \dd t \leq  \frac{\wh\mfD_0(\wh u_0)}{1{-}\theta} < \infty. 
\end{align*}
The identities in the first line follow because the projection
$\bbP^0_{u_\theta}$ is orthogonal with respect to the quadratic form $Q \mapsto
\iint Q^2/m_\theta \dd y \dd t$. In the second line, the first ``$\leq$''
follows from a simple Cauchy-Schwarz estimate, and 
the second ``$\leq$'' uses the concavity of $\mob$ (see
\eqref{eq:ConcaveTwofold}) giving $\mob(u_\theta) 
\geq \theta \mob(1/2)+(1{-}\theta)\mob(\wh u_0)\geq (1{-}\theta)\mob(\wh u_0)$. 
As the first ``$\leq$'' provides a pointwise estimate in $ t\in [0,T]$ of the
integrands, we can apply Lebesgue's dominated convergence theorem and find
$g_\theta\to 0$ as desired. 
\end{proof}
\vspace{0.3em}

\noindent
\begin{proof}[Proof of Step II] We now fix $\theta>0$ and use the notations 
$\ol u_\delta:=\wt u_{\delta,\theta} = \calT_\delta u_\theta$ and $m_\delta =
\bbA_0 \mob( \ol u_\delta)$. Using that $(\ol u_\delta, \calT_\delta \ol Q_0)$
with $\ol Q_0=\wt Q_\theta$ with $\wt Q_\theta$ from Step I again satisfies the
limiting continuity equation \eqref{eq:LimitCE.3} and that $\calT_\delta$
allows for the application of Jensen's inequality, we can repeat the same proof 
and it remains to show that 
\[
\ol g_\delta:= \sfD_0(u_\delta, \calT_\delta \ol Q_0) -
\sfD_0(u_\delta,\bbP^0_{\ol u_\delta} \calT_\delta \ol Q_0) \ \to \ 0 \quad
\text{for } \delta \searrow 0.
\]
This follows exactly as in Step I from $(\ol u_\delta,\calT_\delta \ol Q_0)\to
(\ol u_0, \ol Q_0)$ in $\rmL^1(\Omega)$ and $\bbP^0_{\ol u_0} \ol Q_0=\ol
Q_0$. 
\end{proof}
\vspace{0.3em}

\noindent
\begin{proof}[Proof of Step III] We are now ready to treat the case
$\eps>0$. For this, we keep fix $\theta,\,\delta>0$ and set  $u^*=\ol u_\delta
= \wt u_{\delta,\theta}$ which satisfies $u^*\geq \theta/2>0$ a.e.\ in $\Omega$
and $u^* \in \rmC^1([0,T]; \rmL^1(\Omega))$. Moreover, let $Q^*$ be the optimal
flux in the limiting continuity equation \eqref{eq:LimitCE.3}, satisfying
$\wh\mfD_0(u^*)=\sfD_0(u^*,Q^*)$. 

Lemma \eqref{le:RecovCE} states that $(u^*,Q^*_\eps)$ with $Q^*_\eps = Q^*+ \eps
J_*$ solves the continuity equation \eqref{eq:rho.eps.CE} for $\eps>0$, where 
$J_*(t,y)=\int_0^y\pl_t u^*(s,\ol y)\dd \ol y$ satisfies $J_*\in
\rmL^\infty(\Omega)$.  As in Steps I and II it remains to show that 
\[
\wh\mfD_\eps(u^*) = \sfD_\eps(u^*, \bbP^\eps_{u^*} Q^*_\eps)\ \to \ 
\sfD_0(u^*,Q^*) = \wh\mfD_0(u^*).
\]
Because $\sfD_\eps$ and $\sfD_0$ only differ by the trivial dependence of
$\bbA_\eps$ converging uniformly to $\bbA_0$, it suffices to show that 
$g^*_\eps := g^{(1)}_\eps + g^{(2)}_\eps \to 0$, where
\[
g^{(1)}_\eps:=\sfD_\eps( u^*,Q^*)- \sfD_\eps(u^*,\bbP^\eps_{u^*} Q^*) 
\text{ and } 
g^{(2)}_\eps:=\sfD_0( u^*, \bbP^\eps_{u^*} Q^* )- \sfD_0(u^*,\bbP^\eps_{u^*}
Q^*_\eps ).
\] 
The convergence $g^{(1)}_\eps \to 0$ follows as in Steps I and II, because of
$\bbP^0_{u^*} Q^*=Q^*$. For estimating $g^{(2)}_\eps$ we can use $Q^*_\eps =
Q^*+ \eps J_*$ with $J_*\in \rmL^\infty(\Omega)$ and $m_*:= \mob(u^*)\geq 
\mob(\theta/2) \geq \theta \mob(1/2)>0$ because of $u^*\geq \theta/2$ and the
concavity of $\mob$. In particular, we have 
\[
\big|\bbP^0_{u^*}Q^*_\eps - \bbP^0_{u^*}Q^*_0\big| = \eps |\bbP^0_{u^*}J_*|
\leq  \eps\, C_J \quad \text{with } C_J:= 2 \| J_*\|_{\rmL^\infty(\Omega)} < \infty
\] 
and can estimate as follows:
\begin{align*}
|g^{(2)}_\eps| &\leq \iint_\Omega 
\frac{ \big| (\bbP^\eps_{u^*} Q^*)^2  - \SSS (\bbP^\eps_{u^*} \EEE Q_\eps^*)^2\big| } 
    {2\bbA_0 m_*} \dd y \dd t 
\leq  \iint_\Omega \frac{  \eps | \bbP^0_{u^*}J_*|\big(
2|\bbP^0_{u^*} Q^*| + \eps| \bbP^0_{u^*}J_*| \big) }{2\bbA_0 m_*} \, \dd y \dd t
\\
&\leq 
\eps \Big( K_*\iint_\Omega \frac{(\bbP^0_{u^*} Q^* )^2}{2\bbA_0 m_*} \dd y\dd t
\Big)^{1/2} + \eps^2 K_* \quad \text{with } K_*=\frac{6T C_*^2}{\theta
  \mob(1/2)}  \big\|\tdfrac{1}{\bbA_0}\big\|_{\rmL^\infty(\Omega)} < \infty.
\end{align*}
Thus, $g^{(2)}_\eps \to 0$ follows because of $\sfD_0(u^*,\bbP^0_{u^*} Q^* )
\leq \wh\mfD_0(\wh u_0)< \infty$. 
\end{proof}
\vspace{0.0em}

\noindent
\begin{proof}[Proof of Step IV] For $n\in \N$ we apply Steps I
  to III using a standard diagonalization argument. 
By Step I we choose $\theta_n$ satisfying
\[
\theta_n \in {]0,1/n[}, \quad \big| \wh\mfD_0(\wt u_{0,\theta_n}) -
\wh\mfD_0(\wh u_0) \big| \leq 1/(3n), \quad \| \wt u_{0,\theta_n} - \wh
u_0\|_{\rmL^1(\Omega)} \leq 1/(2n).
\]
Using Step II, we choose $\delta_n$ with 
\[
 \delta_n \in {]0,1/n[}, \quad \big| \wh\mfD_0(\wt u_{\delta_n,\theta_n}) -
\wh\mfD_0(\wt u_{0,\theta_n}) \big| \leq 1/(3n), \quad 
\| \wt u_{\delta_n,\theta_n} - \wt u_{0,\theta_n} \|_{\rmL^1(\Omega)} \leq 1/(2n).
\]
Using Step III we find $\eps_n$ with 
\[
 \eps_n \in {]0,1/n[}  \qand \big| \wh\mfD_\eps(\wt u_{\delta_n,\theta_n}) -
\wh\mfD_0(\wt u_{\delta_n,\theta_n}) \big| \leq 1/(3n).
\]
Thus, we have found the recovery sequence $(\eps_n,\wh u_n)_{n\in \N}$, where
$\wh u_n := \wt u_{\delta_n,\theta_n}$,  which satisfies 
\[
\eps_n \to 0, \quad  \big\| \wh u_n - \wh u_0 \big\|_{\rmL1(\Omega)} \leq 1/n
\to 0, \quad  \big| \wh\mfD_{\eps_n}(\wh u_n) - \wh\mfD_0(\wh u_0)\big| \leq 1/n
\to 0. 
\]
This finishes the proof of Part B of Theorem \ref{th:Main.Gamma.Cvg}.  
\end{proof}

\section{NESS and interpretation of \TOS{$\bfR_\mb $}{bfRmb} via BER functions}
\label{se:NESS.Rmemb} 

\subsection{NESS as saddle points of BER functions}
\label{su:SaddleNESS}

We use here the characterization of Non-Equilibrium Steady States (NESS) 
via the saddle-point theory developed in \cite{Miel23NESS, Miel25?PGSN}.  The main tool is the
so-called BER functional for a generalized gradient system $(\bbX,\scrE,\scrR)$,
which is given by 
\[
\scrB_{\scrE,\scrR}(u,\xi):= \scrR^*(u,\xi)- \scrR^*(u,{-}\rmD\scrE(u)). 
\]
The system will be connected to the environment modeled by the port space $\bbY$
via the so-called port mapping $\mfP^*$ (constraining the thermodynamical driving
\SSS    forces \EEE in $\bbX^*$, see \cite{Miel25?PGSN}) 
\[
\mfP^*: \bbX^* \to \bbY^*; \ \xi\mapsto \eta=\mfP^*\xi.   
\]
We assume that $\mfP^*$ is surjective, and hence $\mfP:\bbY\to \bbX$ is injective.
We further assume there is a mapping $\mfT:\bbX\to \bbY$ and a reduced energy 
$\ol E: \bbY\to \R$ such that $\mfP^*\rmD\scrE(u)=\rmD\ol E(\mfT u)$. 

Following \cite[Sec.\,II]{Miel23NESS} we can now define the \emph{reduced
  B-function}  
\[
B_\mafo{red}(y,\eta):= \sup_{u:\, \mfT u=y} \; 
  \inf_{\xi:\,\rule{0pt}{0.65em}\mfP^*\xi =\eta}  \scrB_{\scrE,\scrR}(u,\xi).
\]
The theory relies on the fact that for \AAA each $y \in \bbY$ \EEE there exists a
\emph{global null saddle} $(u,\xi) = \big(U(y),-\rmD\scrE(U(y))\big) $, i.e.\ 
\begin{align}
\label{eq:GlobalSaddle}
\forall\, \wh u\text{ with } & \mfT \wh u= y\ \forall \,\xi \text{ with }
\mfP^*\wh \xi=-\mfP^*\rmD\scrE(U(y))=-\rmD\ol E(y) : \\
\nonumber
& \scrB_{\scrE,\scrR}  \big(\wh u, -\rmD\scrE(U(y)) \big) \leq 
  \scrB_{\scrE,\scrR}  \big(U(y), -\rmD\scrE(U(y)  \big) \leq 
  \scrB_{\scrE,\scrR}  \big(U(y),\wh \xi \big).
\end{align}
The solutions $u=U(y)$ are called NESS, because they are steady states of the
so-called port-gradient-flow equations (cf.\ 
\cite{Miel25?PGSN}), which are identified by the condition 
\begin{equation}
  \label{eq:NESS.Eqn}
  0=\pl_\xi \scrR(u,-\rmD\scrE(u)) - \mfP V , \quad V\in \bbY, \quad
\mfP^*\rmD\scrE(u)=\rmD\ol E(y) ,
\end{equation}
where $V \in \bbY$ is a Lagrange multiplier dual to the constraint
$\mfP^*\rmD\scrE(U(y))=\rmD\ol E(y) $, see \cite[Eq.\,(2.2)]{Miel23NESS} and 
Proposition \ref{pr:RedKinRel} below.

The general reduction theorem for BER functions shows that under suitable
conditions $B_\mafo{red}$ is again a BER function of a reduced gradient system
on $\bbY$. 

\begin{theorem}[Reduction of BER functions {\cite[Thm.\,II.16]{Miel23NESS}}]
\label{th:ReducBER} 
If for all $y \in \bbY$ there exists a unique $U(y)\in \bbX$ such that
\eqref{eq:GlobalSaddle} holds, then
$B_\mafo{red} \equiv \scrB_{\ol E,\ol\bfR_\mafo{red}}$ where the reduced
dissipation potential $\ol\bfR_\mafo{red} $ is given by \AAA its dual via \EEE
\[
{ \AAA \ol\bfR_\mafo{red}^* \EEE } 
(y,\eta) = B_\mafo{red}(y,\eta) -  B_\mafo{red}(y,0).
\]
\end{theorem}
See Proposition \ref{pr:Rmb.olM} for a proof of a closely related result.  

For later use we provide a simple formula for reduced kinetic relation, namely
$\bbY^*\ni \eta \mapsto v=\pl_\eta\ol\bfR_\mafo{red}^*(y,\eta)\in \bbY$. It
heavily relies on the property that a value function defined for critical
points has a simple derivative. More precisely, if $F:X\ti Z\to \R $ has a
(unique) solution $x=X(z)$ for the equation $\rmD_x F(x,z)=0$, then
\begin{equation}
  \label{eq:ValueFcn}
  \text{the value function } f(z):= F(X(z),z) \ \text{ satisfies }\ 
\rmD_z f(z)= \rmD_z F(X(z),z).
\end{equation}

\begin{proposition}[Reduced kinetic relation]
\label{pr:RedKinRel}
Let $\ol\bfR_\mafo{red}$ be defined as above, and assume that
$\scrB_{\scrE,\scrR}(u,\xi)$ is concave-convex in $(u,\xi)$.  For given
$(y,\eta)$ let \SSS $(\wh u_\xi,\xi_*)$ \EEE be the unique saddle point
occurring in the definition of $B_\mafo{red}(y,\eta)$. Then,
\begin{equation}
  \label{eq:Char.pl.Rred}
 \pl_\xi \AAA \scrR^* \EEE (\wh u_\xi,\xi)  = \mfP \pl_\eta
 \AAA \ol\bfR_\mafo{red}^* \EEE (y,\eta), 
\end{equation}
which characterizes $ \pl_\eta \AAA \ol\bfR_\mafo{red}^* \EEE (y,\eta)$
uniquely, as $\mfP$ is 
injective. Moreover, \AAA we have \EEE $ \pl_\xi \AAA \scrR^* \EEE (\wh u_\xi,\xi)\in
(\mafo{ker}\;\!\mfP^*)^\perp$. 
\end{proposition} 
\begin{proof}
Since $y$ is fixed, we can omit it in all the formulas. In particular, we
have 
\[
B_\mafo{red}(\eta)=\inf_{\mfP^* \xi=\eta} \scrG(\xi) \quad \text{with } 
 \scrG(\xi):= \sup_{\mfT u =y} \scrB_{\scrE,\scrR}(u,\xi). 
\]
Using the value-function identity \eqref{eq:ValueFcn} we find
$\pl_\xi \scrG(\xi)= \pl_\xi \AAA \scrR^* \EEE (\wh u_\xi,\xi)$, where $\wh
u_\xi$ is the maximizer in the definition of $\scrG(\xi)$.

To differentiate with respect to the constraint, we introduce 
$\scrQ:\bbX^*\ti \bbY^*\ti \bbY\to \R$ via 
\[
\scrQ(\xi,\eta,\lambda) := \scrG(\xi) - \big\langle \mfP^*\xi{-}\eta,\lambda
\big\rangle_\bbY. 
\]
Thus, we have the relations
$ B_\mafo{red}(\eta) = \inf_{\xi\in \bbX^*\rule{0pt}{0.7em}} \,
\sup_{\lambda\in \bbY} \scrQ(\xi,\eta,\lambda) = \sup_{\lambda\in \bbY} \
\inf_{\xi\in \bbX^*\rule{0pt}{0.7em}} \scrQ(\xi,\eta,\lambda)$. As
$ \scrQ(\xi,\eta,\lambda)$ is convex-concave in $(\xi,\lambda)$, for each
$\eta$ there exists a saddle point $(\wh \xi(\eta),\wh\lambda(\eta))$.

On the one hand, applying the  value-function identity \eqref{eq:ValueFcn} yields
\[
\rmD_\eta B_\mafo{red} (\eta) = \rmD_\eta \scrQ \big(\wh \xi(\eta), \eta, 
\wh\lambda(\eta)\big) = \wh\lambda(\eta). 
\]
On the other hand, the saddle-point property of $(\xi,\lambda)=(\wh
\xi(\eta),\wh\lambda(\eta))$ implies 
\[
0=\rmD_\xi \scrQ\big(\wh\xi(\eta), \eta,\wh\lambda(\eta)\big) 
 = \pl_\xi \scrG\big(\wh\xi(\eta)\big) - \mfP \wh\lambda(\eta).
\]
We now use that $B_\mafo{red}$ equals the
BER function $\scrB_{\ol E,\ol\bfR_\mafo{red}}$, which implies $\rmD_\eta
B_\mafo{red}(\eta) = \pl_\eta \AAA \ol\bfR_\mafo{red}^* \EEE (y,\eta)$. 
Recalling $\pl_\xi \scrG(\xi)= \pl_\xi \AAA \scrR^* \EEE (\wh u_\xi,\xi)$ from
the beginning of the proof, we conclude 
\[
\pl_\xi \AAA \scrR^* \EEE (\wh u_\xi,\xi) =\mfP \wh\lambda(\eta) 
 = \mfP \rmD_\eta B_\mafo{red} (\eta) = \mfP \pl_\eta 
 \AAA \ol\bfR_\mafo{red}^* \EEE (y,\eta),
\]
which is the desired result.
\end{proof}

\subsection{Reduction of BER functions for the membrane problem}
\label{su:RedMembFormulas}

We now apply the abstract theory from the previous subsection to our membrane
problem. Section \ref{su:Rescale} provides the unique NESS via 
$u_\mafo{NESS}(y) = G^{-1}\big((1{-}y)G(u_-){+} y G(u_+)\big)$, see
\eqref{eq:ExplicitNESS}.

To facilitate the identification of the abstract results from above
with the rest of the paper, we use the notation $\bfB=\scrB_{\calR,\calE}$ and 
$\ol B=B_\mafo{red}$, which take the explicit form 
\begin{align*}
&\bfB(u,\xi)= \int_0^1 \!\Big( \frac{\mob(u)}2 |\pl_y\xi|^2 
  - \frac{\mob(u)}2 \big| E''(u) \pl_yu\big|^2 \Big) \dd y, 
& \ol B(u_\PM,\xi_\PM)= \!\max_{u\in \bfP(u_\PM)} \min_{\xi\in \bfD(\xi_\PM)}
      \!\bfB(u,\xi).
\end{align*}
The subset $\bfP(u_\PM) \subset \rmW^{1,1}([0,1])$ is defined in \eqref{eq:RMJ.d} and 
\[
\bfD(\xi_\PM):=\bigset{ \xi\in \rmW^{1,1}([0,1]) }{ \xi(0)= \xi_-,\ \xi(1)=\xi_+}.
\]

The constraints on $u$ and $\xi$ can be realized by the port mapping 
\[
\mfP^*: \bbX:=\rmW^{1,1}([0,T])\to \bbY:=\R^2;\ \xi\mapsto
   (\xi(0),\xi(1)).
\]
Moreover, we define the reduced energy $\ol E(u_\PM) = E(u_+)+E(u_-)$, and see
that the constraints $u\in \bfP(u_\PM)$ and $\xi\in \bfD(\xi_\PM)$  
can be written in terms of the port mapping as
\[
\mfP^* \rmD\scrE(u) = \rmD \ol E(u_\PM) \quad \text{and} \quad \mfP^* \xi=
\xi_\PM.
\]
Hence, the general reduction theory \SSS of Theorem \ref{th:ReducBER} 
guarantees that the reduced functional $\ol B$ \EEE is again a BER functional for a
reduced dissipation potential, \SSS which defines $\bfR^*_\mb$:\EEE

\begin{corollary}[Derivation of $\bfR^*_\mb$ from $\ol B$]
\label{ex:Exist.Rmb} Let $E$ and $\mob$ satisfy the assumption
\eqref{eq:AllAssump} and \eqref{eq:Cvx.Ass.psi}. Then, $\ol B$ is equal to the
BER function $\scrB_{\ol E,\bfR_\mb}$ where the dual dissipation potential
$\bfR^*_\mb:{[0,\infty[}^2\ti \R \to \R $ for the membrane is given by 
\[
\bfR^*_\mb(u_\PM,\xi_+{-}\xi_-):= \ol B(u_\PM,\xi_\PM)- \ol B(u_\PM,0).
\]
Hence $\bfR^*_\mb$ depends on $\xi_\PM$ only via the difference 
$\delta=\xi_+{-}\xi_-$. Moreover, the transmission condition reads (for $G$ see
\eqref{eq:ExplicitNESS}) 
\begin{equation}
  \label{eq:GenTransCond}
 \pl_\delta \bfR^*_\mb \big(u_\PM,E'(u_-){-} E'(u_+)\big) =  G(u_+) -G(u_-) .
\end{equation}
\end{corollary}
\begin{proof} The fact that $\ol B$ is a BER function immediately follows from
  Theorem \ref{th:ReducBER} if we use the transformation $u=\psi(u)$ from
  \eqref{eq:Cvx.Ass.psi}, because $(w,\xi) \mapsto \bfB(\psi(w),\xi)$ is
  concave-convex, and hence $w=\psi(u_\mafo{NESS})$ are global null saddles.
The fact that $\bfR_\mb$ depends only on $\delta=\xi_+{-}\xi_-$ follows from
$\bfB(u,\xi{+}c)=\bfB(u,\xi)$ for every constant $c$. 

For the transmission condition we employ Proposition \ref{pr:RedKinRel}. If $u$
is the NESS associated with $u_\PM$, i.e.\ $(u,-E'(u))$ is the associated null
saddle, we can test the abstract relation \eqref{eq:Char.pl.Rred} by a smooth
test function $\phi\in \rmC^1([0,1])$ and obtain (using $G'(u)= \mob(u)E''(u)$)
\begin{align*}
& \int_0^1\pl_y\phi \,\pl_y G(u) \dd y 
= \int_0^1 \mob(u)\pl_y\phi \pl_y\big(E'(u)\big) \dd y 
= \big\langle  \phi,\pl_\xi \calR^*(u,-E'(u)) \big\rangle 
\\
&\quad 
=\langle \phi, \mfP\pl_{\xi_\pm} \SSS \bfR^*_\mb \EEE \big(u_\PM,-(E'(u_-), 
   \BBB E'(u_+) \EEE\big) \big\rangle  
= \langle  \mfP^*\phi,\pl_{\xi_\PM} \bfR^*_\mb\big(u_\PM ,-(E'(u_-), 
  \BBB E'(u_+) \EEE \big) \big\rangle  
\\
&\quad =  \big(\phi(1){-}\phi(0)\big)\, \pl_\delta
\bfR^*_\mb\big(u_\PM , E'(u_-)-E'(u_+)\big) . 
\end{align*}
Now inserting $u=u_\mafo{NESS}: y \mapsto  G^{-1}\big((1{-}y)G(u_-){+}
y G(u_+)\big)  $ into the left-hand side gives 
\[
\int_0^1 \!\!\pl_y \phi \pl_y 
G(u_\mafo{NESS})\dd y  = \phi \pl_y G(u_\mafo{NESS}) \big|^1_0 - \!\int_0^1 \!\!\phi
\pl_y^2 G(u_\mafo{NESS}) \dd y =\big(\phi(1){-}\phi(0)\big)\;\! 
  \big( G(u_+){-}G(u_-)\big), 
\]
because $y \mapsto \pl_y G(u_\mafo{NESS}(y)$ is affine. 
With this the last assertion is established.  
\end{proof}

\SSS
\subsection{Deriving \texorpdfstring{$\bfR_\mb$}{Rmb} directly from
  \texorpdfstring{$\ol M$}{olM}} 
\EEE

While the NESS theory starts from the BER function and derives the representation
$\ol B(u_\PM,\delta)= \bfR^*_\mb(u_\PM,\delta)-\bfR^*_\mb \big(
u_\PM,E'(u_-){-}E'(u_+)\big)$, our theory of EDP-convergence requires the dual
approach. Starting from the membrane potential $\bfM$ and its reduction $\ol M$
given by 
\[
\bfM(u,Q)= \int_0^1 \Big( \frac{Q^2}{2\mob(u)} + \frac{\mob(u) }2
\big| E''(u) \pl_y u\big|^2 \Big) \dd y \ \text{ and } 
\ol M(u_\PM,\kappa) = \inf_{u\in \bfP(u_\PM)} \bfM(u, \kappa\bm1_{I_0}),
\]
we want to find a representation of $\ol M$ in terms of the primal membrane
potential $\bfR_\mb$. We emphasize that the functional $\bfM$ occurs naturally
as integrand in $\wh\mfD_\eps$, see \eqref{eq:def.mfFeps}. 

\AAA
\begin{remark}
\label{rm:BodineauDerrida}
In \cite{BodDer04CFNE} a large-deviation theory is developed to characterize
certain NESS $u$ as minimizers of the functional 
\[
\ol G(u_\PM,\kappa) :=  \inf\Bigset{ \int_0^1 \frac{\big(\kappa {+} \mob(u)
    E''(u)\pl_y u\big)^2}{\mob(u)} }{u\in \bfP(u_\PM)} ,
\] 
see eqn.\,(14) there. Clearly $\ol G$ satisfies $\ol M(u_\PM,\kappa) = \ol
G(u_\PM,\kappa)- \kappa\big(E'(u_+){-}E'(u_-)\big) $.  
\end{remark}
\EEE

\SSS We note that
$\bfM$ is related to $\bfB$ by Legendre transform, namely \SSS 
\[
\bfB(u,\xi)= \sup_Q \Big(\int_0^1 Q\pl_y \xi\dd y - \bfM(u,Q)\Big) 
\quad \text{and} \quad 
\bfM(u,Q)= \sup_\xi \Big(\int_0^1 Q\pl_y \xi \dd y - \bfB(u,\xi)\Big).
\]
Inserting $Q=\kappa \bm 1_{I_0}$ in the definition of $\ol M$ and using the
min-max characterization of $\ol B$ from Section \ref{su:RedMembFormulas} we obtain 
\begin{equation}
  \label{eq:Relations}
 \ol B(u_\PM, \delta)= \sup_\kappa  \big(\kappa \delta - \ol M(u_\PM,\kappa)\big) 
\quad \text{and} \quad 
\ol M(u_\PM,\kappa)= \sup_\delta \big( \kappa \delta - \ol B(u_\PM,\delta)\Big).
\end{equation}
In particular, we have $  \ol M(u_\PM,0)=-\ol B(u_\PM,0) =
\tdfrac12\big(H(u_+){-}H(u_-)\big)^2$. 

Thus, it is equivalent to derive $\bfR_\mb$ via $\ol B(u_\PM,\delta)= 
\bfR^*_\mb(u_\PM,\delta) +\ol B(u_\PM,0)$ as in
Corollary \ref{ex:Exist.Rmb}  or via $\ol M(u_\PM,\kappa)
=\bfR_\mb(u_\PM,\kappa) - \ol B(u_\PM,0)$, which
is the content of the following Proposition \ref{pr:Rmb.olM}. 
We emphasize that the derivation of $\bfR_\mb$ from $\ol M$ is completely
independent of \EEE the above theory of BER functions. Nevertheless, the proof is
strongly inspired by the NESS theory using BER functions: in particular, \SSS
in both cases \EEE 
the proof strongly relies on the existence of the NESS
$u_\mafo{NESS}(u_\PM;\cdot)$ defined in \eqref{eq:ExplicitNESS} solving the ODE
$0=\pl_y\big( \mob(u)E''(y)\pl_y u\big)$ and the boundary conditions $u(0)=u_-$
and $u_+(1)=u_+$. 

\begin{proposition}[Derivation of $\bfR_\mb $ via $\ol M$]
\label{pr:Rmb.olM}
Let the assumption \eqref{eq:AllAssump} and \eqref{eq:Cvx.Ass.psi} hold and
define
\[
\bfR_\mb(u_\PM,\kappa) := \ol M(u_\PM,\kappa) - \ol M(u_\PM,0).
\]
Then, $\bfR_\mb$ is a dissipation potential and $\ol M$ satisfies the 
relation
\begin{equation}
  \label{eq:bfJ.Rmemb}
  \ol M(u_\PM,\kappa) = \bfR_\mb (u_\PM,\kappa) +
  \bfR^*_\mb \big(u_\PM,E'(u_-){-}E'(u_+)\big).  
\end{equation}
\end{proposition} 
\begin{proof} Since $(w,Q)\mapsto \bfM(\psi(w),Q)$ is jointly convex, we see
that $\kappa\mapsto \ol M(u_\PM,\kappa)$ is convex as well. Since
$ \bfM(u,\cdot)$ is even, also $\ol M(u_\PM,\cdot)$ is even. Hence,
$\bfR_\mb(u_\PM, \cdot)$ defined as above is an even dissipation potential.

To establish the representation of $\ol M$ via $\bfR_\mb$ and $\bfR^*_\mb$, it
remains to show $\ol M(u_\PM,0)=\bfR^*_\mb \big(
u_\PM,E(u_-){-}E(u_+)\big)$. To see this, we do the following elementary
calculations:
\begin{align*}
&\bfR^*_\mb \big( u_\PM, E'(u_-){-} E'(u_+) \big) - \ol M(u_\PM,0) 
\\
& \overset{(1)}= \sup\Bigset{ \kappa \big(E'(u_-){-} E'(u_+)\big) 
  - \bfR_\mb(u_\PM, \kappa) }{ \kappa \in \R} - \ol M(u_\PM,0)
\\
& \overset{(2)} = \sup\Bigset{ \kappa \big(E'(u_-){-} E'(u_+)\big) 
  - \ol M(u_\PM, \kappa)+ \ol M(u_\PM,0) }{ \kappa \in \R}
  - \ol M(u_\PM,0)
\\
& \overset{(3)} = \sup \Bigset{\int_0^1 \!\!\Big( -\kappa \pl_y E'(u(y))
  - \frac{\kappa^2}{2\mob(u)} - \frac{\mob(u)}2 |\pl_y E'(u)|^2 \Big)
  \dd y }{ \kappa \in \R, \ u\in \bfP(u_\PM)}
\\
&= - \inf \Bigset{\int_0^1 \frac12 \Big( 
   \frac{\kappa}{\sqrt{\mob(u)}} + \sqrt{\mob(u)}\, \pl_y E'(u) 
  \Big)^2 \dd y }{ \kappa \in \R, \ u\in \bfP(u_\PM) }.
\end{align*}
Here, $ \overset{(1)}= $ uses the definition of $\bfR^*_\mb$ in terms of
$\bfR_\mb$,  $ \overset{(2)}= $ uses the definition of $\bfR_\mb$ via $\ol M$,
and $ \overset{(3)}= $ uses the definition of $\ol M$ via $\bfM$. 
If we insert $u_\mafo{NESS}(u_\PM,\cdot)$, then $\mob(u)\pl_y(E'(u))$ equals a
constant $\wt\kappa(u_\PM)$ by the NESS equation \eqref{eq:NESS.Eqns}, such
that choosing $\kappa=-\wt\kappa(u_\PM)$ makes the infimum equal to $0$.
Hence, the infimum equals $0$ and \eqref{eq:bfJ.Rmemb} is established.
\end{proof}

\subsection{Explicit relations for power-law structures}
\label{su:PowerLaw}

We consider the \SSS interface problem with \EEE  
\[
\text{power-law mobility } \ \SSS \mob(v) \EEE =v^\beta \quad \text{and} \quad
  \varkappa\text{-energy } \  E=\sfE_\varkappa, \quad \text{i.e.\ }
  E''(v)=v^{\varkappa-2}, 
\] 
where $\beta>0 $ and $ \varkappa> 0$. We define the functionals 
\begin{align*}
&\bfM_{\beta,p}(u,\kappa) := \int_0^1 \Big( 
 \frac{\kappa^2}{2u^\beta} + \frac{u'(y)^2}{2 u^{4-2p-\beta}} \Big) \dd y,&&
\ol M_{\beta,p}(u_\PM,\kappa)= \inf_{u\in \bfD(u_\PM)} \bfM_{\beta,p}(u,\kappa),
\\
&\bfN_\sigma(u,\kappa):=\bfM_{\sigma,2-\sigma}(u,\kappa), &&
\ol N_\sigma(u_\PM,\kappa) := \inf_{u\in \bfD(u_\PM)} \bfN_\sigma(u,\kappa).
\end{align*}
Clearly, $\bfN_\sigma$ is positively $(2{-}\sigma)$-homogeneous, i.e.\ $
\bfN_\sigma\big(\lambda  \,u,\lambda\,\kappa\big) =
 \lambda ^{2-\sigma}  \bfN_\sigma(u,\kappa)$ .

In Lemma \ref{le:Expl.olN.sigma} we show that all $\ol M_{\beta,p}$ can be expressed
in terms $\ol N_\sigma$ for a suitable $\sigma$.  Moreover, we derive growth
bounds for $\kappa \mapsto \ol N_\sigma(u_\PM,\kappa)$. This implies that 
the corresponding  membrane dissipation potential
$\wt\bfR_{\beta,\varkappa}$ (for the case $\mob(u)=u^\beta$ and $E=\sfE_p$)
takes the form  
\[
\wt\bfR_{\beta,\varkappa}\big(v_\PM, \kappa) = \frac1{\alpha^2} \Big(
 \ol N_{\beta/\alpha} \big( v_\PM^\alpha,\alpha \kappa\big) -  \ol
 N_{\beta/\alpha} \big( v_\PM^\alpha,0) \Big)  \quad
\text{with } \alpha=\varkappa+\beta -1. 
\]
We will see that $\sigma=\beta/\alpha=1$ is the critical case, which happens for
the Boltzmann case $\varkappa=1$. Additionally, the case $\sigma\in {[0,1[}$ is
interesting, which corresponds to $\varkappa>1$. For the kinetic relation
$[\xi]_0 = \pl_\kappa \wt\bfR_{\beta,\varkappa }(v_\PM, \kappa)$ between the
flux $\kappa$ and the jump $[\xi]_0$ of the chemical potentials, we then find
the asymptotic power-law behavior
\begin{equation}
  \label{eq:GrwothExpo}
   \kappa \ \sim\   [\xi]_0^{\alpha/(\varkappa{-}1)} \quad  \text{for }
   [\xi]_0  \to \infty .
\end{equation}
To see this, we exploit the expression of $\wt\bfR_{\beta,\varkappa} $ in terms
of $\ol N_\sigma$ with $\sigma=\beta/\alpha \in {]0,1[}$. The following upper
and lower bounds \eqref{eq:olN.uppB1} and \eqref{eq:olN.uppB0} show that the
convex function $\kappa \mapsto \ol N_\sigma(u_\PM,\kappa)$ behaves like
$|\kappa|^{2-\sigma}$. Hence the dual membrane potential
$ \SSS \wt\bfR_\mb^* ( \EEE v_\PM, \delta) $ behaves like
$|\delta|^{1+1/(1{-}\sigma)}$. Thus, the kinetic relation has a power-law
behavior with power $1/(1{-}\sigma)=
\alpha/(\alpha{-}\beta)=\alpha/(\varkappa{-}1)$.  

We now characterize $\bfM_{\beta,p}$ in terms of $\bfN_\sigma$ and give some 
estimates on $\ol N_\sigma$. 

\begin{lemma}
\label{le:Expl.olN.sigma} 
We have the relations 
\begin{align}
    \label{eq:Rel.bfJ.beta.gamma}
&\bfM_{\beta,p}(v,\kappa)
   = \frac1{\alpha^2} \bfN_{\beta/\alpha} \big( v^\alpha, \alpha\kappa
   \big) \quad \text{with } \alpha = p+\beta -1.
\end{align}
Moreover, the function $\ol N_\sigma$ satisfies 
\begin{subequations}
\begin{align}
\label{eq:olN.0.expl}
\sigma=0:\quad & \ol N_0(u_\PM,\kappa) = \frac12(u_+{-}u_-)^2 + \frac12\kappa^2 ,
\\
\label{eq:olN.1.expl}
\sigma=1:\quad &\ol N_1(u_\PM,\kappa)=  \sqrt{u_-u_+}\, \CCC
\big( \kappa/\sqrt{ u_-u_+}\big) + 2\big(\sqrt{u_-}-\sqrt{u_+}\big)^2,
\\
\label{eq:olN.ksppa=0}
\sigma\geq 0:\quad & \ol N_\sigma(u_\PM,0) = 2 \big| \frac{u_-^{1-\sigma/2} -
  u_+^{1-\sigma/2}}{2-\sigma} \big|^2 ,
\\
\label{eq:olN.uppB1}
 \sigma \in {]0,1[}:\quad & \ol N_\sigma(u_\PM,\kappa) \leq  c_\sigma
 |\kappa|^{2-\sigma}  \text{ for }\kappa^2\geq
               \tfrac{2{-}\sigma}\sigma (u_+{-}u_-)^2 \ \text{ with }
  c_\sigma= \tfrac{2^{1-\sigma} (2{-}\sigma)^{1-\sigma}}{(1{-}\sigma)\sigma^{\sigma}} ,
\\
\label{eq:olN.uppB0}
 \sigma \in {]0,1[}:\quad & \ol N_\sigma(u_\PM,\kappa) \geq
  \wh c_\sigma |\kappa|^{2-\sigma} \big( 
 \AAA \min\{1, \tfrac{3\kappa}{4(u_+{+}u_-)}\} \EEE \big)^\sigma 
\ \text{ with } \wh c_\sigma= \tfrac{\AAA(1{-}\sigma)^{\sigma-1}\!} {
  4\,\sigma^\sigma}, 
\\
\label{eq:olN.uppB2}
\sigma\geq 0:\quad & \ol N_\sigma(u_\PM, \kappa) \leq \frac{\kappa^2+(u_+{-}u_-)^2}{2} 
 \frac{u_+^{1-\sigma} - u_-^{1-\sigma}}{(1{-}\sigma)\,(u_+{-}u_-)} ,
\\
\label{eq:olN.uppB3}
\sigma>1:\quad &\ol N_\sigma(u_\pm,\kappa)\geq \frac{|\kappa|}{\sigma{-}1}
\Big| \frac1{u_+^{\sigma-1}} -  \frac1{u_-^{\sigma-1}}\Big|,
\\
\label{eq:olN.uppB4}
\sigma>1:\quad &\ol N_\sigma(u_\pm,\kappa)\leq \frac{2|\kappa|}{\sigma{-}1}
\Big| \frac1{u_+^{\sigma-1}} +  \frac1{u_-^{\sigma-1}}\Big| \ \text{ for }
|\kappa| \geq | u_+{-}u_-|. 
\end{align}
\end{subequations}
\end{lemma}
\begin{proof}
\RRR A direct calculation \EEE yields the relation
$ \bfM_{\beta,p} (w^\sigma, \kappa)= \sigma^2 \bfM_{\sigma\beta, \wt p_\sigma}
(w,\kappa/\sigma)$ with $\wt p_\sigma =1+\sigma(p{-}1)$, and
\eqref{eq:Rel.bfJ.beta.gamma} follows. \AAA The functional $\bfN_\sigma$ has
the explicit form
\[
\bfN_\sigma(u) = \int_0^1 \frac{\kappa^2 + (u')^2}{u^\sigma} \dd y.
\] 
Hence, the \EEE identities \eqref{eq:olN.0.expl} and
\eqref{eq:olN.ksppa=0} are obvious.  Identity \eqref{eq:olN.1.expl} for
$\ol N_1$ is established in Section \ref{su:BoltzmannPower}. Estimate
\eqref{eq:olN.uppB1} follows by inserting
$\wh u(y)=\min\{u_-{+}a\kappa y, u_+{+}a\kappa(1{-}y)\} $ with
$a\kappa \geq |u_+{-}u_-|$. A direct calculation gives
\begin{equation}
  \label{eq:Nsigma.linTest}
  N_\sigma(\wh u,\kappa) = \frac{\kappa(1{+}a^2)}{a(1{-}\sigma)} \Big( 2^\sigma
\big(a\kappa{+}u_+{+}u_-\big)^{1-\sigma} - u_+^{1-\sigma} - u_-^{1-\sigma}\Big)
\leq   \frac{\kappa^{2-\sigma}(1{+}a^2)2^\sigma}{a^\sigma(1{-}\sigma)} .
\end{equation}
For $\sigma \in {]0,1[}$ we choose the optimal $a=(2/(2{-}\sigma))^{1/2}$ giving
the desired result.  

Estimate \eqref{eq:olN.uppB2} is obtained by inserting $u(y)=(1{-}y)u_-+y u_+$
into $\bfN_\sigma$. 

For the lower bound \eqref{eq:olN.uppB0} \SSS we set \EEE 
$u_*=\max\{ u(y)\} \geq \max\{u_+,u_-\}$ \SSS and \EEE estimate 
\begin{align*}
  \bfN_\sigma(u,\kappa)&\geq \int_0^1 \frac{\kappa^2+ |u'|^2}{2 \,u_*^\sigma}
  \dd y \geq \frac{\kappa^2 +(2 u_*{-}u_+{-}u_-)^2}{2\,u_*^\sigma}.
\end{align*}
For the second estimate one minimizes $\int_0^1 |u'|^2\dd y $ over all
functions with $u(0)=u_-$, $u(1)=u_+$, and \AAA $u(y_*)=u_*$ for some
$y_*\in [0,1]$. \EEE The nominator can be estimated from below by
$\kappa^2/2 + a u_*$ with \AAA
$a=4\big( ((u_+{+}u_-)^2{+}\kappa^2/2)^{1/2}-u_+{-}u_-\big)$. \EEE Minimizing
with respect to $u_*$ we obtain the lower bound
$\hat c_\sigma |\kappa|^{2-2\sigma}a^\sigma$. Using
$\sqrt{1{+}y^2}-1 \geq \min\{ 3y^2/8, y/2\}$ \AAA for
$y=\kappa/(\sqrt2(u_+{+}u_-))$ \EEE we can estimate $a$ from below, and the
lower bound \eqref{eq:olN.uppB0} follows.

For $\sigma>1$ we proceed similarly and obtain
\[
\bfN_\sigma (u,\kappa) \geq \int_0^1 |\kappa| \frac{|u'|}{u^\sigma} \dd y \geq 
|\kappa| \Big( \int_{u_-}^{u_*} \frac{\dd u}{u^\sigma} +  \int_{u_+}^{u_*}
\frac{\dd u}{u^\sigma}\Big) = \frac{|\kappa|}{\sigma{-}1} \Big( u_+^{1-\sigma}
+  u_-^{1-\sigma} -2 u_*^{1-\sigma}\big). 
\]
Here, we used $\sigma>1$ and exploiting $ u_* \geq \max\{u_+,u_-\}$ gives the
lower bound \eqref{eq:olN.uppB3}.

Using \eqref{eq:Nsigma.linTest} for $\sigma>1$ we can drop the negative term
$\frac{2^\sigma}{1{-}\sigma}(a\kappa{+}u_+{+}u_-)^{1-\sigma}$ and the desired
estimate \AAA \eqref{eq:olN.uppB4} \EEE follows by choosing
$a=\mafo{sign}(\kappa)$.   
\end{proof}

\subsection{Boltzmann entropy and power-law mobility}
\label{su:BoltzmannPower}

Here we present a very short method to find $\wt \bfR_{\beta,1}$
explicitly. Surprisingly, it is possible to calculate $\ol M_{\beta,1}$
explicitly without determining the associated saddle points. A first, but much
longer derivation is contained in \cite{LMPR17MOGG}, for a slightly more general
approach than the one presented here, see \cite[Thm.\,V.1]{Miel23NESS}.
To find $\ol M_{\beta,1}$ we transform the corresponding BER functional by
inserting $\xi = \log (u z^{-2/\beta})$ and obtain the very simple result
\[
\wt\bfB(u,z) = \bfB_{\beta,1}(u,\log(u z^{-2/\beta})) 
= - \int_0^1 \frac{2}{\beta^2} \,z'\, \big( \tdfrac{u^\beta}z \big)' \dd y . 
\]
Since $(u,\xi) \mapsto \bfB(u,\xi)$ has a unique saddle point, given the
boundary condition $ u \in \bfP(u_\PM)$ and $\xi \in \bfD(\xi_\PM)$, the same
is true for $(u,z) \mapsto \wt \bfB(u,z)$. It can be found by searching for the
unique critical point via 
\[
0= \rmD_u \wt \bfB(u,z) =  \frac{2 u^{\beta-1}}{\beta z}\: z''
\quad \text{and} \quad 
0= \rmD_z \wt\bfB(u,z) = \cdots .
\]
We do not need the second equation, because $ 0= z''$
implies that $y \to z(y)$ is affine, giving $z'(1)=z'(0)=z(1){-}z(0)$. 
Moreover, integration by parts gives 
\[
\wt\bfB(u,z)= -\frac2{\beta^2} \big( z'(1) u_+^\beta/z(1)  - z'(0)
u_-^\beta/z(0)\big) = \frac2{\beta^2} \big( z(1)-z(0)\big) \big( 
 u_-^\beta/z(0) -  u_+^\beta/z(1)   \big) 
\]
Using the transformation $\xi = \log(u z^{-2/\beta})$ we find the boundary
conditions 
\[
z(1)= u_+^{\beta/2} \ee^{-\beta \xi_+/2} \quad \text{and} \quad 
z(0)= u_-^{\beta/2} \ee^{-\beta \xi_-/2} .
\] 
Inserting this we find the desired expression 
\begin{align*}
\ol B_{\beta,1}(u_\PM,\xi_\PM) &= \wt\bfB(u,z)= \frac2{\beta^2}   \big( 
  u_+^{\beta/2} \ee^{-\beta \xi_+/2} - u_-^{\beta/2} \ee^{-\beta \xi_-/2}
 \big) \big( 
 u_-^{\beta/2} \ee^{\beta \xi_-/2}-  u_+^{\beta/2} \ee^{\beta \xi_+/2} 
\big)
\\
& =\frac{2}{\beta^2} \Big( u_-^{\beta/2}u_+^{\beta/2}\big( \ee^{\beta \delta/2} 
+  \ee^{-\beta \delta/2}\big) - u_+^\beta  -u_-^\beta \Big) 
\\
& = \frac1{\beta^2} \Big( u_-^{\beta/2}u_+^{\beta/2} \CCC^*(\beta \delta) - 
2\big( u_+^{\beta/2} - u_-^{\beta/2} \big)^2 \Big)  
\quad \text{with } \delta= \xi_+- \xi_-. 
\end{align*}

Thus,  we have derived the explicit formulas 
\[
\wt\bfR^*_{\beta,1}(u_\PM,\delta)= \frac{ u_-^{\beta/2}u_+^{\beta/2}}{\beta^2}
\CCC^*(\beta \delta) , \quad 
\wt\bfR_{\beta,1}(u_\PM,\kappa) =   \frac{u_-^{\beta/2}u_+^{\beta/2}}{\beta^2}
  \,\CCC\big(\tdfrac{\beta\:\kappa}{ u_-^{\beta/2}u_+^{\beta/2}} \big).
\] 
Moreover, from
$\ol N_1(u_\PM,\kappa)= \wt\bfR_{1,1}(u_\PM,\kappa) + \wt\bfR^*_{1,1} \big( u_\PM,
\log(u_-/u_+) \big)$ we obtain identity \eqref{eq:olN.1.expl}.

\appendix

\section{Appendix}
\label{se:App}

Here we provide some auxiliary results that are not problem specific and are of
general interest. 

\subsection{The entropy relative-entropy  estimate involving $\CCC$}
\label{su:App.calC}

We introduce the perspective function of $\CCC$ defined via
by 
\[
\wh \CCC(a,s) := \left\{ \ba{cl} a\,\CCC(s/a)& \text{for } a>0,\\ 
 0&\text{for } (a,s)=(0,0),\\ \infty&\text{otherwise}.  \ea \right. 
\]
The function $\wh \CCC: \R_\geq \ti \R \to [0,\infty]$ is lower
semi-continuous, jointly convex, and non-increasing in $a$. 
The following result is used in Proposition \ref{pr:kappa.int} to obtain
higher integrability. A similar estimate is provided in
\cite[Prop.\,A.1]{HeMiSt25?DCLN}. 

\begin{proposition}\label{pr:calC.estimate}
Let $q>1$ and assume $a\in \rmL^q(\Sigma)$ with $a\geq 0$ a.e.\ Then,
for every measurable function $s:\Sigma\to \R$ we have the estimate 
\[
 \int_\Sigma \CCC(s(y)) \dd y  \leq \frac{q}{q{-}1} \int_\Sigma
 \wh\CCC(a(y),s(y)) \SSS \dd y + \EEE \frac{2}{q{-}1} \|
 a\|_{\rmL^q(\Sigma)}^q . 
\]
\end{proposition}
\begin{proof} \STEP{1:} 
We set $\theta=1-1/q \in {]0,1[}$ \SSS and \EEE estimate $\wh\CCC$ from below via 
\begin{align}
 \label{eq:whC.bound}
\wh\CCC (a,s)& \geq  \theta \CCC(s) - M(a) \qquad 
\text{with } M(a):= 
  \max\bigset{ \theta \CCC(s) - \wh\CCC(a,s)}{s\in \R}.
\end{align}
For $a\in [0,1]$ we have $\wh\CCC(a,s)\geq \wh\CCC(1,s)=\CCC(s)$,
which implies $M(a)=0$ for $a\in [0,1]$. Taking the Legendre transform
of the inequality defining $M$, we obtain $a\CCC^*(\xi) \leq \theta 
\CCC^*(\xi/\theta) + M(a)$ which gives 
\[
M(a)= \max\bigset{ a \CCC^*(\theta \eta) - \theta \CCC^*(\eta) 
  }{\eta \in \R}.
\]

\STEP{2:} To estimate $M$ we use the special form
$\CCC^*(\xi)=4 \cosh (\xi/2)-4 = 8\big( \sinh(\xi/4)\big)^2$. Starting from the
function $f:{]0,\infty[} \ni \eta \mapsto \log(\ee^\eta{-}\ee^{-\eta}\big)$, we
have $f(\eta) < \eta$ and $f'(\eta)=\coth(\eta)>1$. Hence, the function
$g_\eta: {]0,\infty[} \ni \theta \mapsto \frac1\theta f(\theta \eta)$ satisfies
$g'_\eta(\theta)= \frac1{\theta^2} \big( \theta \eta f'(\theta\eta) - f(\theta
\eta) \big) >0$. From this we conclude $g_\eta(1)\geq g_\eta(\theta) $ for
$\eta>0$ and $\theta \in {]0,1]}$. Exponentiating this we arrive at
\[
\forall \, \eta\in \R\ \forall\, \theta\in {]0,1]}: \quad 
2|\sinh \eta| \geq \big( 2 \sinh(\theta\eta)|\big)^{1/\theta} \quad
\text{or} \quad \CCC^*(\eta)^{q-1} \geq \frac12 \,\CCC^*(\theta \eta )^q. 
\] 

\STEP{3:} With this we estimate $M$ using Young's inequality applied to
$a \CCC^*(\theta \eta)$, namely
\begin{align*}
M(a)
& \leq  \frac1q \big(2^{1/q} a\big)^q + \frac{ q{-}1 }{q}\big(
2^{-1/q}\CCC^*(\theta \eta)\big)^{q/(q-1)} - \theta \CCC^*(\eta) \\
&= 
\frac2q \,a^q + \theta \Big( 2^{1/(1-q)} \CCC^*(\theta \eta)^{q/(q-1)}
 - \CCC^*(\eta) \Big) \leq \frac2q \, a^q.
\end{align*}
Now the desired estimate follows by integrating the pointwise
estimate \eqref{eq:whC.bound}.  
\end{proof}

We point out that the constant $2/(q{-}1)$ in front of $\|a\|_{\rmL^q}^q$ is
optimal for our special choice of $\CCC$. This can be seen easily by
taking $a$ very large.

\subsection{A convex functional on measures}
\label{suA:CvxFunctional}

For a closed  $\Sigma \subset \R^d$ we provide the proper
analytical framework to deal with convex 
functionals of the form
\[
\calF^\circ(w,v) := \int_\Sigma \frac{ \SSS |w(y)|^2 \EEE }{m(y,v(y))} 
 \SSS \dd y \EEE
\]
defined for $(w ,v)\in   \rmL^1(\Sigma;\R^N) \ti\rmL^1(\Sigma)$, 
where $m:\Sigma \ti {[0,\infty[}$ is a Caratheodory function
satisfying 
\begin{equation}
  \label{eq:Cond.m}
  m(y,v)>0 \text{ for }v>0, \quad m(y,v)\leq C(1{+}v), \quad 
m(y,\cdot) \text{ is concave}. 
\end{equation}
The main observation concerning the well-posedness of such functionals
on $\rmL^1(\Sigma)$ stems from the fact that an $\rmL^1$ bound for the
denominator and the finiteness of the functional imply an $\rmL^1$
bound for the nominator. More generally, for $\varrho\in [1,2]$ we have 
\begin{equation}
  \label{eq:A.Holder}
  \| w\|^\varrho_{\rmL^\varrho}= \int_\Sigma |w|^{\varrho}\dd x = \int_\Sigma
\big(\frac{|w|^2}m\big)^{\varrho/2} m^{\varrho/2} \dd x  \leq
\calF^\circ(w,v)^{\varrho/2} 
\big\|m(\cdot,v(\cdot))\big\|_{\rmL^{\varrho/(2{-}\varrho)}}^{1-\varrho/2},
\end{equation}
see also Step 2 in Proposition \ref{pr:SpaTimeComp}. 
Clearly, $\varrho=1$ is critical and yields the desired
$\rmL^1$ bound. 

To define a proper extension $\calF: \mfS\mfM(\Sigma)^N \ti
\mfM(\Sigma) \to [0,\infty]$ of 
$\calF^\circ$ we define the function $h:\Sigma\ti \R^N\ti {[0,\infty[} \to
[0,\infty]$ via 
\[
h(y,w,v)= \left\{ \ba{cl} \frac{|w|^2}{m(y,v)}& \text{for } m(y,v)>0,\\ 
   0     & \text{for } w=0 \text{ and }m(y,v)=0,\\ 
  \infty&\text{for } w\neq 0 \text{ and } m(y,v)=0.
  \ea \right.  
\]
To exploit the lower semi-continuity results in \cite[Thm.\,6.57]{FonLeo07MMCV} 
(in the slightly weaker version of Remark 6.58 there), we first set
$\wt m^\infty(y)=\lim_{v\to +\infty} \frac1v m(y,v)$, which exists by
concavity and lies in ${[0,\infty[}$. Next we set
$m^\infty(y)=\limsup_{x\to y} \wt m^\infty(x)$. With this we obtain the
lower bound $h^\infty$ for the recession function $\wt h^\infty: 
(y,w,v) \mapsto \lim_{\lambda\to \infty}
\frac1\lambda h(y,\lambda w, \lambda v)$, namely 
\[
 h^\infty(y,w,v)= 
  \begin{cases} 
   \frac{|w|^2}{m^\infty(y)v} & \text{for } m^\infty(y)v>0, \\ 
   0 & \text{for }w=0 \text{ and }m^\infty(y)v=0, \\ 
   \infty& \text{for }w\neq 0 \text{ and }m^\infty(y)v=0.
  \end{cases}
\]
The importance is that $h^\infty$ is positively 1-homogeneous, i.e.\
$h^\infty(y,\lambda w,\lambda v)= \lambda h^\infty(y,w,v)$. Thus, it can be
used to describe concentration effects involving measures. 

\begin{theorem}[{Relaxation of $\calF^\circ$ \cite[Thm.\,6.57]{FonLeo07MMCV}}] 
\label{th:Relaxation}
We define $\calF: \mfS\mfM(\Sigma)^N \ti
\mfM(\Sigma) \to [0,\infty]$ as follows 
\begin{equation}
  \label{eq:def.calF}
  \calF(\ol w, \ol v) = \int_\Sigma h(y,w(y), v(y)) \dd y +
  \int_\Sigma h^\infty\big(y, \frac{\rmd \ol w^\rmS}{\rmd \eta},
  \frac{\rmd \ol v^\rmS}{\rmd \eta} \big) \dd \eta ,
\end{equation}
where $\ol w_j  = w_j \LEB^d + \ol  w^\rmS_j$ and $\ol v =  v \LEB^d +
\ol v^\rmS$ are  the Hahn decompositions of (signed)  measures $\ol f$
into a Lebesgue  part $f \LEB^d$ and a singular  part $\ol f{}^\rmS$,
and  $\eta$  is a reference  measure  with   $\ol v^\rmS  +
\sum_{j=1}^d |\ol w_j^\rmS| \,\ll \ol\eta$.
Then, for arbitrary sequences $(u_n)_{n\in \N} $ in
$\rmL^1(\Sigma;\R^N)$ and $(v_n)_{n\in \N}$ in $\rmL^1_{\geq
  0}(\Sigma)$ with
\[
  u_n\LEB^d \weaks \ol u \text{ in } \mfS\mfM(\Sigma)^N \quad
  \text{and} \quad 0\leq v_n\LEB^d \weaks \ol v \text{ in } 
 \mfM(\Sigma),
\]
we have the liminf estimate \ $\liminf_{n\to \infty} \calF^\circ(u_n,v_n) \geq
\calF(\ol u,\ol v)$. 
\end{theorem}

Indeed, it is not difficult to show that for every $(\ol u,\ol v) \in
\mfS\mfM(\Sigma)^N \ti \mfM(\Sigma) $ there \AAA exist \EEE sequences
$(u_n)_{n\in \N}$ and $(v_n)_{n\in \N}$ as above such that
$\calF^\circ(u_n,v_n) \to \calF(\ol u,\ol v)$.

The proof relies on the dual characterization obtained as follows.
Clearly, for a.a.\ $y\in \Sigma$ the function $h(y,\cdot,\cdot)$ is
convex and lower semi-continuous. Hence, it is uniquely determined by
its Legendre-Fenchel transform $h^*:\Sigma\ti \R^N\ti
{]{-}\infty,0]}\to [0,\infty]$:
\[
h^*(y,\mu,\nu)=\sup\bigset{\mu \cdot w+ \nu v - h(y,w,v)}{
  (w,v)\in \R^N \ti \R},
\]
(note that $\nu \leq 0$).  Defining the concave transform
$ \SSS \wt m: \EEE {]{-}\infty,0]} \to [0,\infty] $ via
$\wt m(y,\nu)=\sup\bigset{m(y,v)+\nu v}{v\geq 0} $ \AAA gives \EEE
\[
h^*(y,\mu,\nu) = \begin{cases} \frac{|\mu|^2}4\, 
       \wt m\big(y, \frac{4\nu}{|\mu|^2} \big) 
           &\text{for }\mu\neq 0,\\ \AAA  0 \EEE  
        &\text{for } \mu=0, \end{cases} 
\]
\AAA where $\wt m(y,\wh\nu)=\infty$ for $\wh\nu >- m^\infty(y)$. \EEE 
In the case $m(v)=2v$ we obtain $\wt m(\nu) =\infty$ for
$\nu\in {]{-}2,0]}$ and $\wt m(\nu)=0$ for $\nu \leq -2$, which leads
  to $h^*(y,\mu,\nu) = 0$ for $\nu\leq - |\mu|^2/2$ and $\infty$
  otherwise, cf.\ \cite[(38)\&(39)]{BenBre00CFMS}. 

The dual definition of $\calF:  \mfS\mfM(\Sigma)^N \ti
\mfM(\Sigma) \to [0,\infty]$ is now given by 
\begin{equation}
  \label{eq:calF.ext}
 \begin{aligned}
 \calF(\ol u, \ol v):= \sup\Big\{\;\sum_{k=1}^N\int_\Sigma \mu_k\dd \ol u_k
  + \int_\Sigma \nu \dd \ol v - \int_\Sigma h^*(y,\mu(y),\nu(y)) \dd y
  \; \Big| \qquad \\ 
  \mu\in \rmC^0(\ol\Sigma;\R^N),\ \nu\in \rmC^0(\Sigma),\ \nu\leq 0 \;\Big\}. 
\end{aligned}
\end{equation}
Here $\ol u_k$ are signed Radon measures on $\Sigma$ and $\ol v\in
\mfM(\Sigma)$.

\subsection{Too weak coercivity of \TOS{$\bfM$}{calJ} }
\label{su:Bound.calM}

The proof of Proposition \eqref{pr:liminf.memb} strongly relies on 
the properties of the functional $\ol M_{1,1}=\ol N_1$ which is defined via 
$\ol M_{1,1}(u_\PM,\kappa) = \inf\bigset{ \bfM_{1,1}(\kappa ,u)}{u\in
  \bfP(u_\PM)}$, see Section \ref{su:PowerLaw}.  A crucial point of our theory
is that $\kappa \mapsto \ol M_{1,1}(u_\PM,\kappa) =\ol N_1(u_\PM,\kappa)$ is
superlinear, see \eqref{eq:olN.1.expl} for the explicit formula. In Remark
\ref{rem:kappa} we argue that the theory would be considerably simpler if the
functional 
\begin{equation}
  \label{eq:A.wh.J}
Q \mapsto \calJ(u_\PM,Q):= \bigset{\bfM_{1,1} (u,Q) }{ u \in \bfP(u_\PM)}
\end{equation}
would be superlinear in $\rmL^1(I_0)$. Here we want to show that 
this is \SSS unfortunately not the \EEE case. For this we simply
consider functions having compact support on $I_0={]0,1[}$, e.g.\ 
$Q(x)= 0$ for $x \in {]0,\delta[}\cup {]1{-}\delta, 1[}$. 
For $a>\max\{u^-, u^+\}$ we use the specific competitor in $\bfP(u_\PM)$
\[
u_{a,\delta} (x) := \frac a{\delta^2} \,\min \big\{  (x{+}\eta^-)^2, \,\delta^2,\, 
 (1{+}\eta^+{-}x)^2\big\} \quad \text{ with }\eta^\pm = \delta
 \sqrt{u^\pm/a}\, \in {]0,\delta[}
\]
that satisfies $u_{a,\delta}(0)=u_-$, $u_{a,\delta}(1)=u_+$, and
$u_{a,\delta}(x)= a$ on ${[\delta,1{-}\delta]}$.  Then, we find
\[
\calJ(u_\PM,Q) \leq \bfM_{1,1}(Q,u_{a,\delta}) \leq \frac1{2a}
\int_\delta^{1-\delta}\!\! Q^2 \dd x + 4\frac a\delta. 
\]
Choosing $a$ proportional to $\| Q\|_{\rmL^2}$ we see that
$\calJ(u_\PM,Q)$ has at most linear growth for large functions with compact
support in $I_0$. In particular, it shows that $\calJ(u_\PM,\cdot)$
doesn't have superlinear growth on all of $\rmL^1(I_0)$. 

Indeed, it is easy to show linear coercivity via
\begin{align*}
\|Q\|^2_{\rmL^2} &\leq \int_0^1 \frac{Q^2}{2u}\dd x\: 2\| u\|_{\rmL^\infty}
 \leq \bfM_{1,1}(Q,u) \,\Big(4u^- {+} 4u^++
\int_0^1\frac{(u')^2}{u } \dd x  \Big)\\
& \leq   4 \bfM_{1,1}(Q,u)^2 + 2(u^-{+} u^+)^2. 
\end{align*}
Minimizing over $u\in \bfP(u_\PM)$ implies the coercivity
$\calJ(Q,u_\PM)\geq \frac12 \|Q\|_{\rmL^2} - (u^-{+}u^+)$.

\paragraph*{Acknowledgments.} The authors are grateful for stimulating
discussions with Felix Otto, Mark Peletier, and Andr\'e Schlichting.  \AAA
Moreover, they are thankful to two careful reviewers that helped to improve the
manuscript considerably. \EEE The research has been partially funded by
Deutsche Forschungsgemeinschaft (DFG) through the Collaborative Research Center
SFB 1114 ``\emph{Scaling Cascades in Complex Systems}'' (DFG project no.\
235221301), Subproject C05 ``Effective models for materials and interfaces with
multiple scales''.

\footnotesize
\addcontentsline{toc}{section}{References}


\begin{thebibliography}{11}\itemsep0.1em

\bibitem[AGS05]{AmGiSa05GFMS}
L.~Ambrosio, N.~Gigli, and G.~Savar{\'e}, \emph{Gradient flows in metric spaces
  and in the space of probability measures}, Lectures in Mathematics ETH
  Z\"urich, Birkh\"auser Verlag, Basel, 2005.

\bibitem[AM{\etalchar{*}}12]{AMPSV12PLWG}
S.~Arnrich, A.~Mielke, M.~A.~Peletier, G.~Savar\'e, and M.~Veneroni:
  \emph{Passing to the limit in a {W}asserstein gradient flow: from diffusion
  to reaction}. Calc. Var. Part. Diff. Eqns. \textbf{44} (2012) 419--454.

\bibitem[BeB00]{BenBre00CFMS}
J.-D.~Benamou and Y.~Brenier: \emph{A computational fluid mechanics solution to
  the {M}onge-{K}antorovich mass transfer problem}. Numer. Math. \textbf{84}:3
  (2000) 375--393.

\bibitem[BoD04]{BodDer04CFNE}
T.~Bodineau and B.~Derrida: \emph{Current fluctuations in non-equilibrium
  diffusive systems: an additivity principle}. Phys. Review Letters
  \textbf{92}:18 (2004) 1--4, Revision 30 Oct 2018 arXiv:0402305v1.

\bibitem[Br{\'e}71]{Brez71MMHS}
H.~Br{\'e}zis, \emph{Monotonicity methods in {Hilbert} spaces and some
  applications to nonlinear partial differential equations}, Contribution to
  Nonlinear Functional Analysis, Madison, 1971, {Proc. Sympos. Univ.\
  Wisconsin}, pp.~101--156.

\bibitem[CaM14]{CarMaa14A2WM}
E.~A.~Carlen and J.~Maas: \emph{An analog of the 2-{W}asserstein metric in
  non-commutative probability under which the {Fermionic Fokker-Planck}
  equation is gradient flow for the entropy}. Commun. Math. Phys.
  \textbf{331}:3 (2014) 887--926.

\bibitem[CaM17]{CarMaa17GFEI}
\bysame: \emph{Gradient flow and entropy inequalities for quantum {Markov}
  semigroups with detailed balance}. J. Funct. Analysis \textbf{273}:5 (2017)
  1810--1869.

\bibitem[CDP24]{CiDaPo24EICP}
G.~Ciavolella, N.~David, and A.~Poulain: \emph{Effective interface conditions
  for a porous medium type problem}. Interfaces Free Bound. \textbf{26}:2
  (2024) 161--188.

\bibitem[CG{\etalchar{*}}19]{CGLP19DAEI}
M.~A.~J.~Chaplain, C.~Giverso, T.~Lorenzi, and L.~Preziosi: \emph{Derivation
  and application of effective interface conditions for continuum mechanical
  models of cell invasion through thin membranes}. SIAM J. Appl. Math.
  \textbf{79}:5 (2019) 2011--2031.

\bibitem[CH{\etalchar{*}}12]{CHLZ12FPEF}
S.-N.~Chow, W.~Huang, Y.~Li, and H.~Zhou: \emph{Fokker-{P}lanck equations for a
  free energy functional or {M}arkov process on a graph}. Arch. Rational Mech.
  Anal. \textbf{203}:3 (2012) 969--1008.

\bibitem[DeZ87]{DemZei87LDTA}
A.~Dembo and O.~Zeitouni, \emph{Large deviations techniques and applications},
  2 ed., Springer, New York, 1987.

\bibitem[DFM19]{DoFrMi19GSWE}
P.~Dondl, T.~Frenzel, and A.~Mielke: \emph{A gradient system with a wiggly
  energy and relaxed {EDP}-convergence}. ESAIM Control Optim. Calc. Var.
  \textbf{25} (2019) 68/1--45.

\bibitem[ErM14]{ErbMaa14GFSD}
M.~Erbar and J.~Maas: \emph{Gradient flow structures for discrete porous medium
  equations}. Discr. Cont. Dynam. Systems Ser.~A \textbf{34}:4 (2014)
  1355--1374.

\bibitem[FeG23]{FehGes23NELD}
B.~Fehrman and B.~Gess: \emph{Non-equilibrium large deviations and
  parabolic-hyperbolic {PDE} with irregular drift}. Invent. Math.
  \textbf{234}:2 (2023) 573--636.

\bibitem[Fei72]{Fein72CKCC}
M.~Feinberg: \emph{On chemical kinetics of a certain class}. Arch. Rational
  Mech. Anal. \textbf{46} (1972) 1--41.

\bibitem[FoL07]{FonLeo07MMCV}
I.~Fonseca and G.~Leoni, \emph{Modern methods in the calculus of variations:
  {$L^p$} spaces}, Springer, 2007.

\bibitem[Fre19]{Fren19DEGS}
T.~Frenzel, \emph{On the derivation of effective gradient systems via
  {EDP}-convergence}, Ph.D. thesis, Humboldt-Universit\"at zu Berlin,
  Mathematisch-Naturwissenschaftliche Fakult\"at, 2019, doi 10.18452/21391.

\bibitem[FrL21]{FreLie21EDTS}
T.~Frenzel and M.~Liero: \emph{Effective diffusion in thin structures via
  generalized gradient systems and {EDP}-convergence}. Discr. Cont. Dynam.
  Systems Ser.~S \textbf{14}:1 (2021) 395--425.

\bibitem[GeH25]{GesHey25PMEL}
B.~Gess and D.~Heydecker: \emph{The porous medium equation: large deviations
  and gradient flow with degenerate and unbounded diffusion}. Comm. Pure Appl.
  Math. (2025) .

\bibitem[GlM13]{GliMie13GSSC}
A.~Glitzky and A.~Mielke: \emph{A gradient structure for systems coupling
  reaction-diffusion effects in bulk and interfaces}. Z. angew. Math. Phys.
  (ZAMP) \textbf{64} (2013) 29--52.

\bibitem[GNK17]{GaNeKn17DETC}
M.~Gahn, M.~{Neuss-Radu}, and P.~Knabner: \emph{Derivation of effective
  transmission conditions for domains separated by a membrane for different
  scaling of membrane diffusivity}. Discr. Cont. Dynam. Systems Ser.~S
  \textbf{10}:4 (2017) 773--797.

\bibitem[Grm10]{Grme10MENT}
M.~Grmela: \emph{Multiscale equilibrium and nonequilibrium thermodynamics in
  chemical engineering}. Adv. Chem. Eng. \textbf{39} (2010) 75--128.

\bibitem[HMS25]{HeMiSt25?DCLN}
G.~Heinze, A.~Mielke, and A.~Stephan: \emph{Discrete-to-continuum limit for
  nonlinear reaction-diffusion systems via {EDP} convergence for gradient
  systems}. Submitted (2025) , WIAS Preprint 3194, arXiv:2504.06837.

\bibitem[HPS24]{HePiSc24?GFMG}
G.~Heinze, J.-F.~Pietschmann, and A.~Schlichting: \emph{Gradient flows on
  metric graphs with reservoirs: Microscopic derivation and multiscale limits}.
  Preprint (2024) , arXiv:2412.16775.

\bibitem[JKO98]{JoKiOt98VFFP}
R.~Jordan, D.~Kinderlehrer, and F.~Otto: \emph{The variational formulation of
  the {F}okker-{P}lanck equation}. SIAM J. Math. Analysis \textbf{29}:1 (1998)
  1--17.

\bibitem[L{\'e}o95]{Leon95LDLR}
C.~L{\'e}onard: \emph{Large deviations for long range interacting particle
  systems with jumps}. Annalles Inst. H. Poincar\'e sec.\,B \textbf{31}:2
  (1995) 289--323.

\bibitem[LM{\etalchar{*}}17]{LMPR17MOGG}
M.~Liero, A.~Mielke, M.~A.~Peletier, and D.~R.~M.~Renger: \emph{On microscopic
  origins of generalized gradient structures}. Discr. Cont. Dynam. Systems
  Ser.~S \textbf{10}:1 (2017) 1--35.

\bibitem[Maa11]{Maas11GFEF}
J.~Maas: \emph{Gradient flows of the entropy for finite {M}arkov chains}. J.
  Funct. Anal. \textbf{261} (2011) 2250--2292.

\bibitem[Mar15]{Marc15CECP}
R.~Marcelin: \emph{Contribution a l'\'etude de la cin\'etique
  physico-chimique}. Annales de Physique \textbf{III} (1915) 120--231.

\bibitem[Mie11]{Miel11GSRD}
A.~Mielke: \emph{A gradient structure for reaction-diffusion systems and for
  energy-drift-diffusion systems}. Nonlinearity \textbf{24} (2011) 1329--1346.

\bibitem[Mie13a]{Miel13GCRE}
\bysame: \emph{Geodesic convexity of the relative entropy in reversible
  {Markov} chains}. Calc. Var. Part. Diff. Eqns. \textbf{48}:1 (2013) 1--31.

\bibitem[Mie13b]{Miel13TMER}
\bysame: \emph{Thermomechanical modeling of energy-reaction-diffusion systems,
  including bulk-interface interactions}. Discr. Cont. Dynam. Systems Ser.~S
  \textbf{6}:2 (2013) 479--499.

\bibitem[Mie16]{Miel16EGCG}
\bysame, \emph{On evolutionary {$\Gamma$}-convergence for gradient systems
  {(Ch.\,3)}}, Macroscopic and Large Scale Phenomena: Coarse Graining, Mean
  Field Limits and Ergodicity (A.~Muntean, J.~Rademacher, and A.~Zagaris,
  eds.), Lecture Notes in Applied Math. Mechanics Vol.\,3, Springer, 2016,
  Proc. of Summer School in Twente University, June 2012, pp.~187--249.

\bibitem[Mie23a]{Miel23IAGS}
\bysame, \emph{An introduction to the analysis of gradient systems}, WIAS
  Preprint 3022, arXiv:2306.05026, 2023, (Script of a lecture course 2022/23,
  100\,pp.).

\bibitem[Mie23b]{Miel23NESS}
\bysame: \emph{Non-equilibrium steady states as saddle points and
  {EDP}-convergence for slow-fast gradient systems}. J. Math. Physics
  \textbf{64}:123502 (2023) 1--27.

\bibitem[Mie25]{Miel25?PGSN}
\bysame: \emph{Port gradient systems, non-equilibrium steady states, and
  {Prigogine}'s principle of minimal dissipation}. In preparation (2025) .

\bibitem[MiM17]{MitMie17EGSL}
M.~Mittnenzweig and A.~Mielke: \emph{An entropic gradient structure for
  {Lindblad} equations and couplings of quantum systems to macroscopic models}.
  J. Stat. Physics \textbf{167}:2 (2017) 205--233.

\bibitem[MiN22]{MieNau22EGTW}
A.~Mielke and J.~Naumann: \emph{On the existence of global-in-time weak
  solutions and scaling laws for {Kolmogorov}'s two-equation model of
  turbulence}. Z.\ angew.\ Math.\ Mech. (ZAMM) \textbf{102}:9 (2022)
  e202000019/1--31.

\bibitem[MMP21]{MiMoPe21EFED}
A.~Mielke, A.~Montefusco, and M.~A.~Peletier: \emph{Exploring families of
  energy-dissipation landscapes via tilting --- three types of {EDP}
  convergence}. Contin. Mech. Thermodyn. \textbf{33} (2021) 611--637.

\bibitem[MMZ26]{MaMiZi26?NACR}
D.~Matthes, A.~Mielke, and J.~Zimmer: \emph{A new approach to the chain rule
  for some doubly-nonlinear diffusion equations with {$p$-Laplacian}}. In
  preparation (2026) .

\bibitem[Mou16]{Mous16VCAL}
A.~Moussa: \emph{Some variants of the classical {Aubin--Lions} lemma}. J. Evol.
  Eqns. \textbf{16}:1 (2016) 65--93.

\bibitem[MP{\etalchar{*}}17]{MPPR17NETP}
A.~Mielke, R.~I.~A.~Patterson, M.~A.~Peletier, and D.~R.~M.~Renger:
  \emph{Non-equilibrium thermodynamical principles for chemical reactions with
  mass-action kinetics}. SIAM J. Appl. Math. \textbf{77}:4 (2017) 1562--1585.

\bibitem[MPS21]{MiPeSt21EDPC}
A.~Mielke, M.~A.~Peletier, and A.~Stephan: \emph{{EDP}-convergence for
  nonlinear fast-slow reaction systems with detailed balance}. Nonlinearity
  \textbf{34}:8 (2021) 5762--5798.

\bibitem[MSS25]{MiScSt25?DFOD}
A.~Mielke, A.~Schlichting, and A.~Stephan: \emph{Derivation of the fourth-order
  {DLSS} equation with nonlinear mobility via chemical reactions}. Submitted
  (2025) , arXiv:2510.07149, WIAS Preprint 3221.

\bibitem[NeJ07]{NeuJag07ETCR}
M.~Neuss-Radu and W.~J{\"a}ger: \emph{Effective transmission conditions for
  reaction-diffusion processes in domains separated by an interface}. SIAM J.
  Math. Analysis \textbf{39}:3 (2007) 687--720.

\bibitem[Ott96]{Otto96DDDE}
F.~Otto, \emph{Double degenerate diffusion equations as steepest descent},
  Preprint no.\ 480, SFB 256, University of Bonn, 1996.

\bibitem[Ott01]{Otto01GDEE}
\bysame: \emph{The geometry of dissipative evolution equations: the porous
  medium equation}. Comm. Partial Diff. Eqns. \textbf{26} (2001) 101--174.

\bibitem[PeS23]{PelSch23CGST}
M.~A.~Peletier and A.~Schlichting: \emph{Cosh gradient systems and tilting}.
  Nonlinear Analysis \textbf{231}:113094 (2023) 1--113.

\bibitem[PRV14]{PeReVa14LDSH}
M.~A.~Peletier, F.~Redig, and K.~Vafayi: \emph{Large deviations in stochastic
  heat-conduction processes provide a gradient-flow structure for heat
  conduction}. J. Math. Physics \textbf{55} (2014) 093301/19.

\bibitem[RoS03]{RosSav03TIEC}
R.~Rossi and G.~Savar\'e: \emph{Tightness, integral equicontinuity and
  compactness for evolution problems in {Banach} spaces}. Ann. Scoula Norm.
  Sup. Pisa Cl. Sci (5) \textbf{2}:2 (2003) 395--431.

\bibitem[SaS04]{SanSer04GCGF}
E.~Sandier and S.~Serfaty: \emph{Gamma-convergence of gradient flows with
  applications to {G}inzburg-{L}andau}. Comm. Pure Appl. Math. \textbf{LVII}
  (2004) 1627--1672.

\bibitem[Ste21]{Step21CGED}
A.~Stephan: \emph{{EDP-convergence} for a linear reaction-diffusion system with
  fast reversible reaction}. Calc. Var. Part. Diff. Eqns. \textbf{60}:6 (2021)
  226/35\,pp.

\end{thebibliography}

\newcommand{\etalchar}[1]{$^{#1}$}
\def\cprime{$'$}
\providecommand{\bysame}{\leavevmode\hbox to3em{\hrulefill}\thinspace}
\providecommand{\MR}{}

\end{document}